\documentclass[10pt]{amsart}
\usepackage{amssymb}
\usepackage{bm}
\usepackage{graphicx}
\usepackage[centertags]{amsmath}
\usepackage{amsfonts}
\usepackage{amsthm}
\usepackage{amsbsy}
\usepackage{mathtools}
\usepackage{mathrsfs}
\usepackage{cases}
\usepackage[all]{xy}
\usepackage{relsize}
\usepackage[linktocpage=true]{hyperref}
\usepackage{hyperref}
\hypersetup{
    colorlinks=true,
    linkcolor=blue,
    filecolor=blue,
    citecolor=blue,
    urlcolor=cyan}
\usepackage[margin=4cm]{geometry}
\newtheorem{thm}{Theorem}[section]
\newtheorem{cor}[thm]{Corollary}
\newtheorem{lem}[thm]{Lemma}

\newtheorem{prop}[thm]{Proposition}
\newtheorem{claim}[thm]{Claim}
\newtheorem{fact}[thm]{Fact}

\newtheorem{problem}[thm]{Problem}
\newtheorem{conjecture}[thm]{Conjecture}
\newtheorem{defn}[thm]{Definition}
\newtheorem*{Theorem}{Theorem $($\cite[$\mathrm{Theorem \ 1'}$]{DKT16}$)$}
\theoremstyle{remark}
\newtheorem{rem}[thm]{Remark}
\newtheorem{examp}[thm]{Example}

\newcommand{\rr}{\mathbb{R}}
\newcommand{\nn}{\mathbb{N}}
\newcommand{\ee}{\varepsilon}

\newcommand{\meg}{\geqslant}
\newcommand{\mik}{\leqslant}
\newcommand{\ave}{\mathbb{E}}
\newcommand{\prob}{\mathbb{P}}

\newcommand{\bbx}{\boldsymbol{X}}

\begin{document}

\title[Concentration estimates for high-dimensional random arrays]{Concentration
estimates for functions of finite high-dimensional random arrays}

\author{Pandelis Dodos, Konstantinos Tyros and Petros Valettas}

\address{Department of Mathematics, University of Athens, Panepistimiopolis 157 84, Athens, Greece}
\email{pdodos@math.uoa.gr}

\address{Department of Mathematics, University of Athens, Panepistimiopolis 157 84, Athens, Greece}
\email{ktyros@math.uoa.gr}

\address{Mathematics Department, University of Missouri, Columbia, MO, 65211}
\email{valettasp@missouri.edu}

\thanks{P.V. is supported by Simons Foundation grant 638224.}


\begin{abstract}
Let $\boldsymbol{X}$ be a $d$-dimensional random array on $[n]$ whose entries take values in
a finite set $\mathcal{X}$, that is, $\boldsymbol{X}=\langle X_s:s\in \binom{[n]}{d}\rangle$
is an $\mathcal{X}$-valued stochastic process indexed by the set $\binom{[n]}{d}$ of all
$d$-element subsets of $[n]\coloneqq \{1,\dots,n\}$. We give easily checked conditions
on $\boldsymbol{X}$ that ensure, for instance, that for \emph{every} function
$f\colon \mathcal{X}^{\binom{[n]}{d}}\to\mathbb{R}$ that satisfies
$\ave[f(\boldsymbol{X})]=0$ and $\|f(\boldsymbol{X})\|_{L_p}=1$ for some~$p>1$,
the random variable $f(\boldsymbol{X})$ becomes concentrated after conditioning it on a
large subarray of $\boldsymbol{X}$. These conditions cover several classes of
random arrays with not necessarily independent entries. Applications are given in combinatorics,
and examples are also presented that show~the optimality of various aspects of the results.
\end{abstract}

\maketitle

\tableofcontents


\section{Introduction} \label{sec1}

\numberwithin{equation}{section}

\subsection{Motivation} \label{subsec1.1}
The concentration of measure refers to the powerful phenomenon asserting that a function
that depends smoothly on its variables is essentially constant, as long as the number of the variables is large enough.
There are various ways to quantify this ``smooth dependence" (\textit{e.g.}, Lipschitz conditions, bounds for the $L_2$
norm of the gradient, etc.). Detailed expositions can be found in \cite{Le01, BLM13}.

It is easy to see that this phenomenon is no longer valid if we drop the smoothness assumption.
Nevertheless, one can still obtain some form of concentration under a much milder integrability
condition.
\begin{Theorem}
For every $p>1$ and every $0<\ee\mik 1$, there exists a constant $c>0$ with the following property.
If $n\meg 2/c$ is an integer, $\bbx=(X_1,\dots,X_n)$ is a random vector with independent entries that take values
in a measurable space $\mathcal{X}$, and $f\colon \mathcal{X}^n\to\rr$ is a measurable function with
$\ave[f(\bbx)]=0$ and $\|f(\bbx)\|_{L_p}=1$, then there exists an interval $I$ of\, $[n]$ with $|I|\meg cn$
such that for every nonempty $J\subseteq I$ we have
\begin{equation} \label{e1.1}
\prob\big( \big|\ave[f(\bbx)\,|\, \mathcal{F}_J]\big|\mik\ee\big) \meg 1-\ee,
\end{equation}
where $\ave[f(\bbx)\, |\, \mathcal{F}_J]$ stands for the conditional expectation of $f(\bbx)$ with respect
to the $\sigma$-algebra $\mathcal{F}_J\coloneqq \sigma(\{X_i:i\in J\})$.
\end{Theorem}
(Here, and in what follows, $[n]$ denotes the discrete interval $\{1,\dots,n\}$.) Roughly speaking, this result
asserts that if a function of several variables is sufficiently integrable, then, by integrating out some coordinates,
it becomes essentially constant. It was motivated by---and it has found several applications in---problems
in combinatorics (see \cite{DK16}).

\subsubsection{\!} \label{subsec1.1.1}

The goal of this paper is twofold: to develop workable tools in order to extend the conditional concentration
estimate~\eqref{e1.1} to functions of random vectors $\bbx$ with not necessarily independent entries,
and to present related applications. Of course, to this end some structural property of $\bbx$ is necessary.
We focus on high-dimensional random arrays whose distribution is invariant under certain symmetries.
Besides their intrinsic analytic and probabilistic interest, our choice to study functions of random arrays
is connected to the \textit{density polynomial Hales--Jewett conjecture}, an important combinatorial
conjecture of Bergelson \cite{Ber96}---see Subsection \ref{subsec1.5}.

\subsection{Random arrays} \label{subsec1.2}

At this point it is useful to recall the definition of a random array.
\begin{defn}[Random arrays, and their subarrays/sub-$\sigma$-algebras] \label{d1.1} Let $d$ be a positive integer,
and let $I$ be a set with $|I|\meg d$. A \emph{$d$-dimensional random array on~$I$} is a stochastic process
$\bbx=\langle X_s:s\in \binom{I}{d}\rangle$ indexed by the set $\binom{I}{d}$ of all $d$-element subsets of\, $I$.
If $J$ is a subset of\, $I$ with $|J|\meg d$, then the \emph{subarray of $\bbx$ determined by $J$} is the
$d$-dimensional random array $\bbx_J\coloneqq \langle X_s:s\in \binom{J}{d}\rangle$; moreover, by $\mathcal{F}_J$
we shall denote the $\sigma$-algebra $\sigma(\langle X_s:s \in \binom{J}{d}\rangle)$ generated by $\bbx_J$.
\end{defn}
Of course, one-dimensional random arrays are just random vectors. On the other hand, two-dimensional
random arrays are essentially the same as random symmetric matrices, and their subarrays correspond
to principal submatrices; more generally, higher-dimensional random arrays correspond to random
symmetric tensors. We employ the terminology of random arrays, however, since we are not
using linear-algebraic tools.

\subsubsection{Notions of symmetry} \label{subsubsec1.2.1}

The study of random arrays with a symmetric distribution is a classical topic
that goes back to the work of de~Finetti; see \cite{Au08,Au13,Kal05} for an exposition
of this theory and its applications. Arguably, the most well-known notion of symmetry
is exchangeability: a $d$-dimensional random array $\bbx$ on a (possibly infinite) set
$I$ is called \emph{exchangeable} if for every finite permutation $\pi$ of $I$,
the random arrays $\bbx$ and $\bbx_\pi\coloneqq\langle X_{\pi(s)}:s\in \binom{I}{d}\rangle$
have the same distribution. Another well-known notion of symmetry, which is weaker than
exchangeability, is spreadability:  a $d$-dimensional random array $\bbx$ on a (possibly infinite)
set~$I$ is called \emph{spreadable}\footnote{We point out that this is not standard terminology.
In particular, in \cite{FT85} spreadable random arrays are referred to as \textit{deletion invariant},
while in \cite{Kal05} they are called \textit{contractable}.} if for every pair $J,K$ of finite subsets
of $I$ with $|J|=|K|\meg d$, the subarrays\footnote{If the entries of $\bbx$ take values in a measurable
space $\mathcal{X}$, then, here, we identify $\mathcal{X}^{\binom{J}{d}}$ and $\mathcal{X}^{\binom{K}{d}}$
using the increasing enumerations of $J$ and $K$ respectively.} $\bbx_J$ and $\bbx_K$ have the same distribution.
 Infinite, spreadable, two-dimensional random arrays have been studied by Fremlin and Talagrand
\cite{FT85}, and---in greater generality---by Kallenberg \cite{Kal92}.

Beyond these notions, in this paper we will also consider the following approximate
form of spreadability, which naturally arises in combinatorial applications.
\begin{defn}[Approximate spreadability] \label{d1.2}
Let $\bbx$ be a $d$-dimensional random array on a $($possibly infinite$)$ set $I$, and let $\eta \meg 0$. We say that
$\bbx$ is \emph{$\eta$-spreadable} $($or, simply, \emph{approximately spreadable} if\, $\eta$ is clear from the context$)$,
provided that for every pair $J,K$ of finite subsets of $I$ with $|J|=|K|\meg d$ we have
\begin{equation} \label{e1.2}
\rho_{\mathrm{TV}}(P_J,P_K)\mik \eta,
\end{equation}
where $P_J$ and $P_K$ denote the laws of the random subarrays $\bbx_J$ and $\bbx_K$ respectively, and
$\rho_{\mathrm{TV}}$ stands for the total variation distance.
\end{defn}
We recall that the total variation distance between two probability measures $P$ and $Q$ on a measurable space
$(\Omega,\mathcal{F})$ is the quantity $\rho_{\mathrm{TV}}(P,Q)\coloneqq \sup\big\{|P(A)-Q(A)|:A\in\mathcal{F}\big\}$.
We also note that if $\Omega$ is discrete, then the total variation distance is related to the
$L_1$ norm via the identity $\rho_{\mathrm{TV}}(P,Q)=\frac {1}{2}\|P-Q\|_{L_1}=
\frac{1}{2}\sum_{\omega \in \Omega }|P(\{\omega \})-Q(\{\omega \})|$.

The following proposition justifies Definition \ref{d1.2} and shows that approximately spreadable
random arrays are the building blocks of arbitrary finite-valued, high-dimensional random arrays.
The proof follows by a standard application of Ramsey's theorem \cite{Ra30} taking into account
the fact that the space of all probability measures on a finite set equipped with the
total variation distance is compact (see, also, Fact \ref{f8.3}).
\begin{prop} \label{p1.3}
For every triple $m,n,d$ of positive integers with $n\meg d$, and every $\eta>0$, there exists an integer $N\meg n$ with
the following property. If\, $\mathcal{X}$ is a set with $|\mathcal{X}|=m$ and $\bbx$ is an $\mathcal{X}$-valued,
$d$-dimensional random array on a set $I$ with $|I|\meg N$, then there exists a subset $J$ of $I$ with $|J|=n$ such
that the random array $\bbx_J$ is $\eta$-spreadable.
\end{prop}

\subsection{The concentration estimate} \label{subsubsec1.3}

We are ready to state one of the main extensions of \eqref{e1.1} obtained in this paper;
the question whether \eqref{e1.1} could hold for random vectors with not independent entries,
was asked by an anonymous reviewer of \cite{DKT16} as well as by several colleagues in
personal communication. In this introduction we restrict our discussion to boolean two-dimensional
random arrays, mainly because this case is easier to grasp, but at the same time it is quite
representative of the higher dimensional case. The general version is presented in Theorem~\ref{t5.1}
in Section~\ref{sec5}; further extensions/refinements are given in Section \ref{sec6}.
\begin{thm} \label{t1.4}
Let $1<p\mik 2$, let $0<\ee\mik 1$, let $k\meg 2$ be an integer, and set
\begin{align}
\label{e1.3} C=C(p,\ee,k) & \coloneqq \exp\bigg( \frac{3200}{\ee^8 (p-1)^2}\cdot k^2 \bigg).
\end{align}
Also let $n\meg C$ be an integer, let $\bbx=\langle X_s:s\in \binom{[n]}{2}\rangle$ be a $\{0,1\}\text{-valued}$,
$(1/C)\text{-spreadable}$, two-dimensional random array on $[n]$, and assume that
\begin{equation} \label{e1.4}
\Big| \ave[X_{\{1,3\}} X_{\{1,4\}} X_{\{2,3\}} X_{\{2,4\}} ] -
\ave[X_{\{1,3\}}]\,\ave[X_{\{1,4\}}]\,\ave[X_{\{2,3\}}]\, \ave[X_{\{2,4\}}]\Big| \mik \frac{1}{C}.
\end{equation}
Then for every function $f\colon \{0,1\}^{\binom{[n]}{2}}\to\rr$ with $\ave[f(\bbx)]=0$ and $\|f(\bbx)\|_{L_p}=1$
there exists an interval $I$ of\, $[n]$ with $|I|=k$ such that for every $J\subseteq I$ with $|J|\meg 2$ we have
\begin{equation} \label{e1.5}
\prob\big( \big|\ave[f(\bbx)\,|\, \mathcal{F}_J]\big|\mik\ee\big) \meg 1-\ee.
\end{equation}
\end{thm}
Recall that $\mathcal{F}_J$ denotes the $\sigma$-algebra generated by $\bbx_J$ (see Definition \ref{d1.1}).
Thus, Theorem \ref{t1.4} asserts that the random variable $f(\bbx)$ becomes concentrated after
conditioning it on a \emph{subarray} of $\bbx$. Also observe that \eqref{e1.4} together with
the $(1/C)$-spreadability of $\bbx$ imply that for every $i,j,k,\ell\in [n]$ with $i<j<k<\ell$ we have
\begin{equation} \label{e1.6}
\Big| \ave[X_{\{i,k\}} X_{\{i,\ell\}} X_{\{j,k\}} X_{\{j,\ell\}} ] -
\ave[X_{\{i,k\}}]\,\ave[X_{\{i,\ell\}}]\,\ave[X_{\{j,k\}}]\, \ave[X_{\{j,\ell\}}]\Big| \mik \frac{6}{C}
\end{equation}
(see Figure \ref{figure1}). As we shall shortly see, as the parameter $C$ gets bigger, the estimate \eqref{e1.6}
forces the random variables $X_{\{i,k\}}, X_{\{i,\ell\}}, X_{\{j,k\}}, X_{\{j,\ell\}}$ to behave close to independently.
(It also implies that the correlation matrix of $\bbx$ is close to the identity.) Therefore, we may view \eqref{e1.6}
as an \textit{$($approximate$)$ box independence} condition for $\bbx$. We present various examples of spreadable
random arrays that satisfy the box independence condition~in~Section~\ref{sec7}.

\begin{figure}[htb]
\centering \includegraphics[width=.40\textwidth]{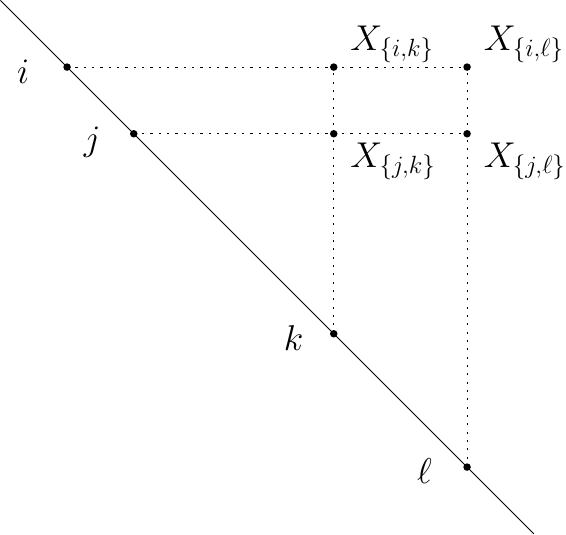}
\caption{The box independence condition.}
\label{figure1}
\end{figure}

Finally we point out that \eqref{e1.6} is essentially an optimal condition
in the sense that for every integer $n\meg 4$ there exist
\begin{enumerate}
\item[---] a boolean, exchangeable, two-dimensional random array $\bbx$ on $[n]$, and
\item[---] a multilinear polynomial $f\colon \rr^{\binom{[n]}{2}}\to \rr$ of degree $4$ with
$\ave[f(\bbx)]=0$ and $\|f(\bbx)\|_{L_\infty}\mik 1$,
\end{enumerate}
such that the correlation matrix of $\bbx$ is the identity, and for which \eqref{e1.6} and \eqref{e1.5} do~not hold
(see Proposition \ref{pa.1}; the case ``$d\meg 3$" is treated in Proposition~\ref{pa.2}).

\subsection{Basic steps of the proof} \label{subsec1.4}

The first step of the proof of Theorem \ref{t1.4}---which can be loosely described as its analytical
part---is to show that the conditional concentration of $f(\bbx)$ is equivalent to an approximate form of
the dissociativity of $\bbx$; this is the content of Theorem \ref{t2.2} in Section \ref{sec2}.
The proof of this step is based on estimates for martingale difference sequences in $L_p$~spaces,
and it applies to random arrays with arbitrary distributions (in particular, not necessarily
approximately spreadable). The main advantage of this reduction is that it enables us to forget
about the function $f$ and focus exclusively on the random~array~$\bbx$.

The second---and more substantial---step is the verification of the approximate dissociativity of $\bbx$.
This is a consequence of the following theorem, which is one of the main results of this paper. (As before,
at this point we restrict our discussion to boolean two-dimensional random arrays; the general version
is given in Theorem \ref{t3.2}.)
\begin{thm}[Propagation of randomness] \label{t1.5}
Let $n\meg 8$ be an integer and $0<\eta,\vartheta \mik 1$.
Also let $\bbx=\langle X_s:s\in \binom{[n]}{2}\rangle$ be a $\{0,1\}\text{-valued}$,
$\eta\text{-spreadable}$, \text{two-dimensional} random array on~$[n]$ such that
for every $i,j,k,\ell\in [n]$ with $i<j<k<\ell$ we have
\begin{equation} \label{e1.7}
\ave[X_{\{i,k\}} X_{\{i,\ell\}} X_{\{j,k\}} X_{\{j,\ell\}}] \mik
\ave[X_{\{i,k\}}]\,\ave[X_{\{i,\ell\}}]\,\ave[X_{\{j,k\}}]\, \ave[X_{\{j,\ell\}}] + \vartheta.
\end{equation}
Then for every nonempty $\mathcal{F}\subseteq \binom{[n]}{2}$ such that $\bigcup\mathcal{F}$
has cardinality at most~$n/2$, we have
\begin{equation} \label{e1.8}
\bigg| \ave\Big[ \prod_{s\in\mathcal{F}}X_s \Big] - \prod_{s\in\mathcal{F}} \ave[X_s] \bigg|
\mik 144\, |\mathcal{F}|\, \big( n^{-1/16} + \eta^{1/16} + \vartheta^{1/16} \big).
\end{equation}
\end{thm}
Theorem~\ref{t1.5} shows that the box independence condition\footnote{Note that in Theorem \ref{t1.5}
we only need the one-sided version \eqref{e1.7} of \eqref{e1.6}. Of course, in retrospect,
Theorem~\ref{t1.5} yields that \eqref{e1.7} is actually equivalent to \eqref{e1.6} albeit with
a slightly different constant.} propagates and forces all, not too large, subarrays of $\bbx$
to behave close to independently. Its proof is based on combinatorial and probabilistic ideas, and it is
analogous\footnote{In fact, this is more than an analogy; indeed, it is easy to see that Theorem \ref{t1.5}
yields the aforementioned property of quasirandom graphs.} to the phenomenon---discovered
in the theory of quasirandom graphs \cite{CGW88,CGW89}---that a graph $G$ that contains (roughly)
the expected number of $4\text{-cycles}$ must also contain the expected number of any other,
not too large, graph~$H$. We comment further on the relation between the box independence condition
and quasirandomness of graphs and hypergraphs in Subsection~\ref{subsec7.1}.

\subsection{Connection with combinatorics} \label{subsec1.5}

We proceed to discuss a representative combinatorial application of our main results.

\subsubsection{Families of graphs} \label{subsubsec1.5.1}

We start by observing that for every integer $n\meg 2$ we may identify a graph $G$ on~$[n]$
with an element of $\{0,1\}^{\binom{[n]}{2}}$ via its indicator function~$\mathbf{1}_G$. (More generally,
for every nonempty finite index set $I$ we identify subsets of $I$ with elements of~$\{0,1\}^I$.)
Thus, we view the set $\{0,1\}^{\binom{[n]}{2}}$ as the \emph{space of all graphs on~$n$ vertices}
and we denote by $\mu$ the uniform probability measure on $\{0,1\}^{\binom{[n]}{2}}$. Our application
is related to the following conjecture of Gowers \cite[Conjecture 4]{Go09}.
\begin{conjecture} \label{con1.6}
Let $0<\delta \mik 1$ and assume that $n$ is sufficiently large in terms of $\delta$.
Then for every family of graphs $\mathcal{A}\subseteq \{0,1\}^{\binom{[n]}{2}}$ with
$\mu(\mathcal{A})\meg \delta$ there exist $G,H\in\mathcal{A}$ with $H\subseteq G$ such that
the difference $G\setminus H$ is a clique, that is, $G\setminus H=\binom{X}{2}$ for some
$X\subseteq [n]$ with $|X|\meg 2$.
\end{conjecture}
Conjecture \ref{con1.6} is a special, but critical, case of the density polynomial
Hales--Jewett conjecture \cite{Ber96}; for a detailed discussion of its significance
we refer to \cite{Go09} where Conjecture \ref{con1.6} was proposed as a polymath project.

Despite the fact that there is considerable interest, there is nearly no information
on Conjecture \ref{con1.6} in the literature (see, however, the online discussion in \cite{Go09}).
This is partly due to the fact that, while the understanding of quasi\-random graphs
is very satisfactory, it is unclear what a quasirandom \emph{family of graphs} actually is.
Our results are pointing precisely in this direction\footnote{Here, it is important to note
that this is a rather basic step of the analysis of Conjecture \ref{con1.6}; indeed,
the combinatorial core of almost every problem in density Ramsey theory is to isolate its
quasirandom and structure components---see, \textit{e.g.}, \cite{Tao08} for an exposition
of this general philosophy.}.

\subsubsection{Quasirandom families of graphs} \label{subsubsubsec1.5.2}

In order to motivate the reader, let us say that a family of graphs $\mathcal{A}\subseteq \{0,1\}^{\binom{[n]}{2}}$
is \emph{isomorphic invariant}\footnote{Isomorphic invariant families of graphs are also referred
to as \emph{graph properties}. It may be argued that Conjecture \ref{con1.6} is more natural for
isomorphic invariant families of graphs, but we do not impose such a restriction in our results.}
if for every permutation $\pi$~of~$[n]$ and every $G\subseteq \binom{[n]}{2}$ we have
\begin{equation} \label{e1.9}
G\in\mathcal{A} \ \ \text{ if and only if } \ \ G_\pi\coloneqq\{\pi(e):e\in G\}\in \mathcal{A};
\end{equation}
that is, $G$ belongs to $\mathcal{A}$ only if every isomorphic copy of $G$ belongs to $\mathcal{A}$.
As we shall see in Proposition \ref{p8.2}, if $\mathcal{A}\subseteq \{0,1\}^{\binom{[n]}{2}}$
is an arbitrary isomorphic invariant family of graphs, then denoting by $\gamma(\mathcal{A})$
the unique nonnegative real such that
\[ \gamma(\mathcal{A})=\mathbf{P}\big( W: W\cup\{i,k\}, W\cup\{i,\ell\}, W\cup\{j,k\}, W\cup\{j,\ell\}\in
\mathcal{A}\big) \]
for every $U=\{i<j<k<\ell\}\in \binom{[n]}{4}$, where $\mathbf{P}$ is
the uniform probability measure on $\{0,1\}^{\binom{[n]}{2}\setminus \binom{U}{2}}$, we have
\[ \gamma(\mathcal{A})\meg \mu(\mathcal{A})^4-o_{n\to\infty}(1). \]
On the other hand, notice that if $\mathcal{A}\subseteq \{0,1\}^{\binom{[n]}{2}}$
is selected uniformly at random, then clearly $\gamma(\mathcal{A})=\mu(\mathcal{A})^4+o_{n\to\infty}(1)$.

Keeping these observations in mind, we view as quasirandom those families of graphs~$\mathcal{A}$
whose parameter $\gamma(\mathcal{A})$ is not significantly larger from the corresponding parameter
of a random family of graphs with the same density. This is, essentially, the content of the following
definition.
\begin{defn}[Quasirandom families of graphs] \label{d1.7}
Let $n\meg 2$ be an integer, let $\theta>0$, and let $\mathcal{A}\subseteq \{0,1\}^{\binom{[n]}{2}}$
be a $($not necessarily isomorphic invariant$)$ family of graphs. We say that $\mathcal{A}$ is
\emph{$\theta$-quasirandom} if, denoting by\, $\mathcal{U}$ the set of all $U=\{i<j<k<\ell\}\in\binom{[n]}{4}$ such that
\begin{equation} \label{e1.10}
\mathbf{P}\big( W: W\cup\{i,k\}, W\cup\{i,\ell\},W\cup\{j,k\}, W\cup\{j,\ell\}\in\mathcal{A}\big)\mik
\mu(\mathcal{A})^4+ \theta,
\end{equation}
\begin{figure}[htb]
\centering \includegraphics[width=.40\textwidth]{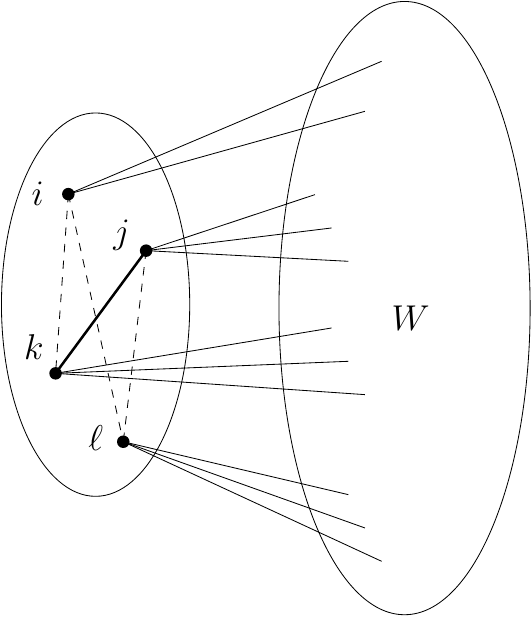}
\caption{Quasirandom families of graphs.}
\label{figure2}
\end{figure}
we have $|\mathcal{U}|\meg (1-\theta) \binom{n}{4}$, where $\mathbf{P}$ is the uniform probability measure
on $\{0,1\}^{\binom{[n]}{2}\setminus \binom{U}{2}}$. Namely, the family $\mathcal{A}$ is $\theta$-quasirandom
if for at least $(1-\theta)$-fraction of increasing quadruples $i<j<k<\ell$ of elements of $[n]$, at most
$\big(\mu(\mathcal{A})^4+ \theta\big)$-fraction of all subgraphs of $\binom{[n]}{2}\setminus \binom{\{i,j,k,\ell\}}{2}$
are such that adding exactly one of the edges $\{i,k\}, \{i,\ell\},\{j,k\},\{j,\ell\}$ yields a graph in $\mathcal{A}$;
see Figure \emph{\ref{figure2}}. $($In particular, if $\mathcal{A}$ is isomorphic invariant, then $\mathcal{A}$
is $\theta$-quasirandom provided that $\gamma(\mathcal{A})\mik \mu(\mathcal{A})^4+\theta$.$)$
\end{defn}
The reader might have already observed the similarity between Definition~\ref{d1.7}
and the classical $4$-cycle condition of quasirandomness of graphs \cite{CGW88,CGW89}.

\subsubsection{\!} \label{subsubsec1.5.3}

The following theorem---which relies on both~conditional concentration and Theorem \ref{t1.5}, and whose
proof is given in Section \ref{sec8}---shows that Definition~\ref{d1.7} is indeed a sensible notion.
\begin{thm} \label{t1.8}
For every $0<\delta\mik 1$ and every integer $k\meg 2$ there exist $\theta>0$ and an integer
$q_0\meg k$ with the following property. Let $n\meg q_0$ be an integer, and let
$\mathcal{A}\subseteq \{0,1\}^{\binom{[n]}{2}}$ be a $\theta$-quasirandom family of graphs
with $\mu(\mathcal{A})\meg \delta$. Then, there exist $K\subseteq [n]$ with $|K|=k$
and $W\subseteq \binom{[n]}{2} \setminus \binom{K}{2}$ such that
\begin{equation} \label{e1.11}
\{W\} \cup \bigg\{ W\cup e: e\in \binom{K}{2}\bigg\}\subseteq\mathcal{A}.
\end{equation}
Thus, there exist $G,H\in\mathcal{A}$ with $H\subseteq G$ such that $G\setminus H$ is a clique.
\end{thm}
Theorem \ref{t1.8} asserts that every non-negligible quasirandom family $\mathcal{A}$
of sufficiently large graphs contains a graph $W$ for which there is a large set $K$ such that
the induced subgraph $W[K]$ of $W$ on $K$ is empty, while at the same time, adding any single edge
from $\binom{K}{2}$ to $W$ does not leave the family $\mathcal{A}$. Note, in particular, that Theorem~\ref{t1.8}
yields an affirmative answer to Conjecture \ref{con1.6} for quasirandom families of graphs in a strong sense:
we can select the graphs $G$ and $H$ so that the difference $G\setminus H$ is a single edge.
Finally, we point out that the proof of Theorem \ref{t1.8} is effective;
see Remark~\ref{r8.6} for its quantitative aspects.

\subsection{Related work} \label{subsec1.6}

Although Theorem \ref{t1.4} (as well as its higher dimensional extension, Theorem \ref{t5.1})
is somewhat distinct from the traditional setting of concentration of smooth functions,
it is related with several results that we are about to discuss.

Arguably, the one-dimensional case---that is, the case of random vectors---is the most heavily investigated.
It is impossible to give here a comprehensive review; we only mention that concentration estimates
for functions of finite exchangeable random vectors have been obtained in \cite{Bob04,Ch06}.

The two-dimensional case is also heavily investigated, in particular, in the literature around various
random matrix models. However, closer to the spirit of this paper is the work of Latala \cite{La06}
and the subsequent papers \cite{AdWo15,GSS19,V19}, which obtain exponential concentration inequalities for
smooth functions (\textit{e.g.}, polynomials) of high-dimensional random arrays whose entries are of the form
\begin{equation} \label{e1.12}
X_s=\prod_{i\in s} \xi_i,
\end{equation}
where $(\xi_1,\dots,\xi_n)$ is a random vector with independent entries and a well-behaved distribution.
Note that all these arrays are dissociated\footnote{See Subsection \ref{subsec2.1} below for the
definition of dissociativity.}, and are additionally exchangeable if the random variables
$\xi_1,\dots,\xi_n$ are identically distributed.

That said, the study of concentration inequalities for functions of more general finite high-dimensional random arrays
is nearly not developed at all, mainly because the structure of finite high-dimensional\footnote{The understanding
is better in the one-dimensional case---see \cite{DF80}.} random arrays is quite complicated (see, also,
\cite[page~16]{Au13} for a discussion on this issue). We make a step in this direction in the companion paper \cite{DTV21}.

\subsection{Organization of the paper} \label{subsec1.7}

We close this section by giving an outline of the contents of this paper.
It is divided into two parts, Part \ref{part1} and Part \ref{part2}, which are largely independent
of each other and can be read separately.

Part \ref{part1} consists of Sections \ref{sec2} up to \ref{sec6}. The main result in Section \ref{sec2}
is Theorem \ref{t2.2}, which reduces conditional concentration to approximate dissociativity.
The next two sections, Sections \ref{sec3} and \ref{sec4}, are devoted to the proof of
Theorem \ref{t1.5} and its higher-dimensional extension, Theorem \ref{t3.2}.
In Section \ref{sec3} we introduce related definitions and we also present some consequences.
The proof of Theorem \ref{t3.2} is given in Section \ref{sec4}; this is the most technically demanding
part of the paper. In Section \ref{sec5} we complete the proofs of Theorem \ref{t1.4} and its higher-dimensional
extension, Theorem \ref{t5.1}. Lastly, in Section~\ref{sec6} we present extensions/refinements
of Theorems \ref{t1.4} and \ref{t5.1} for dissociated random arrays (Theorem~\ref{t6.1}),
for vector-valued functions of random arrays (Theorem~\ref{t6.3}) and a simultaneous
conditional concentration result (Theorem \ref{t6.4}).

Part \ref{part2} consists of Sections \ref{sec7} and \ref{sec8} and it is entirely devoted
to the connection of our results with combinatorics. In Section~\ref{sec7} we give
examples of combinatorial structures for which our conditional concentration results
are applicable, and in Section \ref{sec8} we give the proof of Theorem~\ref{t1.8}.

Finally, in Appendix \ref{Appendix-A} we present examples that show the optimality of the box independence condition.

\subsection*{Acknowledgments}

The authors would like to thank the anonymous referee for numerous comments, remarks and suggestions
that helped us improve the exposition.

The research was supported by the Hellenic Foundation for Research and Innovation
(H.F.R.I.) under the “2nd Call for H.F.R.I. Research Projects to support
Faculty Members \& Researchers” (Project Number: HFRI-FM20-02717).


\part{Proofs of the main results} \label{part1}


\section{From dissociativity to concentration} \label{sec2}

\numberwithin{equation}{section}

\subsection{Main result} \label{subsec2.1}

Let $d$ be a positive integer, and recall that a $d$-dimensional random array $\bbx$ on a (possibly infinite)
subset $I$ of $\nn$ is called \textit{dissociated}\footnote{Notice that this form of dissociativity
(as well as the corresponding approximate version in Definition~\ref{d2.1}) is weaker than the standard
one in the absence of exchangeability, since we do not require independence of $\mathcal{F}_J$ and $\mathcal{F}_K$
for all pairs of disjoint sets $J$ and $K$.} if for every $J,K\subseteq I$ with $|J|,|K|\meg d$
and $\max(J)<\min(K)$, the $\sigma$-algebras $\mathcal{F}_J$ and $\mathcal{F}_K$ are independent,
that is, for every $A\in \mathcal{F}_J$ and $B\in\mathcal{F}_K$ we have $\prob(A\cap B)=\prob(A)\,\prob(B)$.
Dissociativity is a classical concept in probability (see \cite{MS75});
we will need the following approximate version of this notion.
\begin{defn}[Approximate dissociativity] \label{d2.1}
Let $n,\ell,d$ be positive integers such that $n\meg \ell\meg 2d$, and let $0\mik \beta\mik 1$. We say that
a $d$-dimensional random array $\bbx$ on~$[n]$ is \emph{$(\beta,\ell)$-dissociated} provided that for every
$J,K\subseteq [n]$ with $|J|,|K|\meg d$, $|J|+|K|\mik \ell$ and $\max(J)<\min(K)$, and every pair of events
$A\in \mathcal{F}_J$ and $B\in \mathcal{F}_K$ we have
\begin{equation} \label{e2.1}
\big| \prob(A\cap B) -\prob(A)\,\prob(B) \big|\mik \beta.
\end{equation}
\end{defn}
The following theorem---which is the main result in this section---provides the link between
conditional concentration and approximate dissociativity.
\begin{thm} \label{t2.2}
Let $d$ be a positive integer, let $1<p\mik 2$, let $0<\ee\mik 1$, let $k\meg d$ be an integer, and set
\begin{align}
\label{e2.2} \beta=\beta(p,\ee) & \coloneqq \Big(\frac{\ee}{10}\Big)^{\frac{10}{p-1}}, \\
\label{e2.3} \ell=\ell(p,\ee,k) & \coloneqq \bigg\lceil \frac{4}{\ee^4(p-1)}\, k \bigg\rceil.
\end{align}
Also let $n\meg \ell$ be an integer, and let $\bbx$ be a $(\beta,\ell)$-dissociated,
$d$-dimensional random array on $[n]$ whose entries take values in a measurable space $\mathcal{X}$.
Then for every measurable function $f\colon \mathcal{X}^{\binom{[n]}{d}}\to\rr$ with $\ave[f(\bbx)]=0$
and $\|f(\bbx)\|_{L_p}=1$ there exists an interval $I$~of~$[n]$ with $|I|=k$ such that
for every $J\subseteq I$ with $|J|\meg d$ we have
\begin{equation} \label{e2.4}
\prob\big( \big|\ave[f(\bbx)\,|\, \mathcal{F}_J]\big|\mik\ee\big) \meg 1-\ee.
\end{equation}
\end{thm}
We note that for spreadable random arrays there is a converse of Theorem \ref{t2.2}, namely, approximate dissociativity
is in fact necessary in order to have conditional concentration; see Proposition \ref{p2.8}
in Subsection \ref{subsec2.6}.

\subsection{Moment bound} \label{subsec2.2}

The following moment estimate is the main step of the proof of Theorem~\ref{t2.2}.
\begin{thm} \label{t2.3}
Let $d,\ell, n$ be positive integers with $n\meg \ell\meg 2d$, let $0\mik\beta \mik 1$, and
let $\bbx$ be a $d$-dimensional random array on $[n]$ that is $(\beta,\ell)$-dissociated and whose entries
take values in a measurable space $\mathcal{X}$. Then, for every $1<p\mik 2$, every measurable function
$f\colon \mathcal{X}^{\binom{[n]}{d}} \to \mathbb R$ with $f({\bbx})\in L_p$, every integer $k$ with
$d\mik k\mik \lfloor \ell/2 \rfloor$, and every $I\in \binom{[n]}{\ell}$, there exists $J\in \binom{I}{k}$
with the following property. For any $1\mik r < p$, we have
\begin{equation} \label{e2.5}
\big\| \ave[ f({\bbx})\,|\, \mathcal{F}_J]-\ave[f({\bbx})]\big\|_{L_r} \mik
\bigg(\!(p-1)^{-1/2}\, \sqrt{\frac{2k}{\ell}} + 10 \beta^{\frac{1}{r} - \frac{1}{p}}\!\bigg) \,
\big\| f({\bbx}) - \ave[f({\bbx})] \big\|_{L_p},
\end{equation}
where $\mathcal{F}_J$ denotes the $\sigma$-algebra generated by the subarray $\bbx_J$
$($see Definition \emph{\ref{d1.1}}$)$. Moreover, if\, $I$ is an interval of $[n]$,
then $J$ may be chosen to be an interval.
\end{thm}
Theorem \ref{t2.3} easily yields Theorem \ref{t2.2}. We present the details below.
\begin{proof}[Proof of Theorem \emph{\ref{t2.2}} assuming Theorem \emph{\ref{t2.3}}.]
Set $r\coloneqq (p+1)/2$ and notice that with this choice we have $1< r<p \mik 2$. Since $\ave[f(\bbx)]=0$
and $\|f(\bbx)\|_{L_p} =1$, by Theorem \ref{t2.3} applied for the interval $I_1\coloneqq [\ell]$,
there exists an interval $I_2$ of $[\ell]$ with $|I_2|=k$ such that
\begin{equation} \label{e2.6}
\big\|\ave[f(\bbx)\,|\, \mathcal{F}_{I_2}]\big\|_{L_r} \mik
(p-1)^{-1/2}\, \sqrt{\frac{2k}{\ell } } + 10 \beta^{\frac{1}{r} - \frac{1}{p}}.
\end{equation}
We claim that the interval $I_2$ is as desired. Indeed, fix a subset $J$ of $I_2$ with $|J|\meg d$,
and observe that $\mathcal{F}_J\subseteq \mathcal{F}_{I_2}$. Therefore, by \eqref{e2.6}
and the fact that the conditional expectation is a linear contraction on $L_r$, we obtain that
\[ \big\|\ave[f(\bbx)\,|\, \mathcal{F}_J]\big\|_{L_r} \mik
(p-1)^{-1/2}\, \sqrt{\frac{2k}{\ell } } + 10 \beta^{\frac{1}{r} - \frac{1}{p}}. \]
By Markov's inequality, this estimate yields that
\begin{equation} \label{e2.7}
\prob\big( \big|\ave[f(\bbx)\,|\, \mathcal{F}_J]\big|\meg\ee\big)  \mik
(1/\ee)^r \cdot \bigg(\!(p-1)^{-1/2}\, \sqrt{\frac{2k}{\ell } } + 10 \beta^{\frac{1}{r} - \frac{1}{p}}\!\bigg)^r.
\end{equation}
By \eqref{e2.7}, the choice of $r$ and the choice of $\beta$ and $\ell$ in \eqref{e2.2} and \eqref{e2.3} respectively,
we conclude that
\begin{equation} \label{e2.8}
\prob\big( \big|\ave[f(\bbx)\,|\, \mathcal{F}_J]\big|\meg\ee\big)\mik\ee,
\end{equation}
which clearly implies \eqref{e2.4}. The proof of Theorem \ref{t2.3} is completed.
\end{proof}
The rest of this section is devoted to the proof of Theorem \ref{t2.3}, which is based on inequalities for
martingales in $L_p$ spaces. Martingales are, of course, standard tools in the proofs of concentration estimates.
Typically, one decomposes a given random variable $X$ into martingale increments, and then controls an appropriate
norm of~$X$ by controlling the norm of the increments. In the proof of Theorem~\ref{t2.3} we also decompose a given
random variable into martingale increments but, in contrast, we seek to find one of the increments that has controlled norm.
This method, known as the \textit{energy increment strategy}, was introduced in the present probabilistic setting by
Tao \cite{Tao06} for ``$p=2$"\!, and then extended in the full range of admissible $p$'s in \cite{DKT16}.
Having said that, we also note that the main novelty of the present paper lies in the selection of the filtration.

We now briefly describe the contents of the rest of this section. In Subsection \ref{subsec2.3}
we present the analytical estimate that is used\footnote{Square-function estimates could also be used,
but they do not yield optimal dependence with respect to the integrability parameter $p$.} in the proof
of Theorem \ref{t2.3}. In Subsection~\ref{subsec2.4} we prove an orthogonality result for pairs of
$\sigma$-algebras that satisfy the estimate~\eqref{e2.1}. The proof of Theorem \ref{t2.3} is completed in
Subsection \ref{subsec2.5}. Finally, in Subsection \ref{subsec2.6} we show that, for spreadable random arrays,
the assumption of approximate dissociativity in Theorem~\ref{t2.2} is necessary.

\subsection{Martingale difference sequences} \label{subsec2.3}

It is an elementary, though important, fact that martingale difference sequences
are orthogonal in $L_2$. We will need the following extension of this fact.
\begin{prop} \label{p2.4}
Let $1<p \mik 2$. Then for every martingale difference sequence $(d_i)_{i=1}^m$ in $L_p$ we have
\begin{equation} \label{e2.9}
\bigg( \sum_{i=1}^m \|d_i\|^2_{L_p} \bigg)^{1/2} \mik \big(p-1\big)^{-1/2} \, \bigg\| \sum_{i=1}^m d_i\bigg\|_{L_p}.
\end{equation}
In particular,
\begin{equation} \label{e2.10}
\min_{1\mik i\mik m} \|d_i\|_{L_p} \mik \frac{1}{\sqrt{m(p-1)} }\, \bigg\| \sum_{i=1}^m d_i \bigg\|_{L_p}.
\end{equation}
\end{prop}
We note that the constant $(p-1)^{-1/2}$ in \eqref{e2.9} is optimal; this sharp estimate was proved
by Ricard and Xu \cite{RX16} who deduced it from a uniform convexity inequality
for $L_p$ spaces---see \cite[Lemma 4.32]{Pi11}, and also \cite[Appendix~A]{DKK16} for an exposition.

\subsection{Mixing and orthogonality} \label{subsec2.4}

In what follows, it is convenient to introduce the following terminology.
Let $(\Omega, \Sigma , \prob)$ be a probability space, and let $0\mik\beta\mik 1$; given two sub-$\sigma$-algebras
 $\mathcal{A}, \mathcal{B}$ of $\Sigma$, we say that $\mathcal A$ and ${\mathcal B}$ are \textit{$\beta$-mixing}
provided that for every $A\in\mathcal{A}$ and every $B\in \mathcal{B}$ we have
\begin{equation} \label{e2.11}
\big| \prob(A\cap B) - \prob(A)\,\prob(B) \big| \mik \beta.
\end{equation}
Notice that in the extreme case ``$\beta=0$" the estimate \eqref{e2.11} is equivalent to saying that the $\sigma$-algebras
$\mathcal{A}$ and $\mathcal{B}$ are independent, which in turn implies for every random variable $X$ with $\ave[X]=0$
we have $\ave\big[ \ave[X\, |\, \mathcal{A}]\, |\, \mathcal{B}]=0$. The main result in this
subsection (Proposition \ref{p2.7} below) is an approximate version of this fact.

We start with the following lemma.
\begin{lem} \label{l2.5}
Let $(\Omega, \Sigma, \prob)$ be a probability space, let $0\mik\beta\mik 1$, and let $\mathcal{A}, \mathcal{B}$ be
two sub-$\sigma$-algebras of\, $\Sigma$ that are $\beta$-mixing. Then for every real-valued, bounded, random variable $X$
and every $1\mik p\mik \infty$ we have
\begin{equation} \label{e2.12}
\big\| \ave\big[ \ave[X\, |\, \mathcal{A}]\, |\, \mathcal{B}\big] - \ave[X]\big\|_{L_p} \mik
(4 \beta)^{1/p}\, \|X -\ave[X]\|_{L_\infty}.
\end{equation}
\end{lem}
For the proof of Lemma \ref{l2.5} we need the following simple fact.
\begin{fact} \label{f2.6}
Let $(X,\Sigma,\mu)$ be a measure space, and let $f\colon X\to\rr$ be an integrable function. Then we have
\begin{equation} \label{e2.13}
\|f\|_{L_1(\mu)}  \mik 2 \sup_{A\in \Sigma} \bigg| \int_A f \, d\mu \bigg|.
\end{equation}
In particular, if $x_1, \dots, x_m\in \rr$, then
\begin{equation} \label{e2.14}
\sum_{i=1}^m |x_i| \mik 2\max_{\emptyset\neq I \subseteq [m]} \bigg| \sum_{i\in I} x_i \bigg|.
\end{equation}
\end{fact}
\begin{proof}
Since $[f\meg 0], [f<0]\in\Sigma$, we have
\[ \|f\|_{L_1(\mu)} = \bigg| \int_{[f\meg 0]} f\, d\mu \bigg| + \bigg| \int_{[f<0]} f \, d\mu \bigg|
\mik 2\sup_{A\in \Sigma} \bigg| \int_A f \, d\mu \bigg|.\qedhere \]
\end{proof}
We proceed to the proof of Lemma \ref{l2.5}.
\begin{proof}[Proof of Lemma \emph{\ref{l2.5}}]
We prove the $L_1$-estimate; the $L_p$-estimate for $p>1$ follows from the $L_1-L_\infty$ bound,
and the fact that the conditional expectation is a linear contraction on $L_\infty$. Without loss
of generality we may assume that $\ave[X]=0$. (If not, then we work with the random variable
$X'\coloneqq X-\ave[X]$ instead of $X$). Set $Z\coloneqq \ave[X\, |\, \mathcal{A}]$, and observe
that $\ave[Z]=\ave[X]=0$. Hence, by Fact \ref{f2.6}, it suffices to obtain an upper bound for
$\big|\ave[ Z \mathbf{1}_B]\big|$ for arbitrary $B\in \mathcal B$. To this end, note that
$\|Z\|_{L_\infty} \mik \|X\|_{L_\infty}$; therefore, by a standard approximation, we may assume
that $Z$ is of the form $\sum_{i=1}^N a_i \mathbf 1_{A_i}$, where $N$ is a positive integer,
$|a_i|\mik \|Z\|_{L_\infty}$ for every $i\in [N]$, and the family $\{A_1,\dots,A_N\}$ forms a partition
of $\Omega$ into measurable events. Let $B\in \mathcal{B}$ be arbitrary. Using the fact that
$\sum_{i=1}^N a_i\, \mathbb P(A_i) = \ave[Z]=0$ and the triangle inequality, we have
\begin{equation} \label{e2.15}
\big| \ave[ Z \mathbf{1}_B] \big| =\bigg| \sum_{i=1}^N a_i\, \prob(A_i\cap B)\bigg| \mik
\sum_{i=1}^N |a_i| \cdot |\prob(A_i\cap B) - \mathbb P(A_i)\,\mathbb P(B)|.
\end{equation}
If we set $x_i\coloneqq\prob(A_i\cap B) - \prob(A_i)\, \prob(B)$, we obtain that
\begin{equation} \label{e2.16}
\big|\ave[ Z \mathbf{1}_B] \big| \mik \sum_{i=1}^N |a_i| \cdot |x_i| \mik
2\|Z\|_{L_\infty} \, \max_{\emptyset\neq I\subseteq [N]} \bigg| \sum_{i\in I} x_i \bigg|,
\end{equation}
where we have also used the pointwise bound $|a_i| \mik \|Z\|_{L_\infty}$ and Fact \ref{f2.6}.
Finally, setting $A_I\coloneqq\bigcup_{i\in I} A_i$ for every nonempty $I\subseteq [N]$, then we have
\begin{equation} \label{e2.17}
\bigg| \sum_{i\in I} x_i \bigg| = \big| \mathbb{P}(A_I\cap B) - \prob(A_I)\,\prob(B) \big| \mik \beta
\end{equation}
since the sets $A_1,\dots,A_N$ are pairwise disjoint and $A_I\in \mathcal{A}$. We conclude that
\begin{equation} \label{e2.18}
\big| \ave\big[ \ave[Z\,|\, \mathcal{B}]\mathbf{1}_B\big] \big| =
\big| \ave[Z\mathbf{1}_B] \big| \mik 2\beta \|X\|_{L_\infty}.
\end{equation}
Since $B\in \mathcal{B}$ was arbitrary, the result follows.
\end{proof}
We are now ready to state the main result in this subsection.
\begin{prop} \label{p2.7}
Let $(\Omega, \Sigma, \prob)$ be a probability space, let $0\mik \beta\mik 1$, and let $\mathcal{A}, \mathcal{B}$ be
two sub-$\sigma$-algebras of\, $\Sigma$ that are $\beta$-mixing. Let $1\mik r < p \mik \infty$, and let $X\in L_p$.
Then,
\begin{equation} \label{e2.19}
\big\| \ave\big[ \ave[X\, |\, \mathcal{A}]\, |\, \mathcal{B}\big] -\ave[X]\big\|_{L_r} \mik
10 \beta^{\frac{1}{r}-\frac{1}{p}}\, \|X-\ave[X]\|_{L_p}.
\end{equation}
\end{prop}
\begin{proof}
Notice that \eqref{e2.19} is straightforward if $\beta=0$; thus, we may assume that $\beta>0$. In this case,
we will obtain the estimate by truncating $X$ and employing Lemma \ref{l2.5}. We lay out the details. As
in the proof of Lemma \ref{l2.5}, we may assume that $\ave[X]=0$. Let $t>0$ (to be chosen later) be the truncation level,
and set $X_t\coloneqq X \mathbf{1}_{[|X| \mik t]}$. Markov's inequality yields that
$\prob(|X|>t)\mik t^{-p} \| X\|_{L_p}^p$, thus applying H\"older's inequality we obtain that
\begin{equation} \label{e2.20}
\|X_t-X\|_{L_r}^r  = \ave\big[ |X|^r\, \mathbf{1}_{[|X|>t]}\big] \mik \|X\|_{L_p}^{r} \, \prob(|X|>t)^{1-\frac{r}{p}}
\mik \frac{\|X\|_{L_p}^p}{t^{p-r}}
\end{equation}
for any $1\mik r<p$. Therefore,
\begin{align} \label{e2.21}
\big\|  \ave\big[ \ave[X\, |\, \mathcal{A}]\, |\, \mathcal{B}\big] \big\|_{L_r} & \mik
  \big\|  \ave\big[ \ave[X-X_t\, |\, \mathcal{A}]\, | \, \mathcal{B}\big]\big\|_{L_r} + \\
& \;\;\;\;\; + \big\|  \ave\big[ \ave[X_t\, |\, \mathcal{A}]\, |\, \mathcal{B} \big] -\ave[X_t] \big\|_{L_r}
+ \big|\ave[X_t]\big| \nonumber \\
& \mik \|X-X_t\|_{L_r} + (4\beta)^{1/r}\cdot 2t + \|X-X_t\|_{L_1}, \nonumber
\end{align}
where we have used the contraction property of the conditional expectation, Lemma \ref{l2.5} for the random variable $X_t$,
and the fact $\ave[X]=0$, respectively. Taking into account~\eqref{e2.20}, we conclude that
\begin{equation} \label{e2.22}
\big\| \ave\big[ \ave[X\, |\, \mathcal{A}] \mid \mathcal{B} \big] \big\|_{L_r} \mik
2 \frac{\|X\|_{L_p}^{p/r}}{t^{\frac{p}{r}-1}} + 8 \beta^{1/r} t.
\end{equation}
It remains to optimize the latter with respect to $t$; the choice $t\coloneqq \beta^{-1/p} \|X\|_{L_p}$ yields the assertion.
\end{proof}

\subsection{Proof of Theorem \ref{t2.3}} \label{subsec2.5}

We start by observing that the case ``$\beta=0$" follows from the case ``$\beta>0$" by taking the limit in \eqref{e2.5}
as $\beta$ goes to zero. Thus, in what follows, we may assume that $\beta>0$.

After normalizing, we may also assume that
\begin{equation} \label{e2.23}
\big\|f({\bbx}) - \ave[f(\bbx)]\big\|_{L_p}=1.
\end{equation}
Fix an integer $k$ with $d\mik k < \lfloor \ell/2 \rfloor$ and $I\in \binom{[n]}{\ell}$,
and let $\{\iota_1<\cdots<\iota_\ell\}$ denote the increasing enumeration of $I$.
Set $m\coloneqq \lfloor \ell/k \rfloor$. Also let $K_1,\dots,K_m\in \binom{[\ell]}{k}$
be successive intervals with $\min(K_1)=1$, and set $J_i\coloneqq \{\iota_\kappa: \kappa\in K_i\}$
for every $i\in [m]$. Thus, the sets $J_1,\dots,J_m$ are successive subsets of $I$ each of cardinality $k$;
also notice that if $I$ is an interval of $[n]$, then the sets $J_1,\dots,J_m$ are intervals too.

Next, denote by $(\Omega,\Sigma,\prob)$ the underlying probability space on which the random
array $\bbx$ is defined, and for every $i\in [m]$ let $\mathcal{F}_{J_i}$ be the
$\sigma$-algebra generated by the subarray~$\bbx_{J_i}$ (see Definition \ref{d1.1}).
We define a filtration $(\mathcal{A}_i)_{i=0}^m$ by setting $\mathcal{A}_0=\{\emptyset,\Omega\}$ and
\begin{equation} \label{e2.24}
\mathcal{A}_i\coloneqq \bigvee_{l=1}^i \mathcal{F}_{J_l} \ \ \ \text{ for every $i\in [m]$;}
\end{equation}
see Figure \ref{figure3}. We will use variants of this filtration in Section \ref{sec8}.

\begin{figure}[htb]
\centering \includegraphics[width=.40\textwidth]{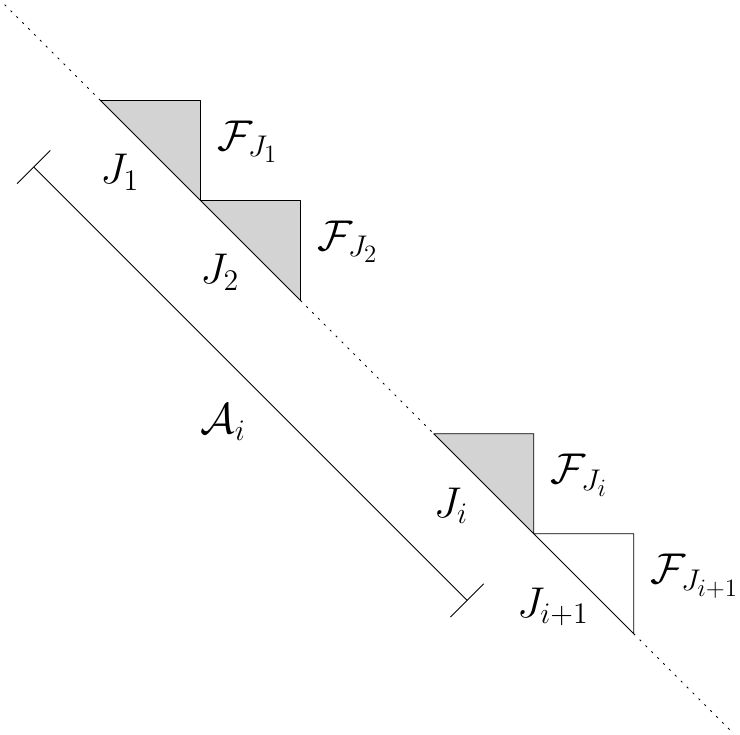}
\caption{The filtration $(\mathcal{A}_i)_{i=0}^m$.}
\label{figure3}
\end{figure}

Let $(d_i)_{i=1}^m$ denote the martingale difference sequence of the Doob martingale for $f(\bbx)$
with respect to the filtration $(\mathcal{A}_i)_{i=0}^m$, that is,
$d_i\coloneqq\ave[f(\bbx)\, |\, \mathcal{A}_i]-\ave[f(\bbx)\, |\, \mathcal{A}_{i-1}]$ for every $i\in [m]$.
Since $\ave[f(\bbx)\,|\, \mathcal{A}_m]-\ave[f(\bbx)]=\sum_{i=1}^m d_i$,
the contractive property of the conditional expectation yields that
\begin{align} \label{e2.25}
\bigg\| \sum_{i=1}^m d_i \bigg\|_{L_p} \mik \big\| f(\bbx) - \ave[f(\bbx)]\big\|_{L_p}
\stackrel{\eqref{e2.23}}{=} 1.
\end{align}
Therefore, by Proposition \ref{p2.4}, there exists an integer $i_0\in [m]$ so that
\begin{align} \label{e2.26}
\|d_{i_0}\|_{L_p} \mik \frac{1}{\sqrt{ m(p-1)} }.
\end{align}
We claim that the set $J\coloneqq J_{i_0}$ is as desired.

To this end, fix $1\mik r <p$. First observe that, conditioning further on $\mathcal{F}_{J_{i_0}}$,
\begin{equation} \label{e2.27}
\big\| \ave[ f(\bbx)\, |\, \mathcal{F}_{J_{i_0}}] - \ave\big[\ave[f(\bbx)\, |\, \mathcal{A}_{i_0-1}]\, |\,
\mathcal{F}_{J_{i_0}}\big]\big\|_{L_p} \! =  \big\| \ave[d_{i_0}\, |\, \mathcal{F}_{J_{i_0}}]\big\|_{L_p}
\mik \frac{1}{\sqrt{m(p-1)}},
\end{equation}
where we have used the fact that $\mathcal{F}_{J_{i_0}}\subseteq \mathcal{A}_{i_0}$, the contractive property
of the conditional expectation once more, and \eqref{e2.26}. By the triangle inequality and taking into
account \eqref{e2.27} and the monotonicity of the $L_p$-norms, we obtain that
\begin{align}
\label{e3.28} \big\| \ave[f(\bbx)\, |\, \mathcal{F}_{J_{i_0}}] -\ave[f(\bbx)]\big\|_{L_r} &\mik
\frac{1}{\sqrt{m(p-1)}} + \\
&  \ \ \ \ + \big\| \ave\big[ \ave[ f(\bbx)\, |\, \mathcal A_{i_0-1}]\, |\, \mathcal{F}_{J_{i_0}}\big]
-\ave[f(\bbx)]\big\|_{L_r}. \nonumber
\end{align}
Finally, by \eqref{e2.24} and our assumption that the random array $\bbx$ is $(\beta,\ell)$-dissociated,
we see that the $\sigma$-algebras $\mathcal{F}_{J_{i_0}}$ and $\mathcal{A}_{i_0-1}$ are $\beta$-mixing
in the sense of Definition \ref{d2.1}. By Proposition \ref{p2.7}, we conclude that
\begin{align} \label{e2.29}
\big\| \ave[ f(\bbx)\, |\, \mathcal{F}_{J_{i_0}}] - \ave[f(\bbx)]\big\|_{L_r} \mik
\frac{1}{ \sqrt{m(p-1)} } + 10 \beta^{\frac{1}{r} - \frac{1}{p}}
\end{align}
and the proof is completed.

\subsection{Necessity of approximate dissociativity} \label{subsec2.6}

We close this section with the following proposition, which shows that the assumption
of approximate dissociativity in Theorem~\ref{t2.2} is necessary.
\begin{prop} \label{p2.8}
Let $n,d,\ell$ be positive integers with $n\meg \ell\meg d$, let $0<\beta\mik 1$, let $\bbx$
be a spreadable, $d$-dimensional random array on $[n]$ whose entries take values in a measurable space $\mathcal{X}$,
and assume that $\bbx$ is not $(\beta,\ell)\text{-dissociated}$. Then there exists a measurable function
$f\colon \mathcal{X}^{\binom{[n]}{d}}\to \{0,1\}$ such that for every $I\in \binom{[n]}{\ell}$ we have
\begin{equation} \label{e2.30}
\prob\big( \big|\ave[f(\bbx)\, |\, \mathcal{F}_I] - \ave[f(\bbx)]\big|\meg \beta/2\big)\meg \beta/2.
\end{equation}
\end{prop}
\begin{proof}
Since the random array $\bbx$ is spreadable and not $(\beta,\ell)$-dissociated, there exist two
integers $j,k\meg d$ with $j+k\mik \ell$, and two events $A\in\mathcal{F}_{[j]}$ and $B\in\mathcal{F}_K$,
where $K\coloneqq \{j+1,\dots, k+j\}$, such that $|\prob(A\cap B)-\prob(A)\, \prob(B)|\meg \beta$.
We select a measurable subset $A'$ of $\mathcal{X}^{\binom{[j]}{d}}$ such that the events
$[\bbx_{[j]}\in A']$ and $A$ agree almost surely, and we set $\widetilde{A}\coloneqq \pi^{-1}(A')$,
where $\pi\colon\mathcal{X}^{\binom{[n]}{d}}\to \mathcal{X}^{\binom{[j]}{d}}$ denotes the natural projection.
Finally, we define $f\colon \mathcal{X}^{\binom{[n]}{d}}\to \{0,1\}$ by $f=\mathbf{1}_{\widetilde{A}}$.

We claim that $f$ is as desired. Indeed, let $I\in \binom{[n]}{\ell}$ be arbitrary.
We select $L\in \binom{I}{k}$ with $\min(L)>j$. Invoking the spreadability of $\bbx$
and the choice of $A$ and~$B$, we may also select $\Gamma\in\mathcal{F}_L$ such that
\begin{equation} \label{e2.31}
\big|\prob(A\cap \Gamma) -\prob(A)\, \prob(\Gamma)\big|\meg \beta.
\end{equation}
Observing that $\prob(A)=\ave[f(\bbx)]$ and $\prob(A\cap \Gamma)= \ave[f(\bbx)\mathbf{1}_{\Gamma}]$,
and using the fact that $\Gamma\in\mathcal{F}_L\subseteq \mathcal{F}_I$, we obtain that
\begin{align} \label{e2.32}
\beta \stackrel{\eqref{e2.31}}{\mik} \big|\ave[ \big(f(\bbx)- \ave[f(\bbx)]\big)\mathbf{1}_{\Gamma}]\big|
 = \big|\ave[ \big( \ave[f(\bbx)\, |\, \mathcal{F}_I]- \ave[f(\bbx)]\big)\mathbf{1}_{\Gamma}]\big|.
\end{align}
Since $\big|\ave[f(\bbx)\, |\, \mathcal{F}_I] - \ave[f(\bbx)]\big|\mik 1$,
\eqref{e2.32} is easily seen to imply \eqref{e2.30}.\end{proof}
\begin{rem} \label{r2.9}
Notice that if the random array $\bbx$ in Proposition \ref{p2.8} is boolean, then the function $f$ defined above
is just a polynomial of degree at most $\binom{\ell}{d}$.
\end{rem}


\section{The box independence condition propagates} \label{sec3}

\numberwithin{equation}{section}

\subsection{The main result} \label{subsec3.1}

We start by introducing some pieces of notation and some terminology.
Let $n,d$ be a positive integers with $n\meg 2d$; for every finite sequence
$\mathcal{H}=(H_1,\dots,H_d)$ of nonempty finite subsets of $[n]$
with\footnote{Note that if $d=1$, then this condition is superfluous.}
$\max(H_i)<\min(H_{i+1})$ for all $i\in[d-1]$, we set
\begin{equation} \label{e3.1}
\mathrm{Box}(\mathcal{H})\coloneqq \bigg\{ s\in \binom{[n]}{d}: |s\cap H_i| = 1 \text{ for all } i\in [d]\bigg\};
\end{equation}
namely, $\mathrm{Box}(\mathcal{H})$ is the complete $d$-uniform, $d$-partite hypergraph
whose parts are the sets $H_1,\dots,H_d$. If, in addition, we have $|H_i|=2$ for all $i\in [d]$,
then we say that the set $\mathrm{Box}(\mathcal{H})$ is a \emph{$d$-dimensional box of $[n]$}.
By $\mathrm{Box}(d)$ we shall denote the $d$-dimensional box corresponding to the sequence
$(\{1,2\},\dots,\{2d-1,2d\})$, that is,
\begin{equation} \label{e3.2}
\mathrm{Box}(d)=\bigg\{ s\in \binom{[n]}{d}: |s\cap \{2i-1,2i\}| = 1 \text{ for all } i\in [d]\bigg\}.
\end{equation}
We proceed with the following definition. Note that the ``$(\vartheta,\mathcal{S})$-box independence"
condition introduced below is the one-sided version of \eqref{e1.6}; we will work with this slightly
weaker version since it is more amenable to an inductive argument.
\begin{defn} \label{d3.1}
Let $n,d$ be positive integers with $n\meg 2d$, let $\mathcal{X}$ be a nonempty finite set,
and let $\bbx=\langle X_s : s\in \binom{[n]}{d} \rangle$ be an $\mathcal{X}$-valued,
$d$-dimensional random array on~$[n]$. Also let $\mathcal{S}$ be a nonempty subset of $\mathcal{X}$.
\begin{enumerate}
\item[(i)] \emph{(Box independence)} Let $\vartheta>0$. We say that $\bbx$ is
\emph{$(\vartheta,\mathcal{S})$-box independent} if for every $d$-dimensional
box $B$ of\, $[n]$ and every $a\in \mathcal{S}$ we have
\begin{equation} \label{e3.3}
\mathbb{P}\Big( \bigcap_{s\in B}[X_s=a] \Big) \mik \prod_{s\in B}\mathbb{P}\big([X_s=a]\big) + \vartheta.
\end{equation}
\item[(ii)] \emph{(Approximate independence)} Set $\ell\coloneqq \binom{\lfloor n/2 \rfloor}{d}$,
and let $\boldsymbol{\gamma}=(\gamma_k)_{k=1}^\ell$ be a finite sequence of positive reals.
We say that $\bbx$ is \emph{$(\boldsymbol{\gamma},\mathcal{S})$-independent} if for every nonempty
subset $\mathcal{F}$ of $\binom{[n]}{d}$ such that $\bigcup\mathcal{F}$ has cardinality at most~$n/2$,
and every collection $(a_s)_{s\in\mathcal{F}}$ of elements of $\mathcal{S}$ we have
\begin{equation} \label{e3.4}
\bigg| \mathbb{P}\Big( \bigcap_{s\in\mathcal{F}}[X_s=a_s] \Big) -
\prod_{s\in\mathcal{F}} \mathbb{P}\big([X_s=a_s]\big) \bigg| \mik \gamma_{|\mathcal{F}|}.
\end{equation}
\end{enumerate}
\end{defn}
We are ready to state the main result in this section. It is the higher-dimensional version of
Theorem \ref{t1.5}, and its proof is given in Section \ref{sec4}. (The numerical invariants
appearing below are defined in Subsection \ref{subsec4.2}, and they are estimated in Lemma \ref{l4.4}.)
\begin{thm} \label{t3.2}
Let $d,n$ be positive integers with $n\meg 4d$, let $0<\eta,\vartheta \mik 1$,
and set $\ell\coloneqq \binom{\lfloor n/2 \rfloor}{d}$. Then there exists a sequence
$\boldsymbol{\gamma}=(\gamma_k(\eta,\vartheta,d,n))_{k=1}^\ell$ of positive reals~such~that
\begin{equation} \label{e3.5}
\gamma_k(\eta,\vartheta,d,n) \mik 36\,k\, 2^d\, \big(\!\sqrt[4^d]{1/n}+\sqrt[4^d]{\eta} +
\sqrt[4^d]{\vartheta}\big)
\end{equation}
for every $k\in [\ell]$, and satisfying the following property.

Let $\mathcal{X}$ be a finite set, let $\mathcal{S}$ be a nonempty subset of $\mathcal{X}$, and let $\bbx$ be an
$\mathcal{X}$-valued, $\eta$-spreadable, $d$-dimensional random array on $[n]$. If $\bbx$ is
$(\vartheta,\mathcal{S})$-box independent, then $\bbx$ is also  $(\boldsymbol{\gamma},\mathcal{S})$-independent.
\end{thm}
Observe that the estimate \eqref{e3.5} yields that the quantity $\gamma_k(\eta,\vartheta,d,n)$
tends to zero as $n$ tends to infinity and $\eta,\vartheta$ go to zero.

\subsection{Consequences} \label{subsec3.2}

The rest of this section is devoted to the proof of two consequences
of Theorem \ref{t3.2}. The first consequence shows that the box independence condition
implies approximate dissociativity. Specifically, we have the following corollary.
\begin{cor} \label{c3.3}
Let $d,\ell,m$ be positive integers with $\ell\meg 2d$ and $m\meg 2$, and let $0<\beta\mik 1$.
Also let $n$ be a positive integer and $0<\eta,\vartheta \mik 1$ with
\begin{equation} \label{e3.6}
\max\big\{ n^{-1},\eta,\vartheta\big\} \mik \frac13\,
\bigg( \frac{\beta}{108\, \binom{\ell}{d}\, 2^d\, m^{3\binom{\ell}{d}}}\bigg)^{4^d}.
\end{equation}
Finally, let $\mathcal{X}$ be a set with $|\mathcal{X}|=m$, let $\mathcal{S}$ be a subset of $\mathcal{X}$
with $|\mathcal{S}|=|\mathcal{X}|-1$, and let $\bbx$ be an $\mathcal{X}$-valued, $\eta$-spreadable,
$d$-dimensional random array on $[n]$. If $\bbx$ is $(\vartheta,\mathcal{S})$-box independent,
then $\bbx$ is $(\beta,\ell)$-dissociated $($see Definition \emph{\ref{d2.1}}$)$.
\end{cor}
The second consequence of Theorem \ref{t3.2} shows that the box independence forces all subarrays indexed
by $d$-dimensional boxes to behave independently. More precisely, we have the following corollary.
\begin{cor} \label{c3.4}
Let $d,m$ be positive integers with $m\meg 2$, and let $0<\gamma \mik 1$. Also let $n$~be a positive integer
and $0<\eta,\vartheta \mik 1$ with
\begin{equation} \label{e3.7}
\max\big\{ n^{-1},\eta,\vartheta\big\} \mik \frac13\, \bigg( \frac{\gamma}{36\, 4^d\, m^{2^d}}\bigg)^{4^d}.
\end{equation}
Finally, let $\mathcal{X}$ be a set with $|\mathcal{X}|=m$, let $\mathcal{S}$ be a subset of $\mathcal{X}$ with $|\mathcal{S}|=|\mathcal{X}|-1$, and let $\bbx$ be an $\mathcal{X}$-valued, $\eta$-spreadable,
$d$-dimensional random array on $[n]$. If\, $\bbx$ is $(\vartheta,\mathcal{S})\text{-box}$ independent,
then for every  $d$-dimensional box\, $B$ of\, $[n]$ and every collection
$(a_s)_{s\in B}$ of elements of\, $\mathcal{X}$ we have
\begin{equation} \label{e3.8}
\bigg| \mathbb{P}\Big( \bigcap_{s\in B}[X_s = a_s] \Big) - \prod_{s\in B} \mathbb{P}\big([X_s = a_s]\big) \bigg| \mik \gamma.
\end{equation}
\end{cor}
\begin{rem}
Although Corollary \ref{c3.4} is weaker than Theorem \ref{t3.2}, a~direct proof of the estimate \eqref{e3.8}
is likely to require the whole machinery presented in Section \ref{sec4}.
\end{rem}
Corollaries \ref{c3.3} and \ref{c3.4} following from the following consequence of Theorem \ref{t3.2}.
\begin{lem} \label{l3.6}
Let $d,m,\kappa$ be positive integers with $m\meg 2$, and let $0<\gamma \mik 1$. Also let $n$~be a positive integer
and $0<\eta,\vartheta \mik 1$ with
\begin{equation} \label{e3.9}
\max\big\{ n^{-1},\eta,\vartheta\big\} \mik \frac13\, \bigg( \frac{\gamma}{36\, \kappa\, 2^d\, m^\kappa}\bigg)^{4^d}.
\end{equation}
Finally, let $\mathcal{X}$ be a set with $|\mathcal{X}|=m$, let $\mathcal{S}$ be a subset of $\mathcal{X}$ with $|\mathcal{S}|=|\mathcal{X}|-1$, and let $\bbx$ be an $\mathcal{X}$-valued, $\eta$-spreadable,
$d$-dimensional random array on $[n]$. If\, $\bbx$ is $(\vartheta,\mathcal{S})\text{-box}$ independent,
then for every nonempty subset $\mathcal{F}$ of\, $\binom{[n]}{d}$ with $|\mathcal{F}|\mik \kappa$
and every collection $(a_s)_{s\in \mathcal{F}}$ of elements of $\mathcal{X}$ we have
\begin{equation} \label{e3.10}
\bigg| \mathbb{P}\Big( \bigcap_{s\in \mathcal{F}}[X_s=a_s] \Big) -
\prod_{s\in \mathcal{F}} \mathbb{P}\big([X_s=a_s]\big) \bigg| \mik \gamma.
\end{equation}
\end{lem}
Notice that the conclusion of Lemma \ref{l3.6} is essentially $(\gamma,\mathcal{X})$-independence for the
constant function $\gamma$, except that it holds when $|\mathcal{F}|\mik \kappa$ instead of
$|\bigcup \mathcal{F}|\mik n/2$. We defer the proof of Lemma \ref{l3.6} to Subsection \ref{subsec3.3} below.
At this point, let us give the proofs of Corollaries \ref{c3.3} and \ref{c3.4}.
\begin{proof}[Proof of Corollary \emph{\ref{c3.3}}]
Set $\kappa\coloneqq \binom{\ell}{d}$ and $\gamma\coloneqq \frac{1}{3}m^{-2\kappa}\beta$.
Also let $J,K$ be subsets of~$[n]$ with $|J|,|K|\meg d$, $|J| + |K| \mik \ell$ and
$\max(J)<\min(K)$, and let $A\in\mathcal{F}_J$ and $B\in\mathcal{F}_K$.
We will show that $|\mathbb{P}(A\cap B) - \mathbb{P}(A)\, \mathbb{P}(B) | \mik \beta$.

Since $A$ belongs to the $\sigma$-algebra generated by $\mathbf{X}_J$, there exists a collection $\mathcal{A}$
of maps of the form $\mathbf{a}\colon \binom{J}{d}\to \mathcal{X}$ such that
\begin{equation} \label{e3.11}
A = \bigcup_{\mathbf{a}\in\mathcal{A}} \bigcap_{s\in \binom{J}{d}} [X_s =\mathbf{a}(s)].
\end{equation}
Similarly, there exists a collection $\mathcal{B}$ of maps of the form
$\mathbf{b}\colon \binom{K}{d}\to \mathcal{X}$ such that
\begin{equation} \label{e3.12}
B = \bigcup_{\mathbf{b}\in\mathcal{B}} \bigcap_{t\in \binom{K}{d}} [X_t = \mathbf{b}(t)].
\end{equation}
For every $\mathbf{a} \in \mathcal{A}$ we set
$A_\mathbf{a}\coloneqq \bigcap_{s\in \binom{J}{d}} [X_s = \mathbf{a}(s)]$, respectively,
for every $\mathbf{b} \in \mathcal{B}$ we set
$B_{\mathbf{b}}\coloneqq \bigcap_{t\in \binom{K}{d}} [X_t = \mathbf{b}(t)]$.
By Lemma \ref{l3.6}, for every $\mathbf{a}\in \mathcal{A}$ and every $\mathbf{b}\in \mathcal{B}$, we have
\begin{align}
\label{e3.13} & \bigg|\mathbb{P}( A_\mathbf{a}\cap B_\mathbf{b} ) -
\prod_{s\in \binom{J}{d}} \mathbb{P}\big([X_s=\mathbf{a}(s)]\big)
\prod_{t\in \binom{K}{d}} \mathbb{P}\big( [X_t=\mathbf{b}(t)]\big) \bigg| \mik \gamma, \\
\label{e3.14} & \bigg|\mathbb{P}( A_{\mathbf{a}}) - \prod_{s\in \binom{J}{d}}
\mathbb{P}\big( [X_s=\mathbf{a}(s)]\big)\bigg| \mik \gamma,  \\
\label{e3.15} & \bigg|\mathbb{P}( B_{\mathbf{b}}) - \prod_{t\in \binom{K}{d}}
\mathbb{P}\big( [X_t=\mathbf{b}(t)]\big)\bigg| \mik \gamma;
\end{align}
consequently, $|\mathbb{P}(A_\mathbf{a} \cap B_\mathbf{b}) - \mathbb{P}(A_\mathbf{a})\, \mathbb{P}(B_\mathbf{b})| \mik 3\gamma$.
On the other hand, by identities \eqref{e3.11} and \eqref{e3.12}, we see that
$A\cap B = \bigcup_{\mathbf{a}\in\mathcal{A},\mathbf{b}\in\mathcal{B}}A_\mathbf{a} \cap B_\mathbf{b}$;
moreover, the collections $\langle A_\mathbf{a}: \mathbf{a}\in \mathcal{A}\rangle$ and
$\langle B_\mathbf{b}: \mathbf{b}\in\mathcal{B}\rangle$ consist of pairwise disjoint events. Thus, we have
\begin{equation} \label{e3.16}
\mathbb{P}(A\cap B) = \!\!\!\sum_{\mathbf{a}\in \mathcal{A},\mathbf{b}\in\mathcal{B}}
\!\!\!\mathbb{P}(A_\mathbf{a}\cap B_\mathbf{b}), \ \
\mathbb{P}(A) = \sum_{\mathbf{a}\in \mathcal{A}}\mathbb{P}(A_\mathbf{a}) \ \text{ and } \
\mathbb{P}(B) = \sum_{\mathbf{b}\in \mathcal{B}}\mathbb{P}(B_\mathbf{b}).
\end{equation}
Therefore, we conclude that
\begin{align}
\label{e3.17} |\mathbb{P}(A\cap B) - \mathbb{P}(A)\,\mathbb{P}(B)|  \mik
\sum_{\mathbf{a}\in \mathcal{A},\mathbf{b}\in\mathcal{B}} & \!\!
|\mathbb{P}(A_\mathbf{a}\cap  B_\mathbf{b}) - \mathbb{P}(A_\mathbf{a})\, \mathbb{P}(B_\mathbf{b})| \\
& \mik \, 3\gamma\, |\mathcal{A}|\, |\mathcal{B|} \mik 3\gamma\, m^{2\binom{\ell}{d}} =\beta. \qedhere
 \nonumber
\end{align}
\end{proof}
\begin{proof}[Proof of Corollary \emph{\ref{c3.4}}]
It follows from Lemma \ref{l3.6} applied for ``$\kappa = 2^d$"\!.
\end{proof}

\subsection{Proof of Lemma \ref{l3.6}} \label{subsec3.3}

The result follows from Theorem \ref{t3.2} and the inclusion-exclusion formula.
We start by setting $\gamma'\coloneqq m^{-\kappa}\gamma$. By \eqref{e3.5} and \eqref{e3.9},
we see that $n\meg \max\{4d,d\kappa\}$ and $\gamma_k(\eta,\vartheta,d,n)\mik \gamma'$ for every $k\in [\kappa]$.
Therefore, by Theorem \ref{t3.2}, for every nonempty $\mathcal{F}^*\subseteq \binom{[n]}{d}$
with $|\mathcal{F}^*|\mik\kappa$ and every collection $(a_s)_{s\in\mathcal{F}^*}$ of elements~of~$\mathcal{S}$,
\begin{equation} \label{e3.18}
\bigg| \mathbb{P}\Big( \bigcap_{s\in \mathcal{F}^*}[X_s=a_s] \Big) -
\prod_{s\in \mathcal{F}^*} \mathbb{P}\big([X_s=a_s]\big) \bigg| \mik \gamma'.
\end{equation}
Let $\mathcal{F}$ be a nonempty subset of $\binom{[n]}{d}$ with $|\mathcal{F}|\mik \kappa$,
and let $(a_s)_{s\in\mathcal{F}}$ be a collection of elements of $\mathcal{X}$.
Set $\mathcal{F}'\coloneqq\{s\in\mathcal{F}: a_s\in \mathcal{S}\}$ and
$\mathcal{G}\coloneqq \mathcal{F} \setminus \mathcal{F}'$; observe that for every $t\in \mathcal{G}$ the events
$\langle [X_t = a]: a\in\mathcal{S}\rangle$ are pairwise disjoint and, moreover,
\begin{equation} \label{e3.19}
[X_t = a_t] = \bigg( \bigcup_{a\in\mathcal{S}} [X_t = a] \bigg)^{\complement}.
\end{equation}
(For any event $E$, by $E^{\complement}$ we denote its complement.) Thus, for every $t\in\mathcal{G}$
we have $\prob\big([X_t = a_t]\big) = 1- \sum_{a\in \mathcal{S}} \prob\big([X_t = a]\big)$ and, consequently,
\begin{align}
\label{e3.20} \prod_{s\in\mathcal{F}} \prob\big([X_s = a_s]\big) & =
\prod_{s\in\mathcal{F}'} \prob\big([X_s = a_s]\big)\, \prod_{t\in\mathcal{G}}
\Big(1- \sum_{a\in \mathcal{S}} \prob\big([X_t = a]\big) \Big) \\
& = \sum_{\substack{\mathcal{W}\subseteq\mathcal{G} \\ \mathbf{a}\colon \mathcal{W}\to\mathcal{S}}}
(-1)^{|\mathcal{W}|} \prod_{t\in\mathcal{W}} \prob\big( [X_t = \mathbf{a}(t)] \big)
\prod_{s\in\mathcal{F}'} \prob\big([X_s = a_s]\big) \nonumber
\end{align}		
with the convention that the product over an empty index-set is equal to $1$. Moreover,
\begin{align}
\label{e3.21} \prob\Big( \bigcap_{s\in\mathcal{F}} & [X_s = a_s] \Big) \stackrel{\eqref{e3.19}}{=}
\prob\bigg( \bigcap_{s\in\mathcal{F}'}[X_s = a_s] \cap \bigcap_{t\in\mathcal{G}}
\Big(\bigcup_{a\in\mathcal{S}} [X_t = a] \Big)^{\complement} \bigg) \\
& = \prob\Big( \bigcap_{s\in\mathcal{F}'}[X_s = a_s] \Big) - \prob\Big( \bigcap_{s\in\mathcal{F}'} [X_s = a_s]
\cap \Big( \bigcup_{t\in\mathcal{G}} \bigcup_{a\in\mathcal{S}} [X_t = a] \Big) \Big). \nonumber
\end{align}
Next observe that for every nonempty subset $\mathcal{W}$ of $\mathcal{G}$ we have
\begin{equation} \label{e3.22}
\bigcap_{t\in\mathcal{W}}\bigcup_{a\in\mathcal{S}}[X_t = a] =
\bigcup_{\mathbf{a}\colon \mathcal{W}\to \mathcal{S}} \Big( \bigcap_{t\in\mathcal{W}} [X_t = \mathbf{a}(t)]\Big)
\end{equation}
and the events $\langle \bigcap_{t\in\mathcal{W}}[X_t = \mathbf{a}(t)] : \mathbf{a}\colon \mathcal{W}\to \mathcal{S}\rangle$
are pairwise disjoint. Hence, by the inclusion-exclusion formula,
\begin{align}
\label{e3.23} \prob\Big( & \bigcap_{s\in\mathcal{F}'}[X_s = a_s] \cap
\Big( \bigcup_{t\in\mathcal{G}}  \bigcup_{a\in\mathcal{S}} [X_t = a] \Big) \Big) \\
& = \sum_{\emptyset\neq\mathcal{W}\subseteq\mathcal{G}} (-1)^{|\mathcal{W}|-1}\,
     \prob\Big( \bigcap_{t\in\mathcal{W}}\Big( \bigcap_{s\in\mathcal{F}'}[X_s = a_s]
                \cap \Big( \bigcup_{a\in\mathcal{S}} [X_t = a] \Big) \Big)\Big) \nonumber \\
& \!\!\!\!\stackrel{\eqref{e3.22}}{=} \sum_{\emptyset\neq\mathcal{W}\subseteq\mathcal{G}} (-1)^{|\mathcal{W}|-1}\,
    \prob\Big( \bigcap_{s\in\mathcal{F}'}[X_s = a_s] \cap
		\Big( \bigcup_{\mathbf{a}\colon\mathcal{W}\to \mathcal{S}} \Big(\bigcap_{t\in\mathcal{W}}
		   [X_t = \mathbf{a}(t)] \Big) \Big) \Big) \nonumber \\
& = \sum_{\emptyset\neq\mathcal{W}\subseteq\mathcal{G}} \; \sum_{\mathbf{a}\colon \mathcal{W}\to \mathcal{S}}
    (-1)^{|\mathcal{W}|-1}\, \prob\Big( \bigcap_{s\in\mathcal{F}'}[X_s = a_s] \cap
		\bigcap_{t\in\mathcal{W}} [X_t = \mathbf{a}(t)] \Big). \nonumber
\end{align}
Combining identities \eqref{e3.21} and \eqref{e3.23}, we see that
\begin{equation} \label{e3.24}
\prob\Big( \bigcap_{s\in\mathcal{F}}[X_s = a_s] \Big) =
\sum_{\substack{\mathcal{W}\subseteq\mathcal{G} \\ \mathbf{a}\colon \mathcal{W}\to \mathcal{S}}}
  (-1)^{|\mathcal{W}|}\, \prob\Big( \bigcap_{s\in\mathcal{F}'}[X_s = a_s] \cap
	\bigcap_{t\in\mathcal{W}} [X_t = \mathbf{a}(t)] \Big)
\end{equation}
with the convention that the intersection over an empty index-set is equal to the whole sample space.
Finally, by identities \eqref{e3.20} and \eqref{e3.24} and the triangle inequality, we conclude that
the quantity $\big|\prob \big(\bigcap_{s\in\mathcal{F}}[X_s = a_s]\big)- \prod_{s\in\mathcal{F}} \prob\big([X_s = a_s]\big)\big|$
is upper bounded~by
\begin{align}
\label{e3.25} \sum_{\substack{\mathcal{W}\subseteq\mathcal{G}\\ \mathbf{a}\colon \mathcal{W}\to \mathcal{S}}}
\bigg| \mathbb{P}\Big( \bigcap_{s\in\mathcal{F}'}& [X_s = a_s] \cap \bigcap_{t\in\mathcal{W}} [X_t = \mathbf{a}(t)] \Big)\; - \\
& - \prod_{s\in\mathcal{F}'} \prob\big( [X_s = a_s]\big) \prod_{t\in\mathcal{W}}
\prob\big( [X_t = \mathbf{a}(t)]\big) \bigg| \stackrel{\eqref{e3.18}}{\mik}m^\kappa \gamma' = \gamma. \nonumber
\end{align}
The proof of Lemma \ref{l3.6} is completed.


\section{Proof of Theorem \ref*{t3.2}} \label{sec4}

\numberwithin{equation}{section}

This section is devoted to the proof of Theorem \ref{t3.2}, which proceeds by induction on the dimension $d$.
In a nutshell, the argument is based on repeated averaging and an appropriate version of the weak law of large
numbers in order to gradually upgrade the box independence condition. The combinatorial heart of the matter
lies in the selection of this averaging.

\subsection{Toolbox} \label{subsec4.1}

We begin by presenting three lemmas that are needed for the proof
of Theorem \ref{t3.2}, but they are not directly related with the main argument.
\begin{lem} \label{l4.1}
Let $m$ be a positive integer, let $\delta>0$ and let $A_1,\dots,A_m$ be events in a probability space
such that for every $i,j\in [m]$ with $i\neq j$ we have
\begin{equation} \label{e4.1}
\mathbb{P}( A_i \cap A_j ) \mik \mathbb{P}(A_i)\,\mathbb{P}(A_j) + \delta.
\end{equation}
Then, setting $Z\coloneqq\frac{1}{m}\sum_{i=1}^{m}\mathbf{1}_{A_i}$, we have
\begin{equation} \label{e4.2}
\mathrm{Var}(Z)\mik \frac{1}{m}+\delta.
\end{equation}
\end{lem}
\begin{proof}
We have
\[ \begin{split}
\mathrm{Var}(Z) = \ave \big[(Z - \ave[Z])^2 \big] = \frac{1}{m^2} \sum_{i,j\in [m]}
\ave\big[\big(\mathbf{1}_{A_i} -\mathbb{P}(A_i)\big) \big( \mathbf{1}_{A_j} - \mathbb{P}(A_j)\big)\big] \\
= \frac{1}{m^2} \Big[ \sum_{i=1}^{m} \big( \mathbb{P}(A_i) -\mathbb{P}(A_i)^2\big)
+ \sum_{\substack{i,j \in [m] \\ i\neq j}} \big(\mathbb{P}(A_i \cap A_j) -
\mathbb{P}(A_i)\, \mathbb{P}(A_j)\big) \Big] \mik \frac{1}{m} + \delta. \qedhere
\end{split} \]
\end{proof}
\begin{lem} \label{l4.2}
Let $m$ be a positive integer, let $\eta,\delta>0$ and let $E,A_1,\dots,A_m$ be events in a probability space
such that for every $i,j\in [m]$ with $i\neq j$ we have
\begin{enumerate}
\item [(i)] $| \mathbb{P}(A_i) - \mathbb{P}(A_j) |\mik \eta$,
\item [(ii)] $| \mathbb{P}(E \cap A_i) - \mathbb{P}(E\cap A_j) |\mik \eta$, and
\item [(iii)] $\mathbb{P}(A_i \cap A_j) \mik \mathbb{P}(A_i)\, \mathbb{P}(A_j) + \delta$.
\end{enumerate}
Then for every $i\in[m]$ we have
\begin{equation} \label{e4.3}
\big| \mathbb{P}(E \cap A_i) - \mathbb{P}(E)\,\mathbb{P}(A_i) \big|\mik 2\eta + \sqrt{\frac{1}{m} + \delta}.
\end{equation}
\end{lem}
\begin{proof}
Set $Z\coloneqq \frac{1}{m} \sum_{j=1}^{m}\mathbf{1}_{A_j}$. Let $i\in[m]$. Notice that, by the triangle inequality,
\begin{align} \label{e4.4}
\!\! \big| \mathbb{P}(E \cap A_i)-\mathbb{P}(E)\,\mathbb{P}(A_i) \big| &
= \big| \ave[\mathbf{1}_E \mathbf{1}_{A_i}] - \ave[\mathbf{1}_E \mathbb{P}(A_i)] \big| \\
& \mik \big| \ave[\mathbf{1}_E \mathbf{1}_{A_i}] - \ave[\mathbf{1}_E Z]\big|
+ \big| \ave[\mathbf{1}_E Z] - \ave[\mathbf{1}_E \ave[Z] ] \big| + \nonumber \\
&\;\;\;\;\;\;  + \big| \ave[\mathbf{1}_E \ave[Z] ] - \ave[\mathbf{1}_E \mathbb{P}(A_i)] \big|. \nonumber
\end{align}
Invoking the triangle inequality again, we have
\begin{align}
\label{e4.5} & \big| \ave[\mathbf{1}_E \mathbf{1}_{A_i}] - \ave[\mathbf{1}_E Z]\big|
\mik \frac{1}{m}\sum_{j=1}^{m}\big|\mathbb{P}(E \cap A_i) - \mathbb{P}(E \cap A_j)\big| \stackrel{\text{(ii)}}{\mik}\eta, \\
\label{e4.6} & \big| \ave[\mathbf{1}_E \ave[Z] ] - \ave[\mathbf{1}_E \mathbb{P}(A_i)] \big|
\mik \mathbb{P}(E)\frac{1}{m}\sum_{j=1}^{m}\big| \mathbb{P}(A_j) - \mathbb{P}(A_i) \big|
\stackrel{\text{(i)}}{\mik}\eta.
\end{align}
Finally, by the Cauchy--Schwarz inequality, hypothesis (iii) and Lemma \ref{l4.1},
\begin{equation} \label{e4.7}
\big| \ave[\mathbf{1}_E Z ] - \ave[\mathbf{1}_E \ave[Z]] \big|
\mik \sqrt{\mathbb{P}(E)} \; \| Z - \ave[Z] \|_{L_2} \mik \sqrt{\frac{1}{m}+\delta}.
\end{equation}
The estimate \eqref{e4.3} follows from \eqref{e4.4}--\eqref{e4.7}.
\end{proof}
\begin{lem} \label{l4.3}
Let $m\meg 1$ be an integer, let $\eta>0$, and let $(A_i)_{i=1}^m$ be an $\eta$-spreadable sequence\footnote{That is,
the random vector $(\mathbf{1}_{A_1},\dots,\mathbf{1}_{A_m})$ is $\eta$-spreadable according to Definition \ref{d1.2}.}
of events in a probability space. Then for every $i,j\in[m]$ with $i\neq j$,
\begin{equation} \label{e4.8}
\mathbb{P}(A_i \cap A_j) \meg \mathbb{P}(A_i)\,\mathbb{P}(A_j) - \frac{1}{m} -3\eta.
\end{equation}
\end{lem}
\begin{proof}
Set $Z\coloneqq\frac{1}{m}\sum_{k=1}^{m}\mathbf{1}_{A_k}$. Fix $i,j\in [m]$ with $i\neq j$. Then,
by $\eta$-spreadability, we~have
\[ \big|\mathbb{P}(A_i \cap A_j) - \ave[Z^2]\big|  =
\bigg| \mathbb{P}(A_i \cap A_j) - \frac{1}{m^2} \sum_{k=1}^{m}\mathbb{P}(A_k)
- \frac{1}{m^2} \sum_{\substack{k,\ell\in [m]\\ k\neq\ell}} \mathbb{P}(A_k\cap A_\ell)\bigg|
\mik \frac{1}{m} + \eta. \]
Notice that $\eta$-spreadability also implies
\[ \big|\mathbb{P}(A_i)\; \mathbb{P}(A_j) - \ave[Z]^2\big|
 \mik \ave[Z] \big|\mathbb{P}(A_j) - \ave[Z]\big| + \mathbb{P}(A_j) \big|\mathbb{P}(A_i) - \ave[Z]\big| \mik2\eta. \]
Since $\ave[Z]^2 \mik \ave[Z^2]$, inequality \eqref{e4.8} follows from the previous two estimates.
\end{proof}

\subsection{Initializing various numerical parameters} \label{subsec4.2}

Our goal in this subsection is to define, by recursion on $d$, the numbers $\gamma_{k}(\eta,\vartheta,d,n)$
as well as some other numerical invariants that are needed for the proof of Theorem \ref{t3.2}.
(The reader is advised to skip this subsection at first reading.)

We start by setting
\begin{equation} \label{e4.9}
\gamma_k(\eta,\vartheta,1,n) \coloneqq \min\bigg\{1, (3k-1)\eta+ (k-1)
\sqrt{\frac{1}{\lfloor n/2 \rfloor} + \vartheta}\bigg\}
\end{equation}
for every $0<\eta\mik 1$, every $\vartheta>0$ and every pair of positive integers $k,n$ with $n\meg2$ and $k\mik n/2$.

Let $d\meg 2$ be an integer, and assume that the numbers $\gamma_{k}(\eta,\vartheta,d-1,n)$
have been defined for every choice of admissible parameters.
Fix $0<\eta\mik 1$ and $\vartheta>0$, and let $n$ be an integer with $n\meg 4d$. We set
\begin{align}
\label{e4.10} \vartheta_1(\eta,\vartheta,d,n) & \coloneqq (n-2d+2)^{-1/2} + (2^d+5)\sqrt{\eta} + \sqrt{\vartheta}, \\
\label{e4.11} \vartheta_2(\eta,\vartheta,d,n) & \coloneqq \frac{2^{d-1}}{n-d+1}+2^d3\eta+\vartheta, \\
\label{e4.12} \vartheta_3(\eta,\vartheta,d,n) & \coloneqq \frac{d-1}{(n-2d+2)^{1/2^{d-1}}} +
(2^d + 5) \eta^{1/2^{d-1}} + \vartheta^{1/2^{d-1}} + 3 \eta.
\end{align}
Next, for every positive integer $k\mik \binom{\lfloor (n-1)/2 \rfloor}{d-1}$ we set
\begin{align}
\label{e4.13} \gamma^{(1)}_k(\eta,\vartheta,d,n) & \coloneqq
\gamma_{k}\big(\eta, \vartheta_1(\eta,\vartheta,d,n), d-1, n-1\big) + (k+1)\eta, \\
\label{e4.14} \gamma^{(2)}_k(\eta,\vartheta,d,n) & \coloneqq
\gamma_k\big(\eta, \vartheta_2(\eta,\vartheta,d,n), d-1, n-2\big), \\
\label{e4.15} \gamma^{(3)}_k(\eta,\vartheta,d,n) & \coloneqq
2\gamma^{(1)}_k(\eta, \vartheta,d,n)+\gamma^{(2)}_k(\eta, \vartheta,d,n) + k \,\vartheta_3(\eta, \vartheta,d,n), \\
\label{e4.16} \gamma^{(4)}_k(\eta,\vartheta,d,n) & \coloneqq
\big( \gamma^{(3)}_k(\eta,\vartheta,d,n) + \lfloor n/2 \rfloor^{-1} +(2k+1)\eta \big)^{1/2} +2\eta.
\end{align}
Moreover, for every positive integer $u$ with $u\mik n/2 $ and every choice $k_1,\dots,k_u$ of positive integers with
$k_1,\dots,k_u \mik \binom{\lfloor (n-2)/2 \rfloor}{d-1}$, set
\begin{equation} \label{e4.17}
\gamma^{(5)}(\eta, \vartheta,d,n, (k_i)_{i=1}^{u})\! \coloneqq \! \gamma^{(1)}_{k_1}\!(\eta,\vartheta,d,n) +
\sum_{i=2}^u \!\big(\gamma^{(1)}_{k_i}\!(\eta,\vartheta,d,n) + \gamma^{(4)}_{k_i}\!(\eta,\vartheta,d,n)\big)
\end{equation}
with the convention that the sum in \eqref{e4.17} is equal to $0$ if $u=1$.
(Note that the sum above has at most $\min\{u-1,k_2+\cdots+k_u\}$ elements.)
Finally, for every positive integer $k\mik \binom{\lfloor n/2 \rfloor}{d}$ we define
\begin{equation} \label{e4.18}
\gamma_k(\eta,\vartheta,d,n) \coloneqq \min\Big\{1,
(k+1)\, \eta + \max\{ \gamma^{(5)}(\eta, \vartheta,d,n, (k_i)_{i=1}^{u})\}\Big\},
\end{equation}
where the above maximum is taken over all choices of positive integers $u,k_1,\dots,k_u$ satisfying
$u\mik n/2 - d$, $k_1,\dots,k_u \mik \binom{\lfloor (n-2)/2 \rfloor}{d-1}$ and $k_1+\dots+k_u = k$.
(Note that there are at most $k^k$ such choices.)

\subsubsection{Estimation of\, \texorpdfstring{$\boldsymbol{\gamma}=(\gamma_k(\eta,\vartheta,d,n))_{k=1}^\ell$}{TEXT}}
\label{subsubsec4.2.1}

The following lemma provides an estimate for the numbers $\gamma_k(\eta,\vartheta,d,n)$ introduced above.
\begin{lem} \label{l4.4}
For every $0<\eta\mik 1$, every $\vartheta>0$, every positive integer $d$, every integer $n\meg 4d$
and every positive integer $k\mik \binom{\lfloor n/2\rfloor}{d}$ we have
\begin{equation} \label{e4.19}
\gamma_k(\eta,\vartheta,d,n) \mik 36\,k\, 2^d\, \Big(\!\sqrt[4^d]{1/n} + \sqrt[4^d]{\eta} +
\max\big\{\vartheta,\!\sqrt[4^d]{\vartheta}\big\}\Big),
\end{equation}
where $\gamma_k(\eta,\vartheta,d,n)$ is as in Subsection \emph{\ref{subsec4.2}}.
\end{lem}
\begin{proof}
We start by observing that for every choice of positive integers $d$ and $k$,
the quantities $\gamma_k(\eta,\vartheta,d,n), \gamma_k^{(1)}(\eta,\vartheta,d,n),
\gamma_k^{(2)}(\eta,\vartheta,d,n), \gamma_k^{(3)}(\eta,\vartheta,d,n)$ and $\gamma_k^{(4)}(\eta,\vartheta,d,n)$
are all decreasing with respect to $n$, and increasing with respect to $\eta$ and $\vartheta$.

It is also convenient to introduce the following notation. For every pair of positive integers
$n,\ell$, every $0<\eta\mik 1$ and every $\vartheta>0$ we set
\begin{equation} \label{e4.20}
\Delta_\ell(\eta,\vartheta,n) \coloneqq \sqrt[\ell]{1/n} + \sqrt[\ell]{\eta} +
\max\big\{\vartheta,\sqrt[\ell]{\vartheta}\big\}.
\end{equation}
Thus, it suffices to prove that
\begin{equation} \label{e4.21}
  \gamma_k(\eta,\vartheta,d,n)\mik 36\, k \, 2^d\Delta_{4^d}(\eta,\vartheta,n)
\end{equation}
for every pair of positive integers $n,d$ with $n\meg 4d$, every $0<\eta\mik 1$, every $\vartheta>0$,
and every positive integer $k\mik \binom{\lfloor n/2 \rfloor}{d}$.

To that end we proceed by induction on $d$. The base case ``$d=1$" follows readily from~\eqref{e4.9}.
Next, let $d$ be a positive integer with $d\meg 2$ and assume that \eqref{e4.21} holds for $d-1$,
every integer $n\meg 4d-4$, every $0<\eta\mik 1$, every $\vartheta>0$, and every positive integer
$k \mik \binom{\lfloor n/2 \rfloor}{d-1}$. Fix an integer $n\meg 4d$, $0<\eta\mik 1$ and
$\vartheta>0$; by \eqref{e4.10}, \eqref{e4.11} and \eqref{e4.12},~we~have
\begin{align}
\label{e4.22} \vartheta_1(\eta,\vartheta,d,n) & \mik (2^d+5) \Delta_2(\eta,\vartheta,n), \\
\label{e4.23} \vartheta_2(\eta,\vartheta,d,n) & \mik 2^d 3   \Delta_1(\eta,\vartheta,n), \\
\label{e4.24}  \vartheta_3(\eta,\vartheta,d,n) & \mik (2^d+8) \Delta_{2^{d-1}}(\eta,\vartheta,n).
\end{align}
Set $\overline{\vartheta}\coloneqq 2^d\,3\,\Delta_2(\eta,\vartheta,n)$ and notice that
$\overline{\vartheta}\meg \max\{\vartheta_1(\eta,\vartheta,d,n), \vartheta_2(\eta,\vartheta,d,n)\}$.
Moreover, for every positive integer $k\mik \binom{\lfloor (n-1)/2 \rfloor}{d-1}$ set
\begin{equation} \label{e4.25}
\overline{\gamma}_k \coloneqq \gamma_k(\eta,\overline{\vartheta},d-1,n-1)
\end{equation}
and observe that, by our inductive assumption,
\begin{align} \label{e4.26}
\overline{\gamma}_k & \mik 36\, k\, 2^{d-1}\Delta_{4^{d-1}}(\eta,\overline{\vartheta}, n-1) \\
& \mik 36\,k\, 2^{d-1}\Big(\!\!\sqrt[4^{d-1}]{2/n} + \!\sqrt[4^{d-1}]{\eta} +
\max\{\overline{\vartheta},\sqrt[4^{d-1}]{\overline{\vartheta}}\}\Big) \nonumber \\
& \mik 36\, k\, 2^{d-1} \Big(\!\sqrt[4^{d-1}]{2/n} + \sqrt[4^{d-1}]{\eta} +
2^d3 \Delta_{2\cdot4^{d-1}}(\eta,\theta,n)\Big) \nonumber \\
& \mik 36\, k\, 2^{d-1} (2^d3+2)\Delta_{2\cdot4^{d-1}}(\eta,\theta,n)
\mik 9\cdot 7 \cdot 2^{2d} k\, \Delta_{2\cdot4^{d-1}}(\eta,\theta,n). \nonumber
\end{align}
Additionally, by \eqref{e4.13}--\eqref{e4.16} and the monotonicity properties
of $\gamma_k(\eta,\vartheta,d,n)$, for every positive integer
$k\mik \binom{\lfloor (n-1)/2 \rfloor}{d-1}$ we have
\begin{align} \label{e4.27}
\gamma^{(4)}_k(\eta,& \vartheta,d,n) \mik
\Big(3\overline{\gamma}_k + \lfloor n/2 \rfloor^{-1} +(4k+3)\eta +
k \,\vartheta_3(\eta, \vartheta,d,n)\Big)^{1/2} +2\eta \\
& \stackrel{\eqref{e4.24}}{\mik} \sqrt{3\overline{\gamma}_k} +
\bigg(\frac{4}{n} +(4k+3)\eta + k \,(2^d+8) \Delta_{2^{d-1}}(\eta,\vartheta,n)\bigg)^{1/2}
+ 2^{\frac{d}{2}}\eta \nonumber \\
& \ \ \mik  \sqrt{3\overline{\gamma}_k} + 4\,k\,2^{\frac{d}{2}}\Delta_{2^{d}}(\eta,\vartheta,n) \nonumber
\end{align}
and therefore, invoking the fact that $\overline{\gamma}_k\mik 1$, we obtain that
\begin{align} \label{e4.28}
\gamma^{(1)}_k(\eta,\vartheta,d,n)  & + \gamma^{(4)}_k(\eta,\vartheta,d,n) \mik
(1+\sqrt{3})\sqrt{\overline{\gamma}_k} + 4\,k\,2^{\frac{d}{2}}\Delta_{2^{d}}(\eta,\vartheta,n) +(k+1)\eta \\
& \ \, \mik (1+\sqrt{3}) \sqrt{\overline{\gamma}_k} + 5\,k\,2^{\frac{d}{2}}\Delta_{2^{d}}(\eta,\vartheta,n)\nonumber \\
& \stackrel{\eqref{e4.26}}{\mik} 3\sqrt{7} (1+\sqrt{3})  2^d\, k\, \Delta_{4^d}(\eta,\theta,n) +
5\,k\, 2^{\frac{d}{2}}\Delta_{2^{d}}(\eta,\vartheta,n) \nonumber \\
& \ \, \mik  27\, k\, 2^d \Delta_{4^d}(\eta,\theta,n). \nonumber
\end{align}
By \eqref{e4.17} and \eqref{e4.18} and using the linearity of the upper bound in \eqref{e4.28}
with respect to the parameter $k$, we conclude that for every positive integer
$k\mik \binom{\lfloor n/2 \rfloor}{d}$,
\begin{equation} \label{e4.29}
\gamma_k(\eta,\vartheta,d,n) \mik 29\, k\, 2^d \Delta_{4^d}(\eta,\theta,n) \mik
36 \,k \,2^d \Delta_{4^d}(\eta,\theta,n). \qedhere
\end{equation}
\end{proof}

\subsection{The inductive hypothesis} \label{subsec4.3}

For every positive integer $d$ by \hyperref[pd]{$\mathrm{P}(d)$} \label{pd} we shall denote the following statement.
\medskip

\noindent \textit{For every integer $n\meg 2d$, every $0<\eta<1$, every $\vartheta>0$,
every nonempty finite set $\mathcal{X}$ and every nonempty subset $\mathcal{S}$ of\, $\mathcal{X}$,
if $\ell\coloneqq \binom{\lfloor n/2 \rfloor}{d}$, $\boldsymbol{\gamma}=(\gamma_k(\eta,\vartheta,d,n))_{k=1}^\ell$
is as in Subsection \emph{\ref{subsec4.2}} and $\bbx$ is an $\mathcal{X}$-valued, $\eta$-spreadable,
$d$-dimensional random array on $[n]$ that is $(\vartheta,\mathcal{S})$-box independent,
then $\bbx$ is $(\boldsymbol{\gamma},\mathcal{S})$-independent. }
\medskip

\noindent By Lemma \ref{l4.4}, it is clear that Theorem \ref{t3.2} follows from the validity of
\hyperref[pd]{$\mathrm{P}(d)$} for every positive integer $d$.

\subsection{The base case ``$d=1$"} \label{subsec4.4}

The initial step of the induction follows from the following lemma.
\begin{lem} \label{l4.5}
Let $n, \eta, \vartheta, \mathcal{X}$ and $ \mathcal{S}$ be as in the statement of\, \hyperref[pd]{$\mathrm{P}(1)$},
and assume that $\bbx=(X_1,\dots,X_n)$ is an $\mathcal{X}\text{-valued}$, $\eta$-spreadable, random vector.
Assume, moreover, that $\bbx$ is $(\vartheta,\mathcal{S})$-box independent, that is,
for every $i,j\in [n]$ with $i\neq j$ and every $a\in\mathcal{S}$~we~have
\begin{equation} \label{e4.30}
\mathbb{P}\big( [X_i=a]\cap [X_j=a]\big) \mik \mathbb{P}\big([X_i=a]\big)\,\mathbb{P}\big([X_j=a]\big) + \vartheta.
\end{equation}
Then $\bbx$ is $(\boldsymbol{\gamma},\mathcal{S})$-independent, that is,
for every nonempty $\mathcal{F}\subseteq [n]$ with $|\mathcal{F}|\mik n/2$ and every collection
$(a_i)_{i\in\mathcal{F}}$ of elements of $\mathcal{S}$, we have
\begin{equation} \label{e4.31}
\bigg| \mathbb{P}\Big(\bigcap_{i\in\mathcal{F}}[X_i = a_i] \Big) -
\prod_{i\in\mathcal{F}} \mathbb{P}\big([X_i = a_i]\big) \bigg| \mik \gamma_{|\mathcal{F}|}(\eta,\vartheta,1,n),
\end{equation}
where $\boldsymbol{\gamma}\coloneqq(\gamma_k(\eta,\vartheta,1,n))_{k=1}^{\lfloor n/2 \rfloor}$
is as in \eqref{e4.9}. In particular, \hyperref[pd]{$\mathrm{P}(1)$} holds true.
\end{lem}
\begin{proof}
Observe that, by the $\eta$-spreadability of $\bbx$, it is enough to show that for every $k\in \{1,\dots,\lfloor n/2 \rfloor\}$
and every $a_1,\dots,a_k\in \mathcal{S}$ we have
\begin{equation} \label{e4.32}
\bigg| \mathbb{P}\Big(\bigcap_{i=1}^k[X_i =a_i]\Big) - \prod_{i=1}^k\mathbb{P}\big([X_i =a_i]\big) \bigg| \mik
(k-1) \bigg( 2\eta + \sqrt{\frac{1}{\lfloor n/2 \rfloor} + \vartheta} \bigg).
\end{equation}
To this end, we proceed by induction of $k$. The case ``$k=1$" is straightforward. Let $k$ be a positive integer
with $k<\lfloor n/2\rfloor$, and assume that \eqref{e4.32} has been verified up to $k$. Fix $a_1,\dots,a_{k+1}\in \mathcal{S}$.
Set $m\coloneqq \lfloor n/2\rfloor$ and $E\coloneqq \bigcap_{i=1}^k [X_i = a_i]$. Also set $A_j\coloneqq [X_{k+j} = a_{k+1}]$
for every $j\in [m]$. Using the $\eta$-spreadability of $\bbx$, for every $j,j'\in [m]$ with $j\neq j'$ we~have
\begin{enumerate}
\item[(i)] $| \mathbb{P}(A_j) - \mathbb{P}(A_{j'}) |\mik \eta$, and
\item[(ii)] $| \mathbb{P}(E \cap A_j) - \mathbb{P}(E\cap A_{j'}) |\mik \eta$.
\end{enumerate}
Moreover, since $a_{k+1}\in\mathcal{S}$, we have
\begin{enumerate}
\item[(iii)] $\mathbb{P}(A_j \cap A_{j'}) \mik \mathbb{P}(A_j)\, \mathbb{P}(A_{j'}) +\vartheta$.
\end{enumerate}
Applying Lemma \ref{l4.2} for ``$\delta = \vartheta$" and using the definition of $A_1$, we see that
\begin{equation} \label{e4.33}
\Big| \mathbb{P}\big(E \cap [X_{k+1}=a_{k+1}]\big) - \mathbb{P}(E)\, \mathbb{P}\big([X_{k+1}=a_{k+1}]\big) \Big|
\mik 2\eta + \sqrt{\frac{1}{m} + \vartheta}.
\end{equation}
On the other hand, by our inductive assumptions, we have
\begin{equation} \label{e4.34}
\bigg|\mathbb{P}(E) - \prod_{j=1}^{k}\mathbb{P}\big([X_j=a_j]\big) \bigg|\mik
(k-1) \bigg( 2\eta + \sqrt{\frac{1}{m} + \vartheta} \bigg).
\end{equation}
Combining \eqref{e4.33} and \eqref{e4.34}, we see that \eqref{e4.32} is satisfied, as desired.
\end{proof}

\subsection{The general inductive step} \label{subsec4.5}

We now enter into the main part of the proof of Theorem \ref{t3.2}. \textit{Specifically,
fix an integer $d\meg 2$. Throughout this subsection, we will assume that\,
\hyperref[pd]{$\mathrm{P}(d-1)$} has been proved}.

We also note that, in what follows, we will estimate the difference of various products in terms
of the differences of the factors, the number of factors and the $L_{\infty}$ norm of the factors.
The reader should have in mind this remark, as we will use this standard telescoping argument
without further notice.

\subsubsection{Step 1: preparatory lemmas} \label{subsubsec4.5.1}

Our goal in this step is to prove two probabilistic lemmas that will be used in the third
and the fourth step of the proof respectively. Strictly speaking, these lemmas are not part
of the proof of \hyperref[pd]{$\mathrm{P}(d)$} since in their proofs we do not use the inductive
assumptions. (In particular, this subsection can be read independently.)

The first lemma essentially shows that the reverse inequality of \eqref{e3.3} always holds true
in the presence of approximate spreadability.
\begin{lem} \label{l4.6}
Let $n$ be an integer with $n\meg 2d$, let $0<\eta<1$, let $\mathcal{X}$ be a nonempty finite set, and let
$\bbx=\langle X_s : s\in \binom{[n]}{d}\rangle$ be an $\mathcal{X}$-valued, $\eta$-spreadable,
$d$-dimensional random array on $[n]$. Then for every $t\in \binom{[n-2]}{d-1}$ and every $a\in\mathcal{X}$ we have
\begin{align}
\label{e4.35} \prob\big( [X_{t\cup\{n-1\}} = a]\big) & \, \prob\big([X_{t\cup\{n\}} = a]\big) \\
& \mik \prob\big( [X_{t\cup\{n-1\}} = a] \cap [X_{t\cup\{n\}} = a] \big) + \frac{1}{n-d+1} + 6 \eta. \nonumber
\end{align}
\end{lem}
\begin{proof}
Set $t_0\coloneqq [d-1]$, and $A_i\coloneqq [X_{t_0 \cup \{d-1+i\}} = a]$ for every $i\in [n-d+1]$.
Observe that the sequence $(A_1\dots,A_{n-d+1})$ is $\eta$-spreadable\footnote{Recall that this means
that the random vector $(\mathbf{1}_{A_1},\dots, \mathbf{1}_{A_{n-d+1}})$ is $\eta$-spreadable.}.
By Lemma~\ref{l4.3}, we obtain that
\begin{equation} \label{e4.36}
\prob(A_1)\,\prob(A_2) \mik \prob(A_1\cap A_2) +\frac{1}{n-d+1} +3\eta.
\end{equation}
By \eqref{e4.36} and the $\eta$-spreadability of $\bbx$, the estimate \eqref{e4.35} follows.
\end{proof}
The second lemma shows that the box independence condition \eqref{e3.3} is inherited by the two-dimensional faces
of $d$-dimensional boxes.
\begin{lem} \label{l4.7}
Let $n$ be an integer with $n\meg 2d$, let $0<\eta<1$, let $\vartheta>0$, let $\mathcal{X}$ be a nonempty finite set,
let $\mathcal{S}$ be a nonempty subset of $\mathcal{X}$, and let $\bbx=\langle X_s : s\in \binom{[n]}{d}\rangle$
be an $\mathcal{X}$-valued, $\eta$-spreadable, $d$-dimensional random array on $[n]$ that is
$(\vartheta,\mathcal{S})\text{-box}$ independent. Then for every $t\in \binom{[n-2]}{d-1}$ and
every $a \in \mathcal{S}$ we have
\begin{align}
\label{e4.37} \prob\big( [X_{t \cup \{n-1\}} = a] & \cap [X_{t \cup \{n\}} = a] \big)  \\
& \mik \prob\big( [X_{t \cup \{n-1\}} = a] \big) \, \prob \big( [X_{t \cup \{n\}} = a] \big)
+ \vartheta_3(\eta,\vartheta,d,n), \nonumber
\end{align}
where $\vartheta_3(\eta,\vartheta,d,n)=\frac{d-1}{(n-2d+2)^{1/2^{d-1}}} +
(2^d + 5) \eta^{1/2^{d-1}} + \vartheta^{1/2^{d-1}} + 3 \eta$ is as in \eqref{e4.12}.
\end{lem}
\begin{proof}
Fix $a \in \mathcal{S}$. We define, recursively, a finite sequence $(\vartheta_i)_{i=0}^{d-1}$
by setting $\vartheta_0 = \vartheta$ and
\begin{equation} \label{e4.38}
\vartheta_{r+1} = \bigg( \frac{1}{n-2d+r+2} + (2^{d-r} + 5)\eta + \vartheta_r \bigg)^{1/2}.
\end{equation}
By induction on $r\in\{0,\dots,d-1\}$, we will show that
\begin{equation} \label{e4.39}
\mathbb{P}\Big( \bigcap_{v\in B_r} [X_{t_r\cup v} = a] \Big) \mik \prod_{v\in B_r}
 \mathbb{P}\big(  [X_{t_r\cup v} = a] \big) +\vartheta_r,
\end{equation}
where $t_{r}\coloneqq [r]$ (and, by convention, $[0] = \emptyset$) and
$B_r\coloneqq \mathrm{Box}\big((H_{r+1},\dots,H_d)\big)$ is the $(d-r)$-dimensional box---see \eqref{e3.1}---determined
by the sequence $(H_{r+1},\dots,H_d)$ with $H_i\coloneqq \{n-2d+2i-1, n-2d+2i\}$ for each $i\in\{r+1,\dots,d\}$.
The case ``$r=0$" follows from the fact that the random array $\bbx$ is $(\vartheta,\mathcal{S})$-box independent.
Next, let $r \in \{0,\dots,d-2\}$ and assume that \eqref{e4.39} has been proved up to $r$. For every $j\in [n-2d+r+2]$ set
\begin{equation} \label{e4.40}
A_j\coloneqq \bigcap_{v \in B_{r+1}} [X_{t_r \cup \{r+j\} \cup v} = a].
\end{equation}
Since $\bbx$ is $\eta$-spreadable, the sequence $(A_1,\dots, A_{n-2d+r+2})$ is $\eta$-spreadable.
Using this observation and the inductive assumptions, we see that
\begin{align}
\label{e4.41} \prob(A_1 \cap A_2)
& \mik \prob( A_{n-2d+r+1} \cap A_{n-2d+r+2} ) + \eta\\
& = \prob\Big( \bigcap_{v\in B_r} [X_{t_r\cup v} = a]  \Big) +\eta
\mik \prod_{v\in B_r} \prob\big(  [X_{t_r\cup v} = a] \big)  + \eta + \vartheta_r. \nonumber
\end{align}
On the other hand, since $\bbx$ is $\eta$-spreadable, we have
\begin{equation} \label{e4.42}
\prod_{v\in B_r} \prob\big( [X_{t_r \cup v} = a] \big) \mik
\bigg( \prod_{v\in B_{r+1}} \prob\big( [X_{t_{r+1} \cup v} = a] \big) \bigg)^2 + 2^{d-r}\eta.
\end{equation}
Moreover, by Lemma \ref{l4.3} applied to the $\eta$-spreadable sequence $(A_j)_{j=1}^{n-2d+r+2}$,
\begin{equation} \label{e4.43}
\prob(A_1 \cap A_2) \meg \prob(A_1) \, \prob(A_2) - \frac{1}{n-2d+r+2} - 3\eta.
\end{equation}
By \eqref{e4.41}--\eqref{e4.43} and using the $\eta$-spreadability of the sequence
$(A_j)_{j=1}^{n-2d+r+2}$ once again, we obtain that
\begin{align}
\label{e4.44} \prob\Big( \bigcap_{v \in B_{r+1}} [X_{t_{r+1} \cup v} = a] \Big)^2 & =
\prob(A_1)^2 \mik \prob(A_1) \, \prob(A_2) + \eta \\
& \mik \bigg( \prod_{v\in B_{r+1}} \prob\big( [X_{t_{r+1} \cup v} = a] \big) \bigg)^2 + \vartheta_{r+1}^2. \nonumber
\end{align}
Taking square-roots, this estimate completes the inductive proof of \eqref{e4.39}.

Now notice that
\begin{align}
\label{e4.45}
\vartheta_{d-1} & \mik \vartheta^{\frac{1}{2^{d-1}}} + \sum_{j=1}^{d-1} \Big((n-2d+2)^{-1/2^j}
 + ((2^d + 5)\eta)^{1/2^j}\Big)  \\
& \mik \frac{d-1}{(n-2d+2)^{1/2^{d-1}}}  + (d-1) (2^d + 5)^{1/2} \eta^{1/2^{d-1}} +
\vartheta^{1/2^{d-1}} \nonumber \\
& \mik \frac{d-1}{(n-2d+2)^{1/2^{d-1}}} + (2^d + 5) \eta^{1/2^{d-1}} + \vartheta^{1/2^{d-1}}. \nonumber
\end{align}
Setting $s_1\coloneqq [d-1] \cup \{n-1\}$ and $s_2\coloneqq [d-1] \cup \{n\}$,
by \eqref{e4.39} and \eqref{e4.45}, we have
\begin{align}
\label{e4.46} \prob\big( [X_{s_1} = a] \cap [X_{s_2} = a] \big) & \mik
  \prob\big( [X_{s_1} = a] \big) \, \prob\big( [X_{s_2} = a] \big) + \\
&  \, + \frac{d-1}{(n-2d+2)^{1/2^{d-1}}}  + (2^d + 5) \eta^{1/2^{d-1}}\!\! + \vartheta^{1/2^{d-1}}. \nonumber
\end{align}
Taking into account the $\eta$-spreadability of $\bbx$ and the definition of $\vartheta_3(\eta,\vartheta,d,n)$,
the estimate \eqref{e4.37} follows from \eqref{e4.46}.
\end{proof}

\subsubsection{Step 2: rewriting the inductive assumptions} \label{subsubsec4.5.2}

We proceed with the following lemma, which will enable us to use \hyperref[pd]{$\mathrm{P}(d-1)$}
in a more convenient form.
\begin{lem} \label{l4.8}
Let $n, \eta,\vartheta, \mathcal{X}, \mathcal{S}$ be as in \hyperref[pd]{$\mathrm{P}(d)$},
and let $\bbx=\langle X_s : s\in \binom{[n]}{d} \rangle$ be an $\mathcal{X}\text{-valued}$, $\eta$-spreadable,
$d$-dimensional random array on~$[n]$ that is $(\vartheta,\mathcal{S})$-box independent.
We define $\widetilde{\bbx}=\langle\widetilde{X}_t : t\in \binom{[n-1]}{d-1} \rangle$ by setting
\begin{equation} \label{e4.47}
\widetilde{X}_t\coloneqq X_{t\cup\{n\}}.
\end{equation}
Then the random array $\widetilde{\bbx}$ is $\mathcal{X}$-valued, $\eta$-spreadable
and $( \vartheta_1(\eta,\vartheta,d,n),\mathcal{S}) $-box independent, where
$\vartheta_1(\eta,\vartheta,d,n)=(n-2d+2)^{-1/2} + (2^d+5)\sqrt{\eta} + \sqrt{\vartheta}$ is as \eqref{e4.10}.
\end{lem}
\begin{proof}
Since $\bbx$ is $\mathcal{X}$-valued and $\eta$-spreadable, by \eqref{e4.47},
we see that these properties are inherited to $\widetilde{\bbx}$.
Thus, we only need to check that $\widetilde{\bbx}$  is $(\vartheta_1(\eta,\vartheta,d,n),\mathcal{S})$-box independent.

To this end, fix $a\in \mathcal{S}$ and a finite sequence $\mathcal{H}=(H_1,\dots,H_{d-1})$
of $2\text{-element}$ subsets of $[n-1]$ with $\max(H_i)<\min(H_{i+1})$ for all $i\in[d-2]$;
let $B\coloneqq \mathrm{Box}(\mathcal{H})$ denote the $(d-1)$-dimensional box determined
by the sequence $\mathcal{H}$. Moreover, set
\[ B_0\coloneqq \mathrm{Box}\big((\{1,2\},\dots,\{2d-3,2d-2\})\big) \]
and $A_r\coloneqq \bigcap_{t\in B_0}[X_{t\cup\{2d-2+r\}}=a]$ for every $r\in[n-2d+2]$.
Notice that the sequence $(A_1,\dots, A_{n-2d+2})$ is $\eta$-spreadable. Therefore, by Lemma \ref{l4.3},
\begin{equation} \label{e4.48}
\mathbb{P}(A_1)^2 \mik \mathbb{P}(A_1) \, \mathbb{P}(A_2) + \eta
\mik \mathbb{P}(A_1\cap A_2) + \frac{1}{n-2d+2} + 4\eta.
\end{equation}
Next, set $B'\coloneqq \mathrm{Box}\big((\{1,2\},\dots,\{2d-3,2d-2\},\{2d-1,2d\})\big)$,
and observe that $B'$ is a $d\text{-dimensional}$ box and $A_1\cap A_2 =\bigcap_{s \in B'}[X_s = a]$.
Since $\bbx$ is $(\vartheta,\mathcal{S})$-box independent and $a\in\mathcal{S}$, we see that
\begin{align}
\label{e4.49} \mathbb{P}(A_1 \cap A_2) & \mik \prod_{s\in B'} \prob\big([X_s =a]\big) +\vartheta \\
& = \Big( \prod_{t\in B_0} \prob\big([X_{t \cup \{2d-1\}} =a]\big) \Big)
\Big( \prod_{t\in B_0} \prob\big([X_{t \cup \{2d\}} = a]\big) \Big) +\vartheta \nonumber \\
& \mik \bigg( \prod_{t\in B_0} \prob\big([X_{t \cup \{2d-1\}} = a ]\big) \bigg) ^2 + 2^{d-1}\eta +\vartheta, \nonumber
\end{align}
where the last inequality follows from the $\eta$-spreadability of $\bbx$.
By \eqref{e4.48}, \eqref{e4.49} and the definition of $A_1$, we obtain
\begin{align}
\label{e4.50} \prob\Big( \bigcap_{t\in B_0} [X_{t\cup\{2d-1\}} = a] \Big)
& \mik \prod_{t\in B_0} \prob\big([X_{t \cup \{2d-1\}} = a]\big) + \\
& \;\;\;\;\;\;\;\;\;\;\; + \bigg(\frac{1}{n-2d+2} + (2^{d-1}+4)\eta+\vartheta\bigg)^{1/2} \nonumber \\
& \mik \prod_{t\in B_0} \prob\big([X_{t \cup \{2d-1\}} = a]\big) + \nonumber \\
& \;\;\;\;\;\;\;\;\;\;\;  +(n-2d+2)^{-1/2} + (2^{d-1}+4)\sqrt{\eta}+ \sqrt{\vartheta}. \nonumber
\end{align}
On the other hand, using the $\eta$-spreadability of $\bbx$, we have
\begin{align}
\label{e4.51} & \bigg| \prob\Big( \bigcap_{t\in B_0} [X_{t\cup\{2d-1\}} = a]\Big)-
\prob\Big( \bigcap_{t\in B} [\widetilde{X}_t = a] \Big) \bigg| \mik \eta \mik \sqrt{\eta}, \\
\label{e4.52} & \bigg| \prod_{t\in B_0} \prob\big([X_{t \cup \{2d-1\}} = a]\big) -
\prod_{t\in B} \prob\big([\widetilde{X}_t = a]\big) \bigg| \mik 2^{d-1}\eta \mik 2^{d-1} \sqrt{\eta}.
\end{align}
Combining \eqref{e4.50}--\eqref{e4.52} and invoking the definition
of $\vartheta_1(\eta,\vartheta,d,n)$ in \eqref{e4.10}, we conclude that
\begin{equation} \label{e4.53}
\mathbb{P}\Big( \bigcap_{t\in B} [\tilde{X}_t = \alpha] \Big) \mik \prod_{t\in B} \prob\big([\tilde{X}_t = a]\big)
 + \vartheta_1(\eta,\vartheta,d,n).
\end{equation}
Since $a$ and $B$ were arbitrary, the result follows.
\end{proof}
By Lemma \ref{l4.8} and \hyperref[pd]{$\mathrm{P}(d-1)$}, we have the following corollary.
\begin{cor} \label{c4.9}
Let $n, \eta, \vartheta, \mathcal{X}, \mathcal{S},\bbx$ be as in Lemma \emph{\ref{e4.7}}.
Then for every nonempty subset $\mathcal{G}$ of\, $\binom{[n-1]}{d-1}$
with $|\bigcup \mathcal{G}| \mik (n-1)/2$, every collection $(a_t)_{t\in\mathcal{G}}$ of elements of~$\mathcal{S}$,
and every $r\in [n]$ with $r>\max\big(\bigcup\mathcal{G}\big)$ we have
\begin{equation} \label{e4.54}
\bigg| \mathbb{P}\Big( \bigcap_{t\in\mathcal{G}}[X_{t\cup\{r\}} = a_t]\Big)-
\prod_{t\in\mathcal{G}} \prob\big( [X_{t\cup\{r\}} = a_t]\big) \bigg| \mik
\gamma^{(1)}_{|\mathcal{G}|}(\eta, \vartheta, d, n),
\end{equation}
where $\gamma^{(1)}_{|\mathcal{G}|}(\eta, \vartheta, d, n)=
\gamma_{|\mathcal{G}|}\big(\eta, \vartheta_1(\eta,\vartheta,d,n), d-1, n-1\big) +
(|\mathcal{G}|+1)\eta$ is as in \eqref{e4.13}.
\end{cor}

\subsubsection{Step 3: doubling} \label{subsubsec4.5.3}

The following lemma complements Lemma \ref{l4.8}. It is also based on the inductive hypothesis
\hyperref[pd]{$\mathrm{P}(d-1)$}, but it will enable to use it in a rather different form.
\begin{lem}[Doubling] \label{l4.10}
Let $n, \eta,\vartheta, \mathcal{X}, \mathcal{S}$ be as in the statement of\, \hyperref[pd]{$\mathrm{P}(d)$},
and assume that $\bbx=\langle X_s : s\in \binom{[n]}{d} \rangle$ is an $\mathcal{X}$-valued, $\eta$-spreadable,
$d$-dimensional random array on~$[n]$ that is $(\vartheta,\mathcal{S})$-box independent.
We define a $(d-1)$-dimensional random array
$\widetilde{\bbx}'=\langle \widetilde{X}'_t : t\in \binom{[n-2]}{d-1} \rangle$ by setting
\begin{equation} \label{e4.55}
\widetilde{X}'_t \coloneqq (X_{t \cup \{n-1\}}, X_{t \cup \{n\}} ).
\end{equation}
Then $\widetilde{\bbx}'$ is $(\mathcal{X}\times\mathcal{X})$-valued, $\eta$-spreadable and
$(\vartheta_2(\eta,\vartheta,d,n), \{(a,a): a\in\mathcal{S}\})$-box independent,
where $\vartheta_2(\eta,\vartheta,d,n)=\frac{2^{d-1}}{n-d+1}+2^d3\eta+\vartheta$ is as in \eqref{e4.11}.
\end{lem}
\begin{proof}
It is clear that $\widetilde{\bbx}'$ is $(\mathcal{X}\times\mathcal{X})$-valued and $\eta$-spreadable. So, we only need
to show that $\widetilde{\bbx}'$ is  $(\vartheta_2(\eta,\vartheta,d,n), \{(a,a): a\in\mathcal{S}\})$-box independent.

Let $H_1,\dots,H_{d-1}$ be $2$-element subsets of $[n-2]$ with $\max(H_i)<\min(H_{i+1})$ for all $i\in[d-2]$.
Also let $a\in\mathcal{S}$. Set $\widetilde{B}\coloneqq \mathrm{Box}\big((H_1,\dots,H_{d-1})\big)$;
also set $H_d\coloneqq\{n-1,n\}$ and $B\coloneqq \mathrm{Box}\big((H_1,\dots,H_{d-1},H_d)\big)$. Since $\bbx$
is $(\vartheta,\mathcal{S})$-box independent, we see that
\begin{align}
\label{e4.56} \prob\Big( \bigcap_{t\in \widetilde{B}} [\widetilde{X}'_t = (a,a)] \Big)
 & = \prob\Big( \bigcap_{t\in \widetilde{B}} \big( [X_{t\cup\{n-1\}}=a] \cap [X_{t\cup\{n\}}=a] \big) \Big) \\
 & = \prob\Big( \bigcap_{s\in B} [X_s=a]\Big) \mik \prod_{s\in B} \prob\big( [X_s=a] \big) +\vartheta. \nonumber
\end{align}
By Lemma \ref{l4.6}, we have
\begin{align}
\label{e4.57} \prod_{s\in B} \prob\big( [X_s&=a]\big) = \prod_{t\in \widetilde{B}}
\prob\big( [X_{t \cup \{n-1\}}=a]\big) \, \prob\big( [X_{t \cup \{n\}}=a]\big) \\
& \mik \prod_{t\in \widetilde{B}} \prob\big( [X_{t \cup \{n-1\}}=a] \cap [X_{t \cup \{n\}}=a]\big)
    +\frac{2^{d-1}}{n-d+1}+2^{d-1}6\eta \nonumber \\
& = \prod_{t\in \widetilde{B}} \prob\big( [\widetilde{X}'_t = (a,a)]\big) +\frac{2^{d-1}}{n-d+1}+2^{d}3\eta. \nonumber
\end{align}
By \eqref{e4.56} and \eqref{e4.57} and the definition of $\vartheta_2(\eta,\vartheta,d,n)$, the result follows.
\end{proof}
The following corollary---which is an immediate consequence of Lemma \ref{l4.10} and the inductive
assumption \hyperref[pd]{$\mathrm{P}(d-1)$}---is the analogue of Corollary \ref{c4.9}.
\begin{cor} \label{c4.11}
Let $n, \eta, \vartheta, \mathcal{X}, \mathcal{S},\bbx, \widetilde{\bbx}'$ be as in Lemma \emph{\ref{e4.9}}.
Then the random array $\widetilde{\bbx}'$ is
$\big((\gamma^{(2)}_k\!(\eta, \vartheta, d, n))_{k=1}^\ell, \{(a,a): a\in\mathcal{S}\}\big)$-independent,
where $\ell= \binom{\lfloor (n-2)/2\rfloor}{d-1}$ and
$\gamma^{(2)}_k(\eta, \vartheta, d, n)=\gamma_k\big(\eta, \vartheta_2(\eta,\vartheta,d,n), d-1, n-2\big)$
is as in \eqref{e4.14} for each $k\in [\ell]$.
\end{cor}

\subsubsection{Step 4: gluing} \label{subsubsec4.5.4}

This is the main step of the proof. Specifically, our goal is to prove the following proposition.
\begin{prop}[Gluing] \label{p4.12}
Let $n\meg 2d+2$ be an integer, let $\eta,\vartheta,\mathcal{X},\mathcal{S}$ be as in the statement
of\, \hyperref[pd]{$\mathrm{P}(d)$}, and assume that $\bbx=\langle X_s : s\in \binom{[n]}{d}\rangle$
is an $\mathcal{X}$-valued, $\eta$-spreadable, $d$-dimensional random array on $[n]$ that is
$(\vartheta,\mathcal{S})$-box independent. Finally, let $r$ be an integer with $d< r\mik n/2$,
let $\mathcal{G}$ be a nonempty subset of $\binom{[r-1]}{d-1}$, let $(a_t)_{t\in\mathcal{G}}$
be a collection of elements of\, $\mathcal{S}$, let $\mathcal{F}$ be a nonempty subset of $\binom{[r-1]}{d}$,
and let $(b_s)_{s\in\mathcal{F}}$ be a collection of elements of\, $\mathcal{S}$. Then we have
\begin{align} \label{e4.58}
\bigg| \prob\Big( \bigcap_{s\in\mathcal{F}}& [X_s = b_s] \cap \bigcap_{t\in\mathcal{G}} [X_{t \cup \{r\}}= a_t] \Big) -  \\
 & - \prob\Big( \bigcap_{s\in\mathcal{F}} [X_s = b_s]\Big)\, \prob\Big( \bigcap_{t\in\mathcal{G}}
[X_{t \cup \{r\}} = a_t] \Big) \bigg| \mik \gamma^{(4)}_{|\mathcal{G}|}(\eta,\vartheta,d,n), \nonumber
\end{align}
where $\gamma_{|\mathcal{G}|}^{(4)}(\eta, \vartheta, d,n)=
\big( \gamma^{(3)}_{|\mathcal{G}|}(\eta,\vartheta,d,n) +
\lfloor n/2 \rfloor^{-1} +(2|\mathcal{G}|+1)\eta \big)^{1/2} +2\eta$ is as in \eqref{e4.16}.
\end{prop}
Proposition \ref{p4.12} follows by carefully selecting a sequence of events,
and then applying the averaging argument presented in Lemma \ref{l4.2}.
In order to do so, we need to control the variances of the corresponding averages.
This is the content of the following lemma.
\begin{lem}[Variance estimate] \label{l4.13}
Let $n,\eta,\vartheta,\mathcal{X},\mathcal{S}$ be as in the statement of\, \hyperref[pd]{$\mathrm{P}(d)$},
and assume that $\bbx=\langle X_s : s\in \binom{[n]}{d}\rangle$ is an $\mathcal{X}$-valued,
$\eta$-spreadable, $d$-dimensional random array on $[n]$ that is $(\vartheta,\mathcal{S})$-box independent.
Then for every nonempty subset $\mathcal{G}$ of\, $\binom{[n-2]}{d-1}$ with $|\bigcup \mathcal{G}|\mik (n-2)/2$,
and every collection $(a_t)_{t\in\mathcal{G}}$ of elements of\, $\mathcal{S}$ we have
\begin{align}
\label{e4.59} \prob\Big( \bigcap_{t\in\mathcal{G}}  [& X_{t \cup \{n-1\}}=a_t] \cap \bigcap_{t\in\mathcal{G}}
[X_{t \cup \{n\}}=a_t] \Big) \\
\mik & \, \prob\Big( \bigcap_{t\in\mathcal{G}} [X_{t \cup \{n-1\}}=a_t]\Big)\,
\prob\Big( \bigcap_{t\in\mathcal{G}} [X_{t \cup \{n\}}=a_t] \Big) +
\gamma^{(3)}_{|\mathcal{G}|}(\eta, \vartheta,d,n), \nonumber
\end{align}
where $\gamma^{(3)}_{|\mathcal{G}|}(\eta, \vartheta, d,n)=
2\gamma^{(1)}_{|\mathcal{G}|}(\eta, \vartheta,d,n)+ \gamma^{(2)}_{|\mathcal{G}|}(\eta, \vartheta,d,n) +
|\mathcal{G}| \,\vartheta_3(\eta, \vartheta,d,n)$ is as in \eqref{e4.15}.
\end{lem}
\begin{proof}
Let $\mathcal{G}$ be a subset of $\binom{[n-2]}{d-1}$ with $|\bigcup \mathcal{G}|\mik (n-2)/2$,
and let $(a_t)_{t\in\mathcal{G}}$ be a collection of elements of $\mathcal{S}$. By Corollary \ref{c4.11}, we have
\begin{align}
\label{e4.60} \prob\Big( \bigcap_{t\in\mathcal{G}} [X_{t \cup \{n-1\}} & =a_t] \cap
\bigcap_{t\in\mathcal{G}} [X_{t \cup \{n\}}=a_t] \Big)  \\
 \mik & \prod_{t\in\mathcal{G}} \prob\big(  [X_{t \cup \{n-1\}}=a_t] \cap [X_{t \cup \{n\}}=a_t] \big)
   + \gamma^{(2)}_{|\mathcal{G}|}(\eta, \vartheta, d, n). \nonumber
\end{align}
Moreover, by Lemma \ref{l4.7},
\begin{align}
\label{e4.61} \prod_{t\in\mathcal{G}} & \, \prob\big( [X_{t \cup \{n-1\}}=a_t] \cap [X_{t \cup \{n\}}=a_t] \big) \\
& \mik \prod_{t\in\mathcal{G}} \prob\big( [X_{t \cup \{n-1\}}=a_t] \big)\, \prob\big( [X_{t \cup \{n\}}=a_t] \big)
  +\, |\mathcal{G}| \, \vartheta_3(\eta,\vartheta,d,n). \nonumber
\end{align}
Finally, by Corollary \ref{c4.9}, we see that
\begin{align}
\label{e4.62} \prod_{t\in\mathcal{G}} \prob\big( [X_{t \cup \{n-1\}}=a_t] \big) & \mik
\prob\Big( \bigcap_{t\in\mathcal{G}} [X_{t \cup \{n-1\}}=a_t] \Big) +
\gamma^{(1)}_{|\mathcal{G}|}(\eta, \vartheta, d, n), \\
\label{e4.63} \prod_{t\in\mathcal{G}} \prob \big( [X_{t \cup \{n\}}=a_t] \big) & \mik
\prob\Big( \bigcap_{t\in\mathcal{G}} [X_{t \cup \{n\}}=a_t] \Big) + \gamma^{(1)}_{|\mathcal{G}|}(\eta, \vartheta, d, n).
\end{align}
The estimate \eqref{e4.59} follows by combining \eqref{e4.60}--\eqref{e4.63} and invoking the definition
of the constant $\gamma^{(3)}_{|\mathcal{G}|}(\eta, \vartheta, d,n)$ in \eqref{e4.15}.
\end{proof}
We are now ready to give the proof of Proposition \ref{p4.12}.
\begin{proof}[Proof of Proposition \emph{\ref{p4.12}}]
Set $E\coloneqq\bigcap_{s\in\mathcal{F}}[X_s = b_s]$ and $A_i\coloneqq \bigcap_{t\in\mathcal{G}} [X_{t \cup \{r-1+i\}}]$
for every $i\in\{1,\dots,\lfloor n/2 \rfloor\}$. Since $\bbx$ is $\eta$-spreadable,
for every $i,j\in \{1,\dots,\lfloor n/2 \rfloor\}$ with $i\neq j$ we~have
\begin{enumerate}
\item[(i)] $|\mathbb{P}(A_i) - \mathbb{P}(A_j)| \mik \eta$, and
\item[(ii)] $|\mathbb{P}(E\cap A_i) - \mathbb{P}(E\cap A_j)| \mik \eta$.
\end{enumerate}
Moreover, applying Lemma \ref{l4.13} and using the $\eta$-spreadability of $\bbx$ again,
for every $i,j\in \{1,\dots,\lfloor n/2 \rfloor\}$ with $i\neq j$ we have
\begin{enumerate}
\item[(iii)] $\mathbb{P}(A_i \cap A_j) \mik \mathbb{P}(A_i) \, \mathbb{P}(A_j)
 +\gamma^{(3)}_{|\mathcal{G}|}(\eta,\vartheta,d,n) + (2|\mathcal{G}|+1)\eta$.
\end{enumerate}
By Lemma \ref{l4.2} applied for ``$\delta = \gamma^{(3)}_{|\mathcal{G}|}(\eta,\vartheta,d,n) + (2|\mathcal{G}|+1)\eta$"
and taking into account the definition of the constant $\gamma^{(4)}_{|\mathcal{G}|}(\eta, \vartheta, d,n)$,
we conclude that \eqref{e4.58} is satisfied.
\end{proof}

\subsubsection{Step 5: completion of the proof} \label{subsubsec4.5.5}

This is the last step of the proof. Recall that we need to prove that the statement
\hyperref[pd]{$\mathrm{P}(d)$} holds true, or equivalently, that the estimate~\eqref{e3.4} is~satisfied
for the sequence $\boldsymbol{\gamma}=(\gamma_k(\eta,\vartheta,d,n))_{k=1}^\ell$ defined in
Subsection \ref{subsec4.2}. As expected, the verification of this estimate will be reduced
to Proposition \ref{p4.12}. To this end, we will decompose an arbitrary nonempty subset
$\mathcal{F}$ of $\binom{[n]}{d}$ into several components that are easier to handle.
The details of this decomposition are presented in the following definition.
\begin{defn}[Slicing profile] \label{d4.14}
Let $n,d$ be positive integers with $n\meg d$ and let $\mathcal{F}$ be a nonempty subset
of\, $\binom{[n]}{d}$. Note that there exist, unique,
\begin{enumerate}
\item[$\bullet$] $u\in [n]$,
\item[$\bullet$] $r_1,\dots,r_u\in [n]$ with $d\mik r_1<\dots< r_u$, and
\item[$\bullet$] for every $i\in [u]$ a nonempty subset\, $\mathcal{G}_i$ of\, $\binom{[r_i-1]}{d-1}$,
\end{enumerate}
such that
\begin{equation} \label{e4.64}
\mathcal{F} = \big\{t\cup\{r_i\} : i\in [u], t\in \mathcal{G}_i\big\};
\end{equation}
indeed, set $M\coloneqq \{\max(s):s\in \mathcal{F}\}$ and $u\coloneqq |M|$, let $\{r_1<\dots<r_u\}$
denote the increasing enumeration of $M$, and set
$\mathcal{G}_i\coloneqq \big\{t\in\binom{[r_i - 1]}{d-1}: t\cup\{r_i\}\in\mathcal{F}\big\}$ for every $i\in [u]$.
We refer to the triple $(u, (r_i)_{i=1}^u, (\mathcal{G}_i)_{i=1}^u)$ as the \emph{slicing} of $\mathcal{F}$,
and to the sequence $(|\mathcal{G}_i|)_{i=1}^u$ as the \emph{slicing profile} of $\mathcal{F}$.
Finally, we denote by $\mathrm{SP}(n)$ the set of all nonempty finite sequences $(k_i)_{i=1}^u$
that are the slicing profile of some nonempty subset $\mathcal{F}$ of\, $\binom{[n]}{d}$; notice that
\[ \mathrm{SP}(n)= \Bigg\{ (k_i)_{i=1}^u: u\in [n-d+1], \text{ and }
k_i\in \bigg[ \binom{n-1-u+i}{d-1}\bigg] \text{ for every } i\in [u]\Bigg\} .\]
\end{defn}
\begin{figure}[htb]
\centering \includegraphics[width=.45\textwidth]{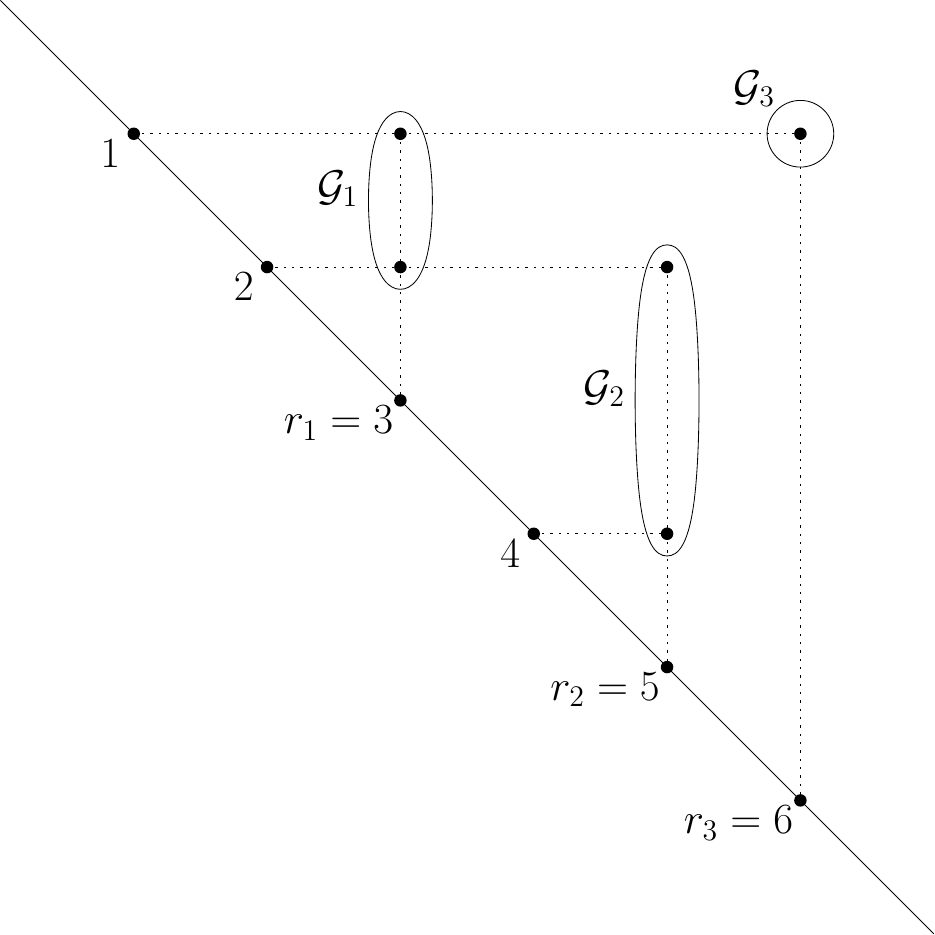}
\caption{The slicing profile of $\mathcal{F}$.}
\label{figure4}
\end{figure}
\begin{examp} \label{ex4.15}
Let $d=2$, $n=6$, and let $\mathcal{F}$ be the subset of $\binom{[6]}{2}$ defined by
\[ \mathcal{F}\coloneqq \big\{ \{1,3\}, \{1,6\}, \{2,3\}, \{2,5\}, \{4,5\}\big\}. \]
Then the slicing of $\mathcal{F}$ is the triple $\big(3,(r_1,r_2,r_3),(\mathcal{G}_1,\mathcal{G}_2,\mathcal{G}_3)\big)$,
where $r_1=3$, $r_2=5$, $r_3=6$, $\mathcal{G}_1=\{1,2\}$, $\mathcal{G}_2=\{2,4\}$ and $\mathcal{G}_3=\{1\}$; in particular,
the slicing profile of $\mathcal{F}$ is the sequence $(2,2,1)$ (see Figure \ref{figure4}).
\end{examp}
We have the following lemma.
\begin{lem} \label{l4.16}
Let $n,\eta,\vartheta,\mathcal{X},\mathcal{S}$ be as in \hyperref[pd]{$\mathrm{P}(d)$}.
Let $\bbx=\langle X_s : s\in \binom{[n]}{d} \rangle$ be an $\mathcal{X}$-valued, $\eta$-spreadable, $d$-dimensional
random array on~$[n]$ that is $(\vartheta,\mathcal{S})$-box independent. Also let $u\mik (n/2) -d+1$ be a positive
integer, and let $(k_i)_{i=1}^u\in \mathrm{SP}(\lfloor n/2 \rfloor)$. If $\mathcal{F}$ is a nonempty
subset of $\binom{[\lfloor n/2 \rfloor]}{d}$ with slicing profile $(k_i)_{i=1}^u$, then for every collection
$(a_s)_{s\in\mathcal{F}}$ of elements of $\mathcal{S}$, we have
\begin{equation} \label{e4.65}
\bigg| \prob\Big( \bigcap_{s\in\mathcal{F}} [X_s = a_s] \Big)-
\prod_{s\in\mathcal{F}} \prob\big( [X_s = a_s] \big) \bigg| \mik \gamma^{(5)}(\eta, \vartheta,d,n, (k_i)_{i=1}^{u}),
\end{equation}
where $\gamma^{(5)}(\eta, \vartheta,d,n, (k_i)_{i=1}^{u})=
\gamma^{(1)}_{k_1}\!(\eta,\vartheta,d,n) +
\sum_{i=2}^u \!\big(\gamma^{(1)}_{k_i}\!(\eta,\vartheta,d,n) + \gamma^{(4)}_{k_i}\!(\eta,\vartheta,d,n)\big)$ is as in \eqref{e4.17}.
\end{lem}
\begin{proof}
We proceed by induction on $u$. The case ``$u=1$" follows from Corollary \ref{c4.9}. Let $u\mik (n/2)-d$ be a positive integer,
and assume that \eqref{e4.65} has been proved up to $u$. Let $(k_i)_{i=1}^{u+1} \in \mathrm{SP}(\lfloor n/2 \rfloor)$,
let $\mathcal{F}$ be a subset of $\binom{[\lfloor n/2 \rfloor]}{d}$ with slicing profile $(k_i)_{i=1}^{u+1}$,
and let $(a_s)_{s\in\mathcal{F}}$ be a collection of elements of $\mathcal{S}$.

First observe that $n\meg 2d+2$ since there exists a nonempty subset of $\binom{[\lfloor n/2 \rfloor]}{d}$
with slicing profile of length at least $2$; in particular, in what follows, Proposition~\ref{p4.12} can be applied.
Let $(u, (r_i)_{i=1}^{u+1},(\mathcal{G}_i)_{i=1}^{u+1})$ denote the slicing of $\mathcal{F}$, and decompose $\mathcal{F}$
as $\mathcal{F}_1\cup\mathcal{F}_2$, where
\begin{equation} \label{e4.66}
\mathcal{F}_1\coloneqq \big\{ t\cup\{r_i\}: t\in \mathcal{G}_i, i\in[u]\big\} \ \text{ and } \
\mathcal{F}_2\coloneqq \big\{ t\cup\{r_{u+1}\} : t\in\mathcal{G}_{u+1}\big\}.
\end{equation}
Notice that $d\mik r_1< r_{u+1} \mik n/2$, $\mathcal{G}_{u+1} \subseteq \binom{[r_{u+1}-1]}{d-1}$ and
$|\mathcal{G}_{u+1}| = k_{u+1}$. By Proposition \ref{p4.12} applied for ``$r=r_{u+1}$",
``$\mathcal{G}=\mathcal{G}_{u+1}$", ``$(a_t)_{t\in\mathcal{G}}= (a_{t\cup\{r_{u+1}\}})_{t\in\mathcal{G}_{u+1}}$",
``$\mathcal{F}=\mathcal{F}_1$ and ``$(b_s)_{s\in\mathcal{F}} = (a_s)_{s\in\mathcal{F}_1}$", we have
\begin{align} \label{e4.67}
\bigg| \prob\Big( \bigcap_{s\in\mathcal{F}_1}& [X_s = a_s] \cap
  \bigcap_{s\in \mathcal{F}_2} [X_s =a_s] \Big) - \\
& - \prob\Big( \bigcap_{s\in\mathcal{F}_1} [X_s = a_s]\Big) \,
  \prob\Big( \bigcap_{s\in\mathcal{F}_2} [X_s =a_s] \Big) \bigg| \mik
	\gamma^{(4)}_{k_{u+1}}(\eta,\vartheta,d,n). \nonumber
\end{align}
On the other hand, by our inductive assumptions, we obtain that
\begin{equation} \label{e4.68}
\bigg| \prob\Big( \bigcap_{s\in\mathcal{F}_1} [X_s = a_s] \Big) -
\prod_{s\in\mathcal{F}_1} \prob\big( [X_s = a_s] \big) \bigg| \mik
\gamma^{(5)}(\eta, \vartheta,d,n, (k_i)_{i=1}^{u}).
\end{equation}
Moreover, since $|\mathcal{G}_{u+1}| = k_{u+1}$, by Corollary \ref{c4.9},
\begin{equation} \label{e4.69}
\bigg| \prob\Big( \bigcap_{s\in\mathcal{F}_2} [X_s = a_s] \Big) -
\prod_{s\in\mathcal{F}_2} \prob\big( [X_s = a_s] \big) \bigg| \mik
\gamma^{(1)}_{k_{u+1}}(\eta,\vartheta,d,n).
\end{equation}
The inductive step is completed by combining \eqref{e4.68} and \eqref{e4.69} and using
the definition of the constant $\gamma^{(5)}(\eta, \vartheta,d,n, (k_i)_{i=1}^{u+1})$ in \eqref{e4.17}.
\end{proof}
It is clear that Lemma \ref{l4.16} implies that \hyperref[pd]{$\mathrm{P}(d)$} holds true.
This completes the proof of the general inductive step, and so the entire proof of Theorem \ref{t3.2} is completed.


\section{Proof of Theorem \ref*{t1.4} and its higher-dimensional version} \label{sec5}

\numberwithin{equation}{section}

The following theorem is the higher-dimensional version of Theorem \ref{t1.4}.
(Also note that the case ``$d=1$" corresponds to random vectors.)
\begin{thm} \label{t5.1}
Let $d,m$ be two positive integers with $m\meg 2$, let $1< p\mik 2$, let $0<\ee\mik 1$,
let $k\meg d$ be an integer, and set
\begin{align}
\label{e5.1} C=C(d,m,p,\ee,k) & \coloneqq \exp\bigg(16\cdot \frac{20^d}{d!} \cdot
\frac{\ln m}{\ee^{4d}(p-1)^{d}}\cdot k^d \bigg).
\end{align}
Also let $n\meg C$ be an integer, let $\mathcal{X}$ be a set with $|\mathcal{X}|=m$, and
let $\bbx=\langle X_s:s\in \binom{[n]}{d}\rangle$ be an $\mathcal{X}$-valued, $(1/C)$-spreadable,
$d$-dimensional random array on~$[n]$. Assume that there exists $\mathcal{S}\subseteq \mathcal{X}$
with $|\mathcal{S}|=|\mathcal{X}|-1$ such that for every $a\in \mathcal{S}$ we have
\begin{equation} \label{e5.2}
\bigg| \prob\Big( \bigcap_{s\in\mathrm{Box}(d)}[X_s=a]\Big) - \prod_{s\in\mathrm{Box}(d)}\!
\prob\big([X_s=a]\big) \bigg| \mik \frac{1}{C},
\end{equation}
where $\mathrm{Box}(d)$ denotes the $d$-dimensional box defined in \eqref{e3.2}. Then for every function
$f\colon \mathcal{X}^{\binom{[n]}{d}}\to\rr$ with  $\ave[f(\bbx)]=0$ and $\|f(\bbx)\|_{L_p}=1$
there exists an interval $I$ of\, $[n]$ with $|I|=k$ such that for every $J\subseteq I$ with $|J|\meg d$ we have
\begin{equation} \label{e5.3}
\prob\big( \big|\ave[f(\bbx)\,|\, \mathcal{F}_J]\big|\mik\ee\big) \meg 1-\ee.
\end{equation}
\end{thm}
\begin{proof}
Fix $\ee$ and $k$, let $\beta=\beta(p,\ee)=\big(\frac{\ee}{10}\big)^{\frac{10}{p-1}}$ and
$\ell=\ell(p,\ee,k) =\big\lceil \frac{4}{\ee^4(p-1)}\, k \big\rceil$ be as in~\eqref{e2.2}
and \eqref{e2.3} respectively, and set
\begin{equation} \label{e5.4}
C_1=C_1(d,m,p,\ee,k)\coloneqq 3\, \bigg( \frac{108\, \binom{\ell}{d}\, 2^d\, m^{3\binom{\ell}{d}}}{\beta}\bigg)^{4^d}.
\end{equation}
\begin{claim} \label{c5.2}
We have $C_1\, (2+2^d)\mik C$.
\end{claim}
\begin{proof}[Proof of Claim \emph{\ref{c5.2}}]
By the choice of $C_1$, we have
\begin{equation} \label{e5.5}
\ln\big((2+2^d)\, C_1)= 4^d \bigg(\!\! \ln\Big(\!\!\sqrt[4^d]{3(2+2^d)}\Big) + \ln 108+
\ln \binom{\ell}{d}\!+ d\ln 2 + 3\,\binom{\ell}{d}\ln m + \ln\frac{1}{\beta}\bigg).
\end{equation}
Using the fact that $\binom{\ell}{d}\meg \ell\meg 4k\meg 4d$ and \eqref{e2.2}, we see that
\begin{align}
\label{e5.6} \ln\Big(\!\!\sqrt[4^d]{3(2+2^d)}\Big) & + \ln 108+ d\ln 2 \mik 3\, \binom{\ell}{d}\ln m, \\
\label{e5.7} & \ln\frac{1}{\beta} \mik 9\, \binom{\ell}{d}\ln m.
\end{align}
Finally, observe that
\begin{equation} \label{e5.8}
\binom{\ell}{d} \mik \frac{\ell^d}{d!} \stackrel{\eqref{e2.3}}{\mik}
\frac{1}{d!} \, \Big( \frac{5 k}{\ee^4(p-1)}\Big)^d.
\end{equation}
The claim follows by \eqref{e5.1} and \eqref{e5.4}--\eqref{e5.8}.
\end{proof}
By Claim \ref{c5.2}, the random array $\bbx$ is $\big(\frac{1}{C_1}\big)$-box independent
in the sense of part (i) of Definition \ref{d3.1}, and consequently, by Corollary \ref{c3.3},
$\bbx$ is also $(\beta,\ell)$-dissociated. The proof is completed by applying Theorem \ref{t2.2}.
\end{proof}


\section{Extensions/Refinements} \label{sec6}

\numberwithin{equation}{section}

\subsection{Dissociated random arrays} \label{subsec6.1}

The following theorem is a version of Theorem~\ref{t5.1} for the case of dissociated random arrays.
\begin{thm} \label{t6.1}
Let $1<p\mik 2$, let\, $0<\ee\mik 1$, and set
\begin{align}
\label{e6.1} c=c(\ee,p)\coloneqq \frac14\,\ee^{\frac{2(p+1)}{p}} (p-1).
\end{align}
Also let $n,d$ be positive integers with $n\meg 2d/c$, and let $\bbx$ be a dissociated,
$d\text{-dimensional}$ random array on $[n]$ whose entries take values in a measurable space~$\mathcal{X}$.
Then for every measurable function $f\colon \mathcal{X}^{\binom{[n]}{d}}\to\rr$ with $\ave[f(\bbx)]=0$
and $\|f(\bbx)\|_{L_p}=1$ there exists an interval $I$ of\, $[n]$ with $|I|\meg cn$ such that
for every $J\subseteq I$ with $|J|\meg d$ we have
\begin{equation} \label{e6.2}
\prob\big( \big|\ave[f(\bbx)\,|\, \mathcal{F}_J]\big|\mik\ee\big) \meg 1-\ee.
\end{equation}
\end{thm}
\begin{proof}
Set $k\coloneqq \lceil cn\rceil$ and $\ell\coloneqq n$, and note that $d\mik k \mik \lfloor \ell/2\rfloor$.
Using the continuity of the $L_p$-norms and the fact that the random array $\bbx$ is $(\beta,\ell)$-dissociated
for every $0<\beta\mik 1$, by Theorem \ref{t2.3} and taking the limit in \eqref{e2.5} first as $\beta$ goes
to zero and then as $r\to p^-$, there exists $I\in \binom{[n]}{k}$ such that
\begin{equation} \label{e6.3}
\big\|\ave[f(\bbx)\,|\,\mathcal{F}_I]\big\|_{L_p} \mik \frac{1}{\sqrt{p-1}} \, \sqrt{\frac{2k}{\ell}}
\mik \ee^{\frac{p+1}{p}}.
\end{equation}
By the contractive property of conditional expectation, this in turn implies that
for every $J\subseteq I$ with $|J|\meg d$ we have
\begin{equation} \label{e6.4}
\big\|\ave[f(\bbx)\,|\,\mathcal{F}_J]\big\|_{L_p} \mik \ee^{\frac{p+1}{p}}.
\end{equation}
The result follows from \eqref{e6.4} and Markov's inequality.
\end{proof}
Note that Theorem \ref{t6.1} improves upon Theorem \ref{t5.1} in two ways.
Firstly, observe that in Theorem \ref{t6.1} no restriction is imposed on
the distributions of the entries of $\bbx$. Secondly, note that the random variable $f(\bbx)$
becomes concentrated by conditioning it on a subarray whose size is proportional to $n$.

An important---especially, from the point of view of applications---class of random arrays for which Theorem \ref{t6.1}
is applicable consists of those random arrays whose entries are of the form \eqref{e1.12}, where $(\xi_1,\dots,\xi_n)$
is a random vector with independent (but not necessarily identically distributed) entries.
\begin{rem} \label{r6.2}
Observe that the lower bound on the cardinality of the set $I$ obtained by Theorem \ref{t6.1} depends polynomially
on the parameter $\ee$ and, in particular, it becomes smaller as $\ee$ gets smaller. We note that this sort
of dependence is actually necessary. This can be seen by considering (appropriately normalized) linear
functions of i.i.d. Bernoulli random variables and invoking the Berry--Esseen theorem.
\end{rem}

\subsection{Vector-valued functions of random arrays} \label{subsec6.2}

Recall that a Banach space $E$ is called \textit{uniformly convex} if for every $\ee>0$ there exists
$\delta>0$ such that for every $x,y\in E$ with $\|x\|_E=\|y\|_E=1$ and $\|x-y\|_E\meg \ee$ we have that
$\|(x+y)/2\|_E\mik 1-\delta$. It is a classical fact (see \cite{Ja72,GG71}) that for every uniformly
convex Banach space $E$ and every $p>1$ there exist an exponent $q\meg 2$ and a constant
$C>0$ such that for every $E$-valued martingale difference sequence $(d_i)_{i=1}^m$ we have
\begin{equation} \label{e6.5}
\Big( \sum_{i=1}^m \|d_i\|_{L_p(E)}^q \Big)^{1/q} \mik C\, \bigg\| \sum_{i=1}^m d_i \bigg\|_{L_p(E)}
\end{equation}
(see, also, \cite{Pi11,Pi16} for a proof and a detailed presentation of related material). Using \eqref{e6.5}
instead of Proposition \ref{p2.4} and arguing precisely as in Section~\ref{sec2}, we obtain the following
vector-valued version of Theorem \ref{t5.1}.
\begin{thm} \label{t6.3}
For every uniformly convex Banach space $E$, every pair $d,m$ of positive integers with $m\meg 2$, every $p>1$,
every $0<\ee\mik 1$ and every integer $k\meg d$, there exists a constant $C>0$ with the following property.

Let $n\meg C$ be an integer, let $\mathcal{X}$ be a set with $|\mathcal{X}|=m$ and
let $\bbx=\langle X_s:s\in \binom{[n]}{d}\rangle$ be an $\mathcal{X}$-valued, $(1/C)$-spreadable,
$d$-dimensional random array on~$[n]$. Assume that there exists $\mathcal{S}\subseteq \mathcal{X}$
with $|\mathcal{S}|=|\mathcal{X}|-1$ such that for every $a\in \mathcal{S}$ we have
\begin{equation} \label{e6.6}
\bigg| \prob\Big( \bigcap_{s\in\mathrm{Box}(d)}[X_s=a]\Big) -
\prod_{s\in\mathrm{Box}(d)} \prob\big([X_s=a]\big) \bigg| \mik \frac{1}{C},
\end{equation}
where $\mathrm{Box}(d)$ denotes the $d$-dimensional box defined in \eqref{e3.2}. Then for every function
$f\colon \mathcal{X}^{\binom{[n]}{d}}\to E$ with  $\ave[f(\bbx)]=0$ and $\|f(\bbx)\|_{L_p(E)}=1$
there exists an interval $I$ of\, $[n]$ with $|I|=k$ such that for every $J\subseteq I$ with $|J|\meg d$ we have
\begin{equation} \label{e6.7}
\prob\big( \big\|\ave[f(\bbx)\,|\, \mathcal{F}_J]\big\|_E\mik\ee\big) \meg 1-\ee.
\end{equation}
\end{thm}

\subsection{Simultaneous conditional concentration} \label{subsec6.3}

Our last result in this section can be loosely described as ``simultaneous conditional concentration";
it asserts that we can achieve concentration by conditioning on the same subarray for almost all
members of a given family of approximate spreadable random arrays with the box independence condition.
\begin{thm} \label{t6.4}
Let $d,m,p,\ee,k$ be as in Theorem \emph{\ref{t5.1}}, set\, $C'\coloneqq C(d,m,p,\ee^2/2,k)$,
where $C(d,m,p,\ee^2/2,k)$ is as in \eqref{e5.1}, and let $n\meg C'$ be an integer.
Also let $(\mathcal{V},\lambda)$ be a finite probability space, and for every
$v\in \mathcal{V}$ let $\bbx_v=\langle X^v_s:s\in \binom{[n]}{d}\rangle$
be an $(1/C')$-spreadable, $d$-dimensional random array on $[n]$ that takes values
in a set $\mathcal{X}_v$ with $|\mathcal{X}_v|=m$; assume that there exists
$\mathcal{S}_v\subseteq \mathcal{X}_v$ with $|\mathcal{S}_v|=|\mathcal{X}_v|-1$
such that for every $a\in \mathcal{S}_v$ we have
\begin{equation} \label{e6.8}
\bigg| \prob\Big( \bigcap_{s\in\mathrm{Box}(d)}[X^v_s=a]\Big) - \prod_{s\in\mathrm{Box}(d)}\!
\prob\big([X^v_s=a]\big) \bigg| \mik \frac{1}{C'},
\end{equation}
where $\mathrm{Box}(d)$ is as in \eqref{e3.2}. Finally, for every $v\in\mathcal{V}$ let
$f_v\colon \mathcal{X}_v^{\binom{[n]}{d}}\to\rr$ be a function such that $\ave[f_v(\bbx_v)]=0$
and $\|f_v(\bbx_v)\|_{L_p}=1$. Then there exist $G\subseteq \mathcal{V}$ with $\lambda(G)\meg 1-\ee$
and an interval $I$ of\, $[n]$ with $|I|=k$ such that for every $v\in G$ and every $J\subseteq I$
with $|J|\meg d$ we have
\begin{equation} \label{e6.9}
\prob\big( \big|\ave[f_v(\bbx_v)\,|\, \mathcal{F}_J]\big|\mik\ee\big) \meg 1-\ee.
\end{equation}
\end{thm}
\begin{proof}
Let $\beta'\coloneqq\beta(p,\ee^2/2)=\big(\frac{\ee^2}{20}\big)^{\frac{10}{p-1}}$ and
$\ell'\coloneqq\ell(p,\ee^2/2,k) =\big\lceil \frac{64}{\ee^8(p-1)}\, k \big\rceil$
be as in \eqref{e2.2} and \eqref{e2.3} respectively,~and
\begin{equation} \label{e6.10}
C_1'\coloneqq C_1(d,m,p,\ee^2/2,k) \stackrel{\eqref{e5.4}}{=}
3\, \bigg( \frac{108\, \binom{\ell'}{d}\, 2^d\, m^{3\binom{\ell'}{d}}}{\beta'}\bigg)^{4^d}.
\end{equation}
By Claim \ref{c5.2}, we have $C_1'\,(2+2^d)\mik C'$ and consequently, by \eqref{e6.8} and Corollary \ref{c3.3},
the random array $\bbx_v$ is $(\beta',\ell')$-dissociated for every $v\in\mathcal{V}$.

Next, set $m'\coloneqq \lfloor\ell'/k\rfloor$ and observe that,
by the choices of $\ell',C'$ and \eqref{e2.3} and \eqref{e5.1},
\begin{equation} \label{e6.11}
m'\meg \frac{\ell'}{2k} \meg \frac{2^5}{\ee^{8}(p-1)} \ \ \ \text{ and } \ \ \
m'k\mik \ell'\mik C'\mik n.
\end{equation}
Let $v\in \mathcal{V}$ be arbitrary. For every $i\in [m']$ set $J_i\coloneqq \big\{ k(i-1)+j:j\in [k]\big\}$
and let $\mathcal{F}^v_{J_i}$ be the $\sigma$-algebra generated by the subarray of $\bbx_v$
determined by $J_i$ (see Definition \ref{d1.1}). As in~\eqref{e2.24}, we define a filtration
$(\mathcal{A}^v_i)_{i=0}^{m'}$ by setting $\mathcal{A}_0^v\coloneqq \{\emptyset,\Omega\}$ and
\begin{equation} \label{e6.12}
\mathcal{A}_i\coloneqq \bigvee_{l=1}^i \mathcal{F}_{J_l} \ \ \ \text{ for every $i\in [m']$.}
\end{equation}
Finally, let $(d^v_i)_{i=1}^{m'}$ denote the martingale difference sequence
of the Doob martingale for $f_v(\bbx_v)$ with respect to the filtration $(\mathcal{A}^v_i)_{i=0}^{m'}$.
By Proposition \ref{p2.4} and our assumptions,
\begin{equation} \label{e6.13}
\sum_{i=1}^{m'} \Big( \underset{v\sim\lambda}{\ave} \, \|d^v_i\|_{L_p}^2\Big) =
\underset{v\sim\lambda}{\ave}\, \bigg( \sum_{i=1}^{m'} \|d^v_i\|_{L_p}^2\bigg)
\stackrel{\eqref{e2.9}}{\mik} \frac{1}{p-1},
\end{equation}
and so, there exists $i_0\in [m']$ such that
\begin{equation} \label{e6.14}
\underset{v\sim\lambda}{\ave} \, \|d^v_{i_0}\|_{L_p} \mik \frac{1}{\sqrt{m'(p-1)}}.
\end{equation}
Set $J\coloneqq J_{i_0}$ and $r\coloneqq \frac{p+1}{2}$.
By Markov's inequality, there exists $G\subseteq \mathcal{V}$~with
\begin{equation} \label{e6.15}
\lambda(G)\meg 1- \big(m'(p-1)\big)^{-1/4} \stackrel{\eqref{e6.11}}{\meg} 1-\ee
\end{equation}
such that for every $v\in G$ we have
\begin{equation} \label{e6.16}
\big\| \ave[d_{i_0}^v\, |\, \mathcal{F}_J]\big\|_{L_r} \mik
\|d^v_{i_0}\|_{L_p} \mik \frac{1}{\sqrt[4]{m'(p-1)}} \stackrel{\eqref{e6.11}}{\mik} \frac{\ee^2}{2}
\mik \frac12\, \ee^{\frac{r+1}{r}},
\end{equation}
where we have used  the monotonicity of the $L_p$-norms, the contractive property of conditional expectation,
and the fact that $1<r< p\mik 2$. Next observe that since $\mathcal{F}_J\subseteq \mathcal{A}_{i_0}$ we have
$\ave[d^v_{i_0}\,|\,\mathcal{F}_J] =\ave[f(\bbx_v)\,|\,\mathcal{F}_J]-
\ave\big[ \ave[f(\bbx_v)\,|\,\mathcal{A}_{i_0-1}]\, |\, \mathcal{F}_J\big]$.
Moreover, the fact that the random array $\bbx_v$ is $(\beta',\ell')$-dissociated implies that
the $\sigma$-algebras $\mathcal{A}_{i_0-1}$ and $\mathcal{F}_J$ are
$\beta'$-mixing in the sense of \eqref{e2.11}. Thus, by Proposition \ref{p2.7}, \eqref{e6.16},
the triangle inequality and the choice of $\beta'$, we obtain that for every $v\in G$,
\begin{align} \label{e6.17}
\big\| \ave[ f(\bbx_v)\, |\, \mathcal{F}_J]\big\|_{L_r} \mik
\frac12\, \ee^{\frac{r+1}{r}} + 10 \big(\beta'\big)^{\frac{1}{r} - \frac{1}{p}} \mik
\ee^{\frac{r+1}{r}}.
\end{align}
The proof is completed by \eqref{e6.17} and Markov's inequality.
\end{proof}
\begin{rem} \label{r6.5}
We note that there is also an extension of Theorem \ref{t6.1} in the spirit of Theorem~\ref{t6.4}.
More precisely, if we assume in Theorem \ref{t6.4} that for every $v\in\mathcal{V}$ the random array
$\bbx_v$ is dissociated (not necessarily finite-valued), then the interval $I$ can be selected
so as~$|I|\meg c'n$, where $c'\coloneqq \frac{1}{4}\,\ee^{\frac{2(2p+1)}{p}} (p-1)$.
\end{rem}


\part{Connection with combinatorics} \label{part2}


\section{Random arrays arising from combinatorial structures} \label{sec7}

\numberwithin{equation}{section}

In this section we present examples of boolean, spreadable, high-dimensional random arrays
that arise from combinatorial structures and they satisfy the box independence condition and/or
are approximately dissociated.

\subsection{From graphs and hypergraphs to spreadable random arrays} \label{subsec7.1}

Let $d\meg 2$ be an integer, and let $V$ be a finite set with $|V|\meg d$. With every subset $A$ of $V^d$
we associate a boolean, spreadable, $d$-dimensional random array
$\boldsymbol{X}_A=\langle X^A_s:s\in \binom{\nn}{d}\rangle$
on $\nn$ defined by setting for every $s=\{i_1<\cdots<i_d\}\in \binom{\nn}{d}$,
\begin{equation} \label{e7.1}
X^A_s\coloneqq \mathbf{1}_{A}(\xi_{i_1},\dots,\xi_{i_d}),
\end{equation}
where $(\xi_i)$ is a sequence of i.i.d. random variables uniformly distributed on $V$. (Notice that
if $A$ is a nonempty proper subset of $V^d$, then the entries of $\boldsymbol{X}_A$ are not independent.)

A special case of this construction, which is relevant in the ensuing discussion, is obtained by
considering a $d$-uniform hypergraph on $V$. Specifically, given a $d$-uniform hypergraph $G$~on~$V$,
we identify $G$ with a subset $\mathcal{G}$ of $V^d$ via the rule
\begin{equation} \label{e7.2}
(v_1,\dots,v_d)\in\mathcal{G} \Leftrightarrow \{v_1,\dots,v_d\}\in G,
\end{equation}
and we define $\boldsymbol{X}_G=\langle X^G_s:s\in \binom{\nn}{d}\rangle$ to be the random array
in \eqref{e7.1} that corresponds to the set $\mathcal{G}$. Note that this definition is canonical,
in the sense that various combinatorial parameters of $G$ can be expressed as functions
of the finite subarrays of $\boldsymbol{X}_G$. For instance, let $n\meg d$ be an integer,
and let $F$ be a $d$-uniform hypergraph on $[n]$; then, denoting by $t(F,G)$ the homomorphism
density of $F$ in $G$ (see \cite[Chapter~5]{Lov12}), we have
\begin{equation} \label{e7.3}
t(F,G)= \ave[f_F(\boldsymbol{X}_{G,n})],
\end{equation}
where $f_F\colon \rr^{\binom{[n]}{d}}\to\rr$ is defined by setting
for every $\boldsymbol{x}=(x_t)_{t\in \binom{[n]}{d}}\in \rr^{\binom{[n]}{d}}$
\begin{equation} \label{e7.4}
f_F(\boldsymbol{x})\coloneqq \prod_{s\in F} x_s
\end{equation}
and $\boldsymbol{X}_{G,n}$ denotes the subarray of $\boldsymbol{X}_{G}$ determined by $[n]$
(see Definition \ref{d1.1}). Of course, similar identities are valid for weighted uniform hypergraphs.

As we shall see shortly in Proposition \ref{p7.2} below, in this framework
the box independence condition of the random array $\boldsymbol{X}_G$ is in fact equivalent to a well-known
combinatorial property of $G$, namely its \textit{quasirandomness}. That said, we point out that
the connection between quasirandomness and random arrays with a symmetric distribution has been
observed in much greater generality in the general theory of limits of combinatorial structures;
see, \textit{e.g.}, \cite{Au08,CR20,DJ08,ES12,J11,Lov12,Ra07,To17}.

\subsubsection{Quasirandom graphs and hypergraphs} \label{subsubsec7.1.1}

Quasirandom objects are deterministic discrete structures that behave like random ones for most practical
purposes. The phenomenon was first discovered in the context of graphs by Chung, Graham and Wilson \cite{CGW88,CGW89}
who build upon previous work of Thomason \cite{Tho87}. In the~last~twenty years the theory was also extended
to hypergraphs, and it has found numerous significant applications in number theory and theoretical computer
science (see, \textit{e.g.}, \cite{Ro15}).

\subsubsection*{\emph{7.1.1.1.}} \label{subsubsubsec7.1.1.1}

Much of the modern theory of quasirandomness is developed using the box norms introduced by Gowers \cite{Go07}.
Specifically, let $d\meg 2$ be an integer, let $(\Omega,\Sigma,\mu)$ be a probability space,
and let $\Omega^d$ be equipped with the product measure. For every integrable random variable
$f\colon \Omega^d\to \rr$ we define its \textit{box norm} $\|f\|_\square$ by the rule
\begin{equation} \label{e7.5}
\|f\|_\square \coloneqq \bigg(\int \prod_{\boldsymbol{\epsilon}\in \{0,1\}^d}
f(\boldsymbol{\omega}_{\boldsymbol{\epsilon}})\, d\boldsymbol{\mu}(\boldsymbol{\omega})\bigg)^{1/2^d},
\end{equation}
where $\boldsymbol{\mu}$ denotes the product measure on $\Omega^{2d}$ and, for every
$\boldsymbol{\omega}=(\omega^0_1,\omega^1_1,\dots,\omega^0_d,\omega^1_d)\in \Omega^{2d}$
and every $\boldsymbol{\epsilon}=(\epsilon_1,\dots,\epsilon_d)\in \{0,1\}^d$ we have
$\boldsymbol{\omega}_{\boldsymbol{\epsilon}}\coloneqq (\omega_1^{\epsilon_1},\dots,\omega_d^{\epsilon_d})\in \Omega^d$;
by convention, we set $\|f\|_\square\coloneqq +\infty$ if the integral in \eqref{e7.5} does not exist.
The quantity $\|\cdot\|_\square$ is a norm on the vector space $\{f\in L_1: \|f\|_\square <+\infty\}$,
and it~satisfies the following inequality, known as the \textit{Gowers--Cauchy--Schwarz inequality}:
for every collection $\langle f_{\boldsymbol{\epsilon}}: \epsilon\in \{0,1\}^d\rangle$ of integrable
random variables on $\Omega^d$ we have
\begin{equation} \label{e7.6}
\bigg| \int \prod_{\boldsymbol{\epsilon}\in \{0,1\}^d}
f_{\boldsymbol{\epsilon}}(\boldsymbol{\omega}_{\boldsymbol{\epsilon}})
\, d\boldsymbol{\mu}(\boldsymbol{\omega})\bigg| \mik
\prod_{\boldsymbol{\epsilon}\in \{0,1\}^d} \|f_{\boldsymbol{\epsilon}}\|_\square.
\end{equation}
For proofs of these basic facts, as well as for a more complete presentation of related material, we refer to \cite[Appendix~B]{GT10}.

\subsubsection*{\emph{7.1.1.2.}} \label{subsubsubsec7.1.1.2}

The link between the box norms and quasirandomness is given in the following definition.
\begin{defn}[Box uniformity] \label{d7.1}
Let $d\meg 2$, let $V$ be a finite set with $|V|\meg d$, and let~$\varrho>0$. We say that a $d$-uniform
hypergraph $G$ on $V$ is \emph{$\varrho$-box uniform} $($or, simply, \emph{box uniform} if $\varrho$ is clear from the
context$)$ provided that
\begin{equation} \label{e7.7}
\big\| \mathbf{1}_{\mathcal{G}}- \ave[\mathbf{1}_{\mathcal{G}}]\big\|_{\square}\mik \varrho,
\end{equation}
where $\mathcal{G}$ is as in \eqref{e7.2}. $($Here, we view $V$ as a discrete probability space equipped
with the uniform probability measure.$)$
\end{defn}
Of course, Definition \ref{d7.1} is interesting when the parameter $\varrho$ is much smaller
than~$\ave[\mathbf{1}_{\mathcal{G}}]$. We also note that although box uniformity is defined analytically,
it has a number of equivalent combinatorial formulations. For instance, it is easy to see that a graph $G$
is box uniform if and only if it has roughly the expected number of $4$-cycles; see, \textit{e.g.},
\cite{ACHPS18,CGW88,CGW89,CG90,KRS02,LM15,Ro15,To17} for more information on quasirandomness properties
of graphs and hypergraphs and their relation with analytical properties of box norms.

\subsubsection{The box independence condition via quasirandomness} \label{subsubsec7.1.2}

We have the following pro\-position (see part (i) of Definition \ref{d3.1} for the definition of box independence).
\begin{prop} \label{p7.2}
Let $d\meg 2$ be an integer, and let $V$ be a finite set with $|V|\meg d$. Also let $G$ be a $d$-uniform hypergraph
on $V$, let $\bbx_G=\langle X^G_s:s\in \binom{\nn}{d}\rangle$ be the random array associated with~$G$ via \eqref{e7.1},
and for every integer $n\meg d$ let $\bbx_{G,n}$ denote the subarray of $\bbx_G$ determined by~$[n]$. Finally, let
$\varrho,\vartheta>0$. Then the following hold.
\begin{enumerate}
\item[(i)] If $G$ is $\varrho$-box uniform, then for every integer $n\meg d$ the random array
$\bbx_{G,n}$ is $(2^d\varrho,\{1\})$-box independent.
\item[(ii)] Conversely, if $\bbx_{G,n}$ is $(\vartheta,\{1\})$-box independent for some $($equivalently, every$)$
integer $n\meg d$, then $G$ is $(12\,\vartheta^{1/8^d})$-box uniform.
\end{enumerate}
\end{prop}
\begin{proof}
We start with the following observation, which follows readily from \eqref{e7.1}.
Let $\mathrm{Box}(d)$ be the $d$-dimensional box defined in \eqref{e3.2},
and let $F$ be a nonempty subset of $\mathrm{Box}(d)$. Then there exists a
subset\footnote{Note that this subset is essentially unique.} $H$ of $\{0,1\}^d$ with $|F|=|H|$ and such that
\begin{equation} \label{e7.8}
\ave\Big[ \prod_{s\in F} X^G_s\Big] = \int \prod_{\boldsymbol{\epsilon}\in H}
\mathbf{1}_{\mathcal{G}}(\boldsymbol{\omega}_{\boldsymbol{\epsilon}})\, d\boldsymbol{\mu}(\boldsymbol{\omega}).
\end{equation}
(Here, by $\boldsymbol{\mu}$ we denote the uniform probability measure on $V^{2d}$,
and we follow the conventions described right after \eqref{e7.5}.)

We proceed to the proof of part (i). Notice that $\|\mathbf{1}_{\mathcal{G}}\|_{\square}\mik
\|\mathbf{1}_{\mathcal{G}}\|_{L_\infty}\mik 1$ and, moreover, $\ave[X^G_s]=\ave[\mathbf{1}_{\mathcal{G}}]$
for every $s\in \binom{\nn}{d}$. Taking into account these observations and using our assumption,
identity \eqref{e7.8}, a telescopic argument and the Gowers--Cauchy--Schwarz inequality \eqref{e7.6}, we obtain that
\begin{equation} \label{e7.9}
\bigg| \ave\Big[ \prod_{s\in \mathrm{Box}(d)} X^G_s\Big] - \prod_{s\in \mathrm{Box}(d)} \ave[X^G_s] \bigg|
\mik 2^d \varrho.
\end{equation}
Since the random array $\bbx_G$ is spreadable, by Definition \ref{d3.1} and \eqref{e7.9},
we see that $\bbx_G$ is $(2^d \varrho,\{1\})$-box independent.

For the proof of part (ii) we will need the following fact.
\begin{fact} \label{f7.3}
Let the notation and assumptions be as in part \emph{(ii)} of Proposition \emph{\ref{p7.2}}.
Then for every nonempty subset $F$ of\, $\mathrm{Box}(d)$ we have
\begin{equation} \label{e7.10}
\bigg| \ave\Big[ \prod_{s\in F} X^G_s\Big] - \prod_{s\in F} \ave[X^G_s] \bigg| \mik
36\cdot 2^{2d}\cdot \vartheta^{1/4^d}.
\end{equation}
\end{fact}
\begin{proof}[Proof of Fact \emph{\ref{f7.3}}]
If $\vartheta >1$, then \eqref{e7.10} is straightforward; thus, we may assume that $0<\vartheta\mik 1$.
Let $n\meg 4d$ be arbitrary. Notice that the random array $\bbx_{G,n}$ is $\eta\text{-spreadable}$ for every
$0<\eta \mik 1$ and $(\vartheta,\{1\})$-box independent. Therefore, the result follows by applying Theorem \ref{t3.2}
and taking the limit in the left-hand-side of \eqref{e3.5} as $\eta$ goes to zero and $n$ tends to infinity.
\end{proof}
Using Fact \ref{f7.3}, we shall estimate the quantity
\begin{equation} \label{e7.11}
\big\| \mathbf{1}_{\mathcal{G}}- \ave[\mathbf{1}_{\mathcal{G}}]\big\|_{\square}^{2^d} \stackrel{\eqref{e7.5}}{=}
\sum_{H\subseteq \{0,1\}^d}  (-1)^{2^d-|H|} \ave[\mathbf{1}_{\mathcal{G}}]^{2^d-|H|} \int \prod_{\boldsymbol{\epsilon}\in H}
\mathbf{1}_{\mathcal{G}}(\boldsymbol{\omega}_{\boldsymbol{\epsilon}})\, d\boldsymbol{\mu}(\boldsymbol{\omega}).
\end{equation}
(Here, as in the proof of Lemma \ref{l3.6}, we use the convention that the product over an empty index-set
is equal to 1.) Fix a nonempty subset $H$ of $\{0,1\}^d$, and let $F$ be the subset of $\mathrm{Box}(d)$
with $|F|=|H|$ and such that \eqref{e7.8} is satisfied; since $\ave[X^G_s]=\ave[\mathbf{1}_{\mathcal{G}}]$
for every $s\in \binom{\nn}{d}$, by Fact \ref{f7.3}, we have
\begin{equation} \label{e7.12}
\bigg| \int \prod_{\boldsymbol{\epsilon}\in H} \mathbf{1}_{\mathcal{G}}(\boldsymbol{\omega}_{\boldsymbol{\epsilon}})\,
d\boldsymbol{\mu}(\boldsymbol{\omega}) - \ave[\mathbf{1}_{\mathcal{G}}]^{|H|} \bigg| \mik 36\cdot
2^{2d} \cdot \vartheta^{1/4^d}.
\end{equation}
By \eqref{e7.11}, \eqref{e7.12} and the fact that the resulting sum vanishes,
we conclude that $G$ is $(12\,\vartheta^{1/8^d})$-box uniform, as desired.
\end{proof}

\subsection{Mixtures} \label{subsec7.2}

An important property of the class of boolean, spreadable random arrays~is that it is closed under mixtures.
More precisely, let $n,d,J$ be positive integers with $n\meg d\meg 2$ and let
$\bbx_1=\langle X^1_s:s\in \binom{[n]}{d}\rangle, \dots, \bbx_J=\langle X^J_s:s\in \binom{[n]}{d}\rangle$
be boolean, spreadable, $d$-dimensional random arrays on $[n]$. Then, for any choice $\lambda_1,\dots,\lambda_J$
of convex coefficients, there exists a boolean, spreadable, $d$-dimensional random array
$\bbx=\langle X_s:s\in \binom{[n]}{d}\rangle$ on $[n]$ that satisfies
\begin{equation} \label{e7.13}
\ave\Big[ \prod_{s\in \mathcal{F}} X_s \Big] =
\sum_{j=1}^J \lambda_j \ave\Big[ \prod_{s\in \mathcal{F}} X^j_s \Big]
\end{equation}
for every nonempty finite subset $\mathcal{F}$ of $\binom{[n]}{d}$.

It turns out that boolean, spreadable random arrays that satisfy the box independence condition
are also closed under mixtures under suitable conditions. In particular, we have the following proposition
(its proof follows from a direct computation).
\begin{prop} \label{p7.4}
Let $n,d,J$ be positive integers with $n\meg d\meg 2$, and let $\delta,\vartheta>0$. For every $j\in [J]$ let
$\bbx_j=\langle X^j_s:s\in \binom{[n]}{d}\rangle$ be a boolean, spreadable, $d$-dimensional random array on $[n]$
that is $(\vartheta,\{1\})$-box independent and satisfies $|\ave[X^j_s]-\delta|\mik \vartheta$ for all
$s\in \binom{[n]}{d}$. If $\bbx$ is any mixture of $\bbx_1,\dots,\bbx_J$, then $\bbx$ is
$(2^{d+2}\vartheta,\{1\})$-box independent.
\end{prop}
Observe that, by Propositions \ref{p7.2} and \ref{p7.4}, if $G_1,\dots,G_J$ are quasirandom, $d\text{-uniform}$
hypergraphs with the same edge density, then any mixture of the finite subarrays of $\bbx_{G_1},\dots,\bbx_{G_J}$
satisfies the box independent condition. We note that this fact essentially characterizes the box independence condition.
Specifically, it follows from \cite[Propositions 8.3 and~3.1]{DTV21} that for every boolean, spreadable, $d$-dimensional
random array $\bbx$ that satisfies the box independence condition, there exist quasirandom, $d\text{-uniform}$ hypergraphs
$G_1,\dots,G_J$ with the same edge density, such that the law of $\bbx$ is close, in the total variation distance,
to the law of a mixture of the finite subarrays of $\bbx_{G_1},\dots,\bbx_{G_J}$.

\subsection{Further combinatorial structures} \label{subsec7.3}

Let $n,k,d$ be positive integers with $n\meg d\meg 2$ and $k\mik \binom{n}{d}$, and let
$\boldsymbol{\Xi}=\langle \xi_e: e\in \binom{[n]}{d}\rangle$ be a $d$-dimensional random array with
boolean entries that are uniformly distributed on the set of all $\boldsymbol{x}\in \{0,1\}^{\binom{[n]}{d}}$
that have exactly $k$ ones. (In particular, $\boldsymbol{\Xi}$ is exchangeable.)
The random array $\boldsymbol{\Xi}$ generates the classical fixed size Erd\H{o}s--R\'{e}nyi random
graph/hypergraph, and it is clear that it satisfies the box independence condition. By taking
products\footnote{These products have a natural combinatorial interpretation; \textit{e.g.},
they can be used to count subgraphs of random graphs.} of the entries of $\boldsymbol{\Xi}$
as in~\eqref{e1.12}, one also obtains exchangeable random arrays that are approximately dissociated.

Spreadable random arrays---and, in particular, spreadable random arrays that satisfy the box independence condition---are
also closely related to a class of stochastic processes introduced by Furstenberg and Katznelson \cite{FK91} in their
proof of the density Hales--Jewett theorem (see also \cite{Au11,DT21}). Unfortunately, this relation is not so
transparent as in case of graphs and hypergraphs, and we shall refrain from discussing it further since it requires
several probabilistic and Ramsey-theoretic tools in order to be properly exposed.


\section{Quasirandom families of graphs: proof of Theorem \ref*{t1.8}} \label{sec8}

\numberwithin{equation}{section}

We start with some preparatory material that will be used throughout this section.
If $K\subseteq I$ are two nonempty finite sets, then for every $z\in\{0,1\}^{I}$
by $z\upharpoonright K\in \{0,1\}^{K}$ we shall denote the restriction of $z$ on $K$.
Moreover, for every subset $A$ of $\{0,1\}^I$ and every $x\in \{0,1\}^K$ by
$A_x\coloneqq \big\{y\in \{0,1\}^{I\setminus K}: x\cup y\in A\big\}$
we shall denote the section of $A$~at~$x$. We will need the following lemma.
\begin{lem} \label{l8.1}
Let $0<\ee\mik 1$, and let $\ell,m$ be positive integers such that
\begin{equation} \label{e8.1}
\ell\meg \frac{2^{m+1}}{\ee^2}.
\end{equation}
Let $I$ be a nonempty finite set, and let $D_1,\dots,D_\ell\subseteq I$
be pairwise disjoint nonempty sets each with at most $m$ elements. Then for every $A\subseteq \{0,1\}^I$
there exists $i_0\in [\ell]$ such that
\begin{equation} \label{e8.2}
|\mu_1(A_x)-\mu(A)|\mik \ee
\end{equation}
for every $x\in \{0,1\}^{D_{i_0}}$, where $\mu_1$ denotes the uniform probability measure on
$\{0,1\}^{I\setminus D_{i_0}}$ and $\mu$ denotes the uniform probability measure on $\{0,1\}^{I}$.
\end{lem}
Lemma \ref{l8.1} is a typical combinatorial application of conditional concentration,
and it follows from \cite[Theorem 1$'$]{DKT16}. That said,
for the convenience of the reader we shall briefly recall the argument that also gives
slightly better estimates for this special case.
\begin{proof}[Proof of Lemma \emph{\ref{l8.1}}]
We view $\{0,1\}^I$ as a discrete probability space equipped with $\mu$. For every $i\in [\ell]$
let $\mathcal{D}_i$ denote the $\sigma$-algebra on $\{0,1\}^I$ generated by the partition
\begin{equation} \label{e8.3}
\Big\{ \big\{z\in \{0,1\}^I: z\upharpoonright D_i=x\big\}: x\in \{0,1\}^{D_i}\Big\},
\end{equation}
and let $(\mathcal{F}_i)_{i=0}^\ell$ denote the filtration on $\{0,1\}^I$ generated by the finite sequence
$(\mathcal{D}_i)_{i=1}^\ell$; that is, $\mathcal{F}_0\coloneqq \big\{\emptyset,\{0,1\}^I\big\}$
is the trivial $\sigma$-algebra,~and
\begin{equation} \label{e8.4}
\mathcal{F}_i\coloneqq \bigvee_{j=1}^i \mathcal{D}_{j} \ \ \ \text{ for every $i\in [\ell]$.}
\end{equation}
Let $(d_i)_{i=1}^{\ell}$ be the martingale difference sequence of the Doob martingale
for~$\mathbf{1}_A$ with respect to the filtration $(\mathcal{F}_i)_{i=0}^{\ell}$.
Since $\|d_1\|_{L_2}^2+ \cdots + \|d_\ell\|_{L_2}^2= \|d_1+\cdots+d_\ell\|_{L_2}^2 \mik 1$,
there exists $i_0\in [\ell]$ such that $\|d_{i_0}\|_{L_2}^2\mik 1/\ell$ that further implies,
by the contractive property of conditional expectation, that
$\big\|\ave[d_{i_0}\,|\,\mathcal{D}_{i_0}]\big\|_{L_2}^2\mik 1/\ell$. After observing that
$\ave[d_{i_0}\,|\,\mathcal{D}_{i_0}]=\ave[\mathbf{1}_A\,|\,\mathcal{D}_{i_0}] -\mu(A)$,
by Chebyshev's inequality, we obtain that
\begin{equation} \label{e8.5}
\mu_2\Big(x\in \{0,1\}^{D_{i_0}}: |\mu_1(A_x)-\mu(A)|\meg \ee\Big) \mik \frac{1}{\ell\cdot\ee^2}
\stackrel{\eqref{e8.1}}{\mik} \frac{1}{2^{m+1}}<\frac{1}{2^m},
\end{equation}
where $\mu_2$ denotes the uniform probability measure on $\{0,1\}^{D_{i_0}}$. Since
$\big|\{0,1\}^{D_{i_0}}\big|\mik 2^m$, this yields that $|\mu_1(A_x)-\mu(A)|\mik\ee$ for every
$x\in \{0,1\}^{D_{i_0}}$.
\end{proof}

\subsection{Isomorphic invariant families of graphs} \label{subsec8.1}

Let $n\meg 2$ be an integer, and recall that by $\mu$ we denote the uniform probability measure
on $\{0,1\}^{\binom{[n]}{2}}$. Also recall that a family of graphs
$\mathcal{A}\subseteq \{0,1\}^{\binom{[n]}{2}}$ is called \emph{isomorphic invariant} if for every
$G\subseteq \binom{[n]}{2}$ we have that $G$ belongs to $\mathcal{A}$ only if every isomorphic
copy of $G$ belongs to $\mathcal{A}$; see also \eqref{e1.9}. This subsection is devoted to
the proof of the following proposition.
\begin{prop} \label{p8.2}
Let $n\meg2^{25}$ be an integer, let $\mathcal{A}\subseteq\{0,1\}^{\binom{[n]}{2}}$ be
isomorphic~invariant, and let $\gamma(\mathcal{A})$ denote the unique nonnegative real such that
\begin{equation} \label{e8.6}
\gamma(\mathcal{A})=\mathbf{P}\big( W: W\cup\{i,k\}, W\cup\{i,\ell\},W\cup\{j,k\},W\cup\{j,\ell\}\in\mathcal{A}\big)
\end{equation}
for every $U=\{i<j<k<\ell\}\in \binom{[n]}{4}$, where $\mathbf{P}$ denotes the uniform probability measure
on $\{0,1\}^{\binom{[n]}{2}\setminus \binom{U}{2}}$. Then we have
\begin{equation} \label{e8.7}
\gamma(\mathcal{A})\meg \mu(\mathcal{A})^4 - \frac{18}{\sqrt{\log_2 n}}.
\end{equation}
\end{prop}
\begin{proof}
We start by setting
\begin{equation} \label{e8.8}
k\coloneqq \big\lfloor  \sqrt{\log_2 n}\big\rfloor, \ \ \ \ \
\ee_1\coloneqq \frac{3}{4(k-2)}, \ \ \ \ \
\ee_2 \coloneqq \frac{3}{k-2}
\end{equation}
and we observe that $k\meg 5$. Also notice that we may assume that
$\mu(\mathcal{A})^4\! > 18/\sqrt{\log_2 n}$ (otherwise, the result is straightforward)
that further implies, by \eqref{e8.8}, that
\begin{equation} \label{e8.9}
\mu(\mathcal{A}) \meg \ee_1 \ \ \ \text{ and } \ \ \
\mu(\mathcal{A})^2 \meg 2\ee_1+\frac{1}{k-1}.
\end{equation}
Set $\ell_1\coloneqq \big\lceil 2^{\binom{k}{2}+1}\, \ee_1^{-2}\big\rceil$ and observe that
$\ell_1 k\mik n$; therefore, we may select $H_1,\dots,H_{\ell_1}$ pairwise disjoint subsets of $[n]$
each of size $k$. Since the family $\mathcal{A}$ is isomorphic invariant, by
Lemma \ref{l8.1} applied to the sets $D_1\coloneqq \binom{H_1}{2},\dots, D_{\ell_1}\coloneqq\binom{H_{\ell_1}}{2}$
and $\mathcal{A}$, we~see~that
\begin{equation} \label{e8.10}
\mu_1(\mathcal{A}_x)\meg \mu(\mathcal{A})-\ee_1 \stackrel{\eqref{e8.9}}{\meg} 0
\end{equation}
for every $x\in\{0,1\}^{\binom{[k]}{2}}$, where $\mu_1$ is the uniform probability measure
on $\{0,1\}^{\binom{[n]}{2}\setminus \binom{[k]}{2}}$.
Now for every $i\in [k-1]$ let $x_i\in \{0,1\}^{\binom{[k]}{2}}$ be such that
$x_i^{-1}(\{1\})=\big\{ \{1,i+1\} \big\}$; namely, $x_i$ is the graph on $[k]$ with the single
edge $\{1,i+1\}$. Setting $\delta_1\coloneqq \mu_1(\mathcal{A}_{x_1})$ and
$\delta_2 \coloneqq \mu_1(\mathcal{A}_{x_1}\cap\mathcal{A}_{x_2})$ and using again the fact that $\mathcal{A}$
is isomorphic invariant, for every $i,j\in [k-1]$ with $i\neq j$ we have
\begin{equation} \label{e8.11}
\mu_1(\mathcal{A}_{x_i})=\delta_1 \ \ \text{ and } \ \
\mu_1(\mathcal{A}_{x_i}\cap \mathcal{A}_{x_j})=\delta_2.
\end{equation}
Therefore, by the Cauchy--Schwarz inequality,
\begin{align}
\label{e8.12} \mu(\mathcal{A})^2-2\ee_1 & \mik  \big(\mu(\mathcal{A})-\ee_1\big)^2
\stackrel{\eqref{e8.10}}{\mik} \ave\bigg[\frac{1}{k-1}\sum_{i=1}^{k-1}\mathbf{1}_{\mathcal{A}_{x_i}}\bigg]^2 \\
& \mik \frac{1}{(k-1)^2}\ave\bigg[\Big(\sum_{i=1}^{k-1}\mathbf{1}_{\mathcal{A}_{x_i}}\Big)^2\bigg]
=\frac{\delta_1}{k-1} +\frac{k-2}{k-1}\delta_2 \mik \frac{1}{k-1} + \delta_2. \nonumber
\end{align}
Next, for every $i\in [k-2]$ let $x_{1,i},x_{2,i}\in \{0,1\}^{\binom{[k]}{2}}$ be defined by
$x_{1,i}^{-1}(\{1\}) = \big\{\{ i,k-1 \}\big\}$ and $x_{2,i}^{-1}(\{1\}) = \big\{\{ i,k \}\big\}$.
Using once again the isomorphic invariance of $\mathcal{A}$ and setting
$\delta_4 \coloneqq \mu_1(\mathcal{A}_{x_{1,1}}\cap\mathcal{A}_{x_{2,1}}\cap\mathcal{A}_{x_{1,2}}
\cap\mathcal{A}_{x_{2,2}})$, we see that
\begin{equation} \label{e8.13}
\mu_1(\mathcal{A}_{x_{1,i}}\cap \mathcal{A}_{x_{2,i}})=\delta_2 \ \ \text{ and } \ \
\mu_1(\mathcal{A}_{x_{1,i}}\cap \mathcal{A}_{x_{2,i}}\cap \mathcal{A}_{x_{1,j}}\cap \mathcal{A}_{x_{2,j}})=\delta_4
\end{equation}
for every $i,j\in [k-1]$ with $i\neq j$. Hence, by \eqref{e8.9} and the Cauchy--Schwarz inequality,
\begin{align}
\label{e8.14} \Big(\mu(\mathcal{A})^2 & -2\ee_1 -\frac{1}{k-1}\Big)^2  \stackrel{\eqref{e8.12}}{\mik}
\delta_2^2 =\ave\bigg[\frac{1}{k-2}\sum_{i=1}^{k-2} \mathbf{1}_{\mathcal{A}_{x_{1,i}}\cap \mathcal{A}_{x_{2,i}}}\bigg]^2 \\
& \ \ \mik \frac{1}{(k-2)^2}\, \ave\bigg[\Big(\sum_{i=1}^{k-2}
\mathbf{1}_{\mathcal{A}_{x_{1,i}}\cap \mathcal{A}_{x_{2,i}}}\Big)^2\bigg]
=\frac{\delta_2}{k-2}+\frac{k-3}{k-2}\delta_4\mik\frac{1}{k-2}+\delta_4, \nonumber
\end{align}
which yields that
\begin{equation} \label{e8.15}
\delta_4\meg \mu(\mathcal{A})^4-4\ee_1-\frac{3}{k-2}.
\end{equation}
We will show that the parameter $\gamma(\mathcal{A})$ is roughly equal to $\delta_4$.
Clearly, this is enough to complete the proof.

To this end, let
$z_1,z_2,z_3,z_4\in \{0,1\}^{\binom{[4]}{2}}$ be defined by setting
$z_1^{-1}(\{1\})= \big\{\{1,3\}\big\}$, $z_2^{-1}(\{1\})= \big\{\{1,4\}\big\}$,
$z_3^{-1}(\{1\})= \big\{\{2,3\}\big\}$ and~$z_4^{-1}(\{1\})= \big\{\{2,4\}\big\}$;
for every $i\in [4]$ let $\mathcal{A}_{z_i}\coloneqq
\big\{y\in\{0,1\}^{\binom{[n]}{2}\setminus \binom{[4]}{2}}: z_i \cup y\in\mathcal{A}\big\}$
denote the section of $\mathcal{A}$ at $z_i$ and set
$S\coloneqq \mathcal{A}_{z_1}\cap\mathcal{A}_{z_2}\cap\mathcal{A}_{z_3}\cap\mathcal{A}_{z_4}$.
Moreover, set $\ell_2 \coloneqq \big\lceil 2^{\binom{k}{2}-\binom{4}{2}+1}\, \ee^{-2}_2\big\rceil$
and observe that $\ell_2(k-4)\mik n$; hence, we may select $K_1,\dots,K_{\ell_2}$ pairwise disjoint
subsets of $[n]\setminus [4]$ each of size $k-4$. Using again the isomorphic invariance of $\mathcal{A}$
and applying Lemma \ref{l8.1} to the sets
$D_1\coloneqq \binom{K_1\cup [4]}{2}\setminus \binom{[4]}{2},\dots,
D_{\ell_2}\coloneqq \binom{K_{\ell_2}\cup [4]}{2}\setminus \binom{[4]}{2}$ and $S$, we obtain that
\begin{equation}\label{e8.16}
|\mu_2(S_u)-\gamma(\mathcal{A})|\mik \ee_2
\end{equation}
for every $u\in\{0,1\}^{\binom{[k]}{2}\setminus \binom{[4]}{2}}$, where
$\mu_2$ denotes the uniform probability measure on $\{0,1\}^{\binom{[n]}{2}\setminus \binom{[k]}{2}}$.
Let $u_0\in \{0,1\}^{\binom{[k]}{2}\setminus \binom{[4]}{2}}$ be
such that $u_0^{-1}(\{1\})=\emptyset$. Then, by the definitions of $S$ and $\delta_4$
and the isomorphic invariance of $\mathcal{A}$, we see that $\mu_2(S_{u_0})=\delta_4$ and consequently,
by \eqref{e8.8}, \eqref{e8.15} and \eqref{e8.16} and the fact that $k\meg 5$, we conclude that
\begin{equation}\label{e8.17}
\gamma(\mathcal{A}) \meg \mu(\mathcal{A})^4-4\ee_1-\frac{3}{k-2}-\ee_2 \meg
\mu(\mathcal{A})^4 - \frac{18}{k+1}
\stackrel{\eqref{e8.8}}{\meg}\mu(\mathcal{A})^4 - \frac{18}{\sqrt{\log_2 n}}. \qedhere
\end{equation}
\end{proof}

\subsection{Proof of Theorem \ref{t1.8}} \label{subsec8.2}

The main goal of the proof is to extract out of the quasirandom family $\mathcal{A}$
a boolean two-dimensional approximately spreadable random array $\bbx$ that satisfies the box
independence condition; once this is done, the proof will be completed by an application
of Theorem \ref{t1.5}.

\subsubsection{Preliminary tools} \label{subsubsec8.2.1}

We start with a more precise, quantitative, version of Proposition \ref{p1.3} for boolean
two-dimensional random arrays. Specifically, let $\ell,m,r\meg 2$ be integers with
$\ell\mik m$, and recall that the multicolor hypergraph Ramsey number $R_{\ell}(m,r)$
is the least integer $N\meg m$ such that for every set $X$ with $|X|\meg N$
and every coloring $c\colon \binom{X}{\ell}\to [r]$ there exists
$Y\in \binom{X}{m}$ such that $c$ is constant on $\binom{Y}{\ell}$.
It is a classical result due to Erd\H{o}s and Rado~\cite{ER52} that the numbers
$R_{\ell}(m,r)$ have (at most) a tower-type dependence with respect to the
parameters $\ell,m,r$. The following fact is the promised quantitative version
of Proposition~\ref{p1.3}.
\begin{fact} \label{f8.3}
Let $0<\eta\mik 1$, let $\ell\meg 2$ be an integer, and let $N$ be an
integer such that
\begin{equation} \label{e8.18}
N\meg R_\ell\bigg( 2\ell, \big\lceil 2^{\binom{\ell}{2}}\,\eta^{-1}\big\rceil^{2^{\binom{\ell}{2}}}\bigg).
\end{equation}
Then for every boolean two-dimensional random array $\bbx$ on $[N]$ there exists
$L\in \binom{[N]}{\ell}$ such that the random subarray $\bbx_L$ of $\bbx$ is $\eta$-spreadable
$($see Definition \emph{\ref{d1.1}}$)$.
\end{fact}
\begin{proof}
Fix $\bbx$ and, for notational convenience, set $k\coloneqq 2^{\binom{\ell}{2}}$. Observe that there
exists a partition of the positive cone of the unit ball of $(\rr^k, \|\cdot\|_{\ell_1})$ into
$\lceil k/\eta\rceil^k$ parts, each of $\|\cdot\|_{\ell_1}$-diameter at most~$\eta$.
This partition induces, naturally, a coloring $c$ of $\binom{[N]}{\ell}$ with $\lceil k/\eta\rceil^k$ colors:
color $L\in\binom{[N]}{\ell}$ according to the shell of the partition that contains that law of $\bbx_L$.
Notice, in particular, that for every $L,K\in\binom{[N]}{\ell}$ with $c(L) = c(K)$ we have
$\rho_{\mathrm{TV}}(P_L,P_K)\mik \eta$, where $P_L$ and $P_K$ denote the
laws of the subarrays $\bbx_K$ and $\bbx_L$, respectively (recall that $\rho_{\mathrm{TV}}$
stands for the total variation distance). By \eqref{e8.18}, there exists $M\in\binom{[N]}{2\ell}$
such that the coloring $c$ is constant on $\binom{M}{\ell}$. Let $L$ denote the set of
the first $\ell$ elements~of~$M$. We claim that $\bbx_L$ is $\eta$-spreadable.
Indeed, let $r\in \{2,\dots,\ell\}$ be an integer, and let $Q,R\in \binom{L}{r}$.
We select $J,K\in\binom{M}{\ell}$ such that $Q$ and $R$ are the sets of the first $r$~elements
of $J$~and~$K$, respectively. Then $\rho_{\mathrm{TV}}(P_Q,P_R)\mik\rho_{\mathrm{TV}}(P_J,P_K)\mik\eta$,
where $P_Q,P_R,P_J, P_K$ denote the laws of the subarrays $\bbx_Q,\bbx_R,\bbx_J,\bbx_K$,~respectively.
\end{proof}
We proceed by introducing some terminology and some pieces of notation.
Let $m\meg \ell$ be positive integers and let $F\in \binom{[m]}{\ell}$; given two subsets $L\subseteq M$
of $\nn$ with $|L|=\ell$ and $|M|=m$, we say that the \emph{relative position of $L$ inside $M$ is $F$}
if, denoting by $\{i_1<\cdots <i_m\}$ the increasing enumeration of $M$, we have that
$L=\{i_j:j\in F\}$.

Moreover, for every finite subset $M$ of $\nn$ with $|M|\meg 2$
every $e\in \binom{M}{2}$ we shall denote by $x(e,M)\in \{0,1\}^{\binom{M}{2}}$ the
unique element satisfying $x(e,M)^{-1}(\{1\})=\{e\}$.

It is also convenient to introduce the following definition. (Recall that for every integer $n\meg 2$
by $\mu$ we denote the uniform probability measure on $\{0,1\}^{\binom{[n]}{2}}$.)
\begin{defn}[Admissibility] \label{d8.4}
Let $0<\eta\mik 1$, let $n\meg m\meg \ell\meg 2$ be integers, let $\mathcal{A}\subseteq \{0,1\}^{\binom{[n]}{2}}$
and let $F\in \binom{[m]}{\ell}$. Given $P\subseteq [n]$ with $|P|\meg m$,
we say that $P$ is \emph{$(\mathcal{A},\eta,F)$-admissible} if for every $M\in \binom{P}{m}$,
denoting by $\nu$ the uniform probability measure on $\{0,1\}^{\binom{[n]}{2}\setminus \binom{M}{2}}$,
the following hold.
\begin{enumerate}
\item[($\mathcal{P}$1)] \label{9P1} For every $x\in \{0,1\}^{\binom{M}{2}}$
we have $|\nu(\mathcal{A}_x)-\mu(\mathcal{A})|\mik \eta$, where $\mathcal{A}_x$
is the section of $\mathcal{A}$ at $x$.
\item[($\mathcal{P}$2)] \label{9P2} If\, $L\in \binom{M}{\ell}$ is the unique subset of $M$
whose relative position inside $M$ is $F$, then the two-dimensional random array
$\langle \mathbf{1}_{\mathcal{A}_{x(e,M)}}:e\in \binom{L}{2}\rangle$ is $\eta$-spreadable.
$($Here, we view $\{0,1\}^{\binom{[n]}{2}\setminus \binom{M}{2}}$ as a discrete probability
space equipped with $\nu$ and we denote by $\mathcal{A}_{x(e,M)}$ the section of $\mathcal{A}$ at $x(e,M)$.$)$
\end{enumerate}
\end{defn}
We have the following lemma.
\begin{lem} \label{l8.5}
Let $0<\eta\mik 1$, let $\ell\meg 2$ be an integer, and set
\begin{equation} \label{e8.19}
m=m(\eta,\ell)\coloneqq
R_\ell\bigg( 2\ell, \big\lceil 2^{\binom{\ell}{2}}\, \eta^{-1}\big\rceil^{2^{\binom{\ell}{2}}}\bigg).
\end{equation}
Also let $p$ be an integer with $p\meg m$ and set
\begin{equation} \label{e8.20}
q_1=q_1(\eta,\ell,p)\coloneqq R_{\ell}\bigg(p,\binom{m}{\ell}\bigg) \ \ \ \text{ and } \ \ \
q=q(\eta,\ell,p)\coloneqq q_1\cdot\big\lceil 2^{q_1+1} \, \eta^{-2}\big\rceil.
\end{equation}
If $n\meg q$ is an integer and $\mathcal{A}\subseteq
\{0,1\}^{\binom{[n]}{2}}$ is a family of graphs,
then for every $Q\in \binom{[n]}{q}$ there exist $F\in \binom{[m]}{\ell}$
and $P\in \binom{Q}{p}$
such that $P$ is $(\mathcal{A},\eta,F)$-admissible.
\end{lem}
\begin{proof}
Let $n, \mathcal{A}$ be as in the statement of the lemma, and let $Q\in \binom{[n]}{q}$ be arbitrary.
We set $r\coloneqq \big\lceil 2^{q_1+1}\, \eta^{-2}\big\rceil$ and we select pairwise disjoint sets
$I_1,\ldots,I_r\in\binom{Q}{q_1}$. By Lemma~\ref{l8.1} applied to the sets
$D_1\coloneqq \binom{I_1}{2},\dots,D_r\coloneqq \binom{I_r}{2}$ and $\mathcal{A}$,
there exists $i\in[r]$ such that
\begin{equation} \label{e8.21}
|\mu_1(\mathcal{A}_x)-\mu(\mathcal{A})|\mik \eta
\end{equation}
for every $x\in\{0,1\}^{\binom{I_i}{2}}$, where $\mu_1$ is the uniform probability measure on
$\{0,1\}^{\binom{[n]}{2}\setminus \binom{I_i}{2}}$ and $\mathcal{A}_x$ is the section
of $\mathcal{A}$ at $x$. We set $Q_1\coloneqq I_i$. By double averaging, \eqref{e8.21}
further implies that for every $M\in \binom{Q_1}{m}$ and every $x\in \{0,1\}^{\binom{M}{2}}$
we have $|\nu(\mathcal{A}_x)-\mu(\mathcal{A})|\mik \eta$, where $\nu$ is
the uniform probability measure on $\{0,1\}^{\binom{[n]}{2}\setminus \binom{M}{2}}$.
In other words, property~(\hyperref[9P1]{$\mathcal{P}$1}) in Definition \ref{d8.4} will be
satisfied as long as the desired set $P$ is contained in $Q_1$.

For property (\hyperref[9P2]{$\mathcal{P}$2}) we argue as follows. Let $M\in \binom{Q_1}{m}$
be arbitrary; by the choice of the constant $m$ in \eqref{e8.19} and Fact \ref{f8.3} applied to the boolean,
two-dimensional random array $\langle \mathbf{1}_{\mathcal{A}_{x(e,M)}}:e\in \binom{M}{2}\rangle$,
there exists $F_M\in \binom{[m]}{\ell}$ such that if $L\in \binom{M}{\ell}$ is the unique subset of $M$
whose relative position inside $M$ is $F_M$, then the random array
$\langle \mathbf{1}_{\mathcal{A}_{x(e,M)}}:e\in \binom{L}{2}\rangle$ is
$\eta$-spreadable. By the choice of $q_1$ in \eqref{e8.20} and another application of
Ramsey's theorem, there exist $P\in \binom{Q_1}{p}$ and $F\in \binom{[m]}{\ell}$ such that $F_M=F$ for
every $M\in \binom{P}{m}$. That is, property~(\hyperref[9P2]{$\mathcal{P}$2}) is satisfied for $P$, as desired.
\end{proof}

\subsubsection{Numerical parameters} \label{subsubsec8.2.2}

Our next step is to introduce some numerical parameters.
We fix $0<\delta\mik 1$ and an integer $k\meg 2$, and
we begin by selecting $0<\eta,\theta_0\mik 1$ and an integer $\ell\meg 4k$
such that
\begin{equation} \label{e8.22}
\eta + \bigg( \lfloor\ell/k\rfloor^{-1} + 4\,\binom{k}{2}\eta + 576\,\binom{k}{2}\,
\big(\ell^{-1/16} + \eta^{1/16}+(6\eta+\theta_0)^{1/16}\big) \bigg)^{1/2} < \delta^{\binom{k}{2}+1}.
\end{equation}
Next we set
\begin{align}
\label{e8.23} & m = m(\eta,\ell)\stackrel{\eqref{e8.19}}{\coloneqq} R_\ell\bigg( 2\ell,
\big\lceil 2^{\binom{\ell}{2}}\, \eta^{-1}\big\rceil^{2^{\binom{\ell}{2}}}\bigg),\\
\label{e8.24} & J\coloneqq \big\lceil 2^{\binom{m}{2}-5} \, \eta^{-2} \big\rceil, \\
\label{e8.25} & p \coloneqq 5(m-4) J+4, \\
\label{e8.26} & q_1 \coloneqq q_1(\eta,\ell,p) \stackrel{\eqref{e8.20}}{=} R_m\bigg(p,\binom{m}{\ell} \bigg).
\end{align}
Finally, we define
\begin{equation} \label{e8.27}
q_0 \coloneqq q(\eta,\ell,p) \stackrel{\eqref{e8.20}}{=}
q_1\cdot \big\lceil 2^{q_1+1} \, \eta^{-2} \big\rceil \ \ \ \text{ and } \ \ \
\theta\coloneqq \frac12\, \min\bigg\{ \theta_0,\binom{q_0}{4}^{-1} \bigg\}.
\end{equation}

\subsubsection{Completion of the proof} \label{subsubsec8.2.3}

We are ready for the main part of the argument. As above, let $0<\delta\mik 1$ and $k\meg 2$.
Also let $n\meg q_0$ be an integer and let $\mathcal{A}\subseteq \{0,1\}^{\binom{[n]}{2}}$
be a $\theta$-quasirandom family of graphs with $\mu(\mathcal{A})\meg \delta$,
where $q_0, \theta$ are as in \eqref{e8.27}.

By Lemma \ref{l8.5}, for every $Q \in \binom{[n]}{q_0}$ we fix $P_Q\in \binom{Q}{p}$ and
$F_Q\in \binom{[m]}{\ell}$ such that $P_Q$ is $(\mathcal{A},\eta,F_Q)$-admissible in the sense
of Definition \ref{d8.4}. Moreover, denoting by $\{r_1<\dots<r_p\}$ the increasing enumeration of
$P_Q$, we set $U_Q\coloneqq \big\{r_{j(m-4)J+j}: j\in\{1,2,3,4\}\big\}\in \binom{Q}{4}$. Then observe that
\begin{equation} \label{e8.28}
\bigg|\bigg\{U_Q: Q\in \binom{[n]}{q_0}\bigg\}\bigg|\meg \binom{n}{4}\,\binom{q_0}{4}^{-1}
\stackrel{\eqref{e8.27}}{\meg} 2\theta \binom{n}{4}.
\end{equation}
In order to see that the first inequality in \eqref{e8.28} is satisfied, notice that the uniform
probability measure on $\binom{[n]}{4}$ can be obtain by first sampling a set $Q\in \binom{[n]}{q_0}$
uniformly at random, and then sampling a set $U\in \binom{Q}{4}$ also uniformly at random;
that is, for every $A\subseteq \binom{[n]}{4}$ we have
\begin{equation} \label{e8.29}
\frac{|A|}{\binom{n}{4}} = \frac{1}{\binom{n}{q_0}} \sum_{Q\in \binom{[n]}{q_0}} \,
\frac{\big| A\cap \binom{Q}{4}\big|}{\binom{q_0}{4}}.
\end{equation}
The desired estimate follows by applying \eqref{e8.29} to the set
$A\coloneqq \Big\{ U_Q:Q\in \binom{[n]}{q_0}\Big\}$ and recalling that $U_Q\in \binom{Q}{4}$
for all $Q\in \binom{[n]}{q_0}$.

By \eqref{e8.28} and the fact that the family $\mathcal{A}$ is $\theta$-quasirandom in the sense of
Definition~\ref{d1.7}, we may select $Q_0\in\binom{[n]}{q_0}$ such that, writing $U_{Q_0}=\{u_1<u_2<u_3<u_4\}$
and setting\footnote{Recall that by $x(\{u_1,u_3\},U_{Q_0})\in \{0,1\}^{\binom{Q_0}{2}}$ is such that $x(\{u_1,u_3\},U_{Q_0})^{-1}(\{1\})=\{u_1,u_3\}$, and similarly for
$x(\{u_1,u_4\},U_{Q_0})$, $x(\{u_2,u_3\},U_{Q_0})$, $x(\{u_2,u_4\},U_{Q_0})$.}
$B\coloneqq \mathcal{A}_{x(\{u_1,u_3\},U_{Q_0})} \cap \mathcal{A}_{x(\{u_1,u_4\},U_{Q_0})}
\cap \mathcal{A}_{x(\{u_2,u_3\},U_{Q_0})} \cap \mathcal{A}_{x(\{u_2,u_4\},U_{Q_0})}$, we have
\begin{equation} \label{e8.30}
\mathbf{P}(B)\mik\mu(\mathcal{A})^4+\theta,
\end{equation}
where $\mathbf{P}$ is the uniform probability measure on $\{0,1\}^{\binom{[n]}{2}\setminus \binom{U_{Q_0}}{2}}$.

Next observe that, by the choice of $U_{Q_0}$, the set $P_{Q_0}$ has $J(m-4)$ elements between any two
consecutive $u_j$'s, $J(m-4)$ elements before $u_1$ and $J(m-4)$ elements after $u_4$. Therefore,
we may select $M_1,\ldots,M_J\in\binom{P_{Q_0}}{m}$ such that
\begin{enumerate}
\item[$\bullet$] for every $i\in [J]$, denoting by $L_i$ the unique element of
$\binom{M_i}{\ell}$ whose relative position inside $M_i$ is $F_{Q_0}$, we have that $U_{Q_0}\subseteq L_i$, and
\item[$\bullet$] the sets $M_1\setminus U_{Q_0},\ldots, M_J\setminus U_{Q_0}$ are pairwise disjoint.
\end{enumerate}
By  Lemma \ref{l8.1} applied to
$D_1\coloneqq \binom{M_1}{2}\setminus\binom{U_{Q_0}}{2},\dots, D_J\coloneqq \binom{M_J}{2}\setminus\binom{U_{Q_0}}{2}, B$
and the choice of $J$ in \eqref{e8.24}, there exists $i_0\in [J]$ such that,
setting $M\coloneqq M_{i_0}$ and $L\coloneqq L_{i_0}$,~we~have
\begin{equation} \label{e8.31}
|\nu(B_z)-\mathbf{P}(B)|\mik \eta
\end{equation}
for every $z\in \{0,1\}^{\binom{M}{2}\setminus \binom{U_{Q_0}}{2}}$, where $\nu$ is the uniform probability measure
on $\{0,1\}^{\binom{[n]}{2}\setminus \binom{M}{2}}$ (as usual, $B_z$ denotes the section of $B$ at $z$).
Now let $z_0\in \{0,1\}^{\binom{M}{2}\setminus \binom{U_{Q_0}}{2}}$ be such that $z_0^{-1}(\{1\})=\emptyset$.
By \eqref{e8.30} and \eqref{e8.31}, we obtain that
\begin{equation} \label{e8.32}
\nu(B_{z_0})\mik \mu(\mathcal{A})^4 + \eta + \theta.
\end{equation}
Also notice that $B_{z_0} = \mathcal{A}_{x(\{u_1,u_3\},M)} \cap \mathcal{A}_{x(\{u_1,u_4\},M)}
\cap \mathcal{A}_{x(\{u_2,u_3\},M)} \cap \mathcal{A}_{x(\{u_2,u_4\},M)}$. On the other hand,
recall that $P_{Q_0}$ is $(\mathcal{A},\eta,F_{Q_0})$-admissible and that $L$ is the unique
subset of $M\in \binom{P_{Q_0}}{m}$ whose relative position inside $M$ is $F_{Q_0}$.
Taking into account these observations and using properties (\hyperref[9P1]{$\mathcal{P}$1})
and (\hyperref[9P2]{$\mathcal{P}$2}) in Definition \ref{d8.4}, we see that the boolean random array
$\langle\mathbf{1}_{\mathcal{A}_{x(e,M)}}:e\in \binom{L}{2}\rangle$ is $\eta$-spreadable
and it satisfies \eqref{e1.7} with $\vartheta=6\eta+\theta$. Let $\{s_1<\dots<s_m\}$ be
the increasing enumeration of $M$, and for every $j\in\big\{1,\dots,\lfloor\ell/k\rfloor\big\}$~set
\begin{equation} \label{e8.33}
K_j\coloneqq \big\{s_i:i\in[jk] \setminus[(j-1)k]\big\} \ \ \ \text{ and } \ \ \
\Gamma_j\coloneqq \bigcap_{e\in \binom{K_j}{2}}\mathcal{A}_{x(e,M)}.
\end{equation}
By Theorem \ref{t1.5}, property (\hyperref[9P1]{$\mathcal{P}$1}) in Definition \ref{d8.4}
and the previous discussion, for every $j,j'\in\big\{1,\dots,\lfloor\ell/k\rfloor\big\}$ with $j\neq j'$ we have
\begin{align}
\label{e8.34} \Big|\nu(\Gamma_j)- \mu(\mathcal{A})^{\binom{k}{2}}\Big| & \mik
\eta\, \binom{k}{2}+144\, \binom{k}{2}\, \big(\ell^{-1/16} +\eta^{1/16}+ (6\eta+\theta)^{1/16}\big), \\
\label{e8.35} \Big|\nu(\Gamma_j\cap\Gamma_{j'})- \mu(\mathcal{A})^{2 \binom{k}{2}}\Big| & \mik
2\eta\, \binom{k}{2}+ 288\, \binom{k}{2}\, \big(\ell^{-1/16} +\eta^{1/16}+ (6\eta+\theta)^{1/16}\big).
\end{align}
Introduce the random variable
\begin{equation} \label{e8.36}
X\coloneqq \frac{1}{\lfloor\ell/k\rfloor} \sum_{j=1}^{\lfloor\ell/k\rfloor}\mathbf{1}_{\Gamma_j},
\end{equation}
and observe that, by \eqref{e8.34} and \eqref{e8.35}, we have
\begin{equation} \label{e8.37}
\Big\|X - \mu(\mathcal{A})^{\binom{k}{2}}\Big\|_{L_2}\!\! \mik
\bigg( \!\lfloor\ell/k\rfloor^{-1} + 4\eta\,\binom{k}{2} + 576\,\binom{k}{2}\,
\big(\ell^{-1/16} + \eta^{1/16}+(6\eta+\theta)^{1/16}\big)\!\bigg)^{1/2}.
\end{equation}
Let $x_0\in \{0,1\}^{\binom{M}{2}}$ be the unique element satisfying $x_0^{-1}(\{1\})=\emptyset$;
since $\mu(\mathcal{A})\meg \delta$,~we~have
\begin{align}
\label{e8.38} & \ave[X \mathbf{1}_{\mathcal{A}_{x_0}}] \meg \mu(\mathcal{A})^{\binom{k}{2}} \nu(\mathcal{A}_{x_0}) -
\sqrt{\nu(\mathcal{A}_{x_0})}\, \Big\|X - \mu(\mathcal{A})^{\binom{k}{2}}\Big\|_{L_2} \\
\stackrel{(\hyperref[9P1]{\mathcal{P}1}),\eqref{e8.37}}{\meg} &
\delta^{\binom{k}{2}+1} - \eta - \bigg(\!\lfloor\ell/k\rfloor^{-1} +4\eta\,\binom{k}{2} +
576\, \binom{k}{2} \big(\ell^{-1/16} +\eta^{1/16}+(6\eta+\theta)^{1/16}\big)\!\bigg)^{1/2}. \nonumber
\end{align}
Notice that the choice of $\theta$ in \eqref{e8.27} ensures that $\theta\mik\theta_0$.
Therefore, by \eqref{e8.38} and \eqref{e8.22}, we have $\ave[X\mathbf{1}_{\mathcal{A}_{x_0}}]>0$
that, in turn, implies that there exists $j_0\in \big\{1,\dots,\lfloor\ell/k\rfloor\big\}$ such that
$\mathcal{A}_{x_0}\cap \Gamma_{j_0}\neq\emptyset$. By the choices of $\Gamma_{j_0}$ in \eqref{e8.33}
and $x_0$, it is clear that the set $K_{j_0}\in \binom{[n]}{k}$ is as desired.
The proof of Theorem \ref{t1.8} is completed.
\begin{rem}[Analysis of the bounds] \label{r8.6}
Using the Erd\H{o}s--Rado theorem \cite{ER52}, it is not hard to see that
the proof of Theorem \ref{t1.8} yields a tower-type dependence of $\theta$ and
$\ell_0$ with respect to the parameters $\delta$ and $k$. More precisely,
there exists a primitive recursive $\psi\colon \nn\times\nn\to\nn$ function
belonging to the class $\mathcal{E}^4$ of Grzegorczyk’s hierarchy\footnote{See \cite[Appendix A]{DK16}
for an introduction to Grzegorczyk’s hierarchy and a discussion of its role in analyzing bounds in Ramsey theory.}
such that $\theta^{-1},\ell_0\mik \psi\big( \lceil \delta^{-1}\rceil, k\big)$
for every $0<\delta\mik 1$ and every integer $k\meg 2$.
\end{rem}
\begin{rem}[Extensions to families of uniform hypergraphs] \label{r8.7}
Theorem~\ref{t1.8} can be extended to families of $d$-uniform hypergraphs
$\mathcal{A}\subseteq \{0,1\}^{\binom{[n]}{d}}$ for any integer $d\meg 2$; this can be done
by using Theorem \ref{t3.2} instead Theorem \ref{t1.5} and appropriately modifying the
notion of quasirandomness in Definition \ref{d1.7}. We leave the (fairly straightforward)
formulations of these extensions to the interested reader.
\end{rem}
\begin{rem} \label{r8.8}
Let $\mathcal{A}\subseteq \{0,1\}^{\binom{[n]}{2}}$ be a family of graphs on $[n]$, let
$K\subseteq [n]$ with $|K|\meg 2$, and let $\mathcal{S}\subseteq \{0,1\}^{\binom{K}{2}}$
be a family of graphs on $K$. We say that $\mathcal{A}$ \emph{smashes} $\mathcal{S}$
if there exists $W\subseteq \binom{[n]}{2}\setminus \binom{K}{2}$ such that
$W\cup H\in\mathcal{A}$ for every $H\in\mathcal{S}$. With this terminology,
Conjecture \ref{con1.6} is simply asking whether every dense family of graphs smashes some clique,
while Theorem \ref{t1.8} is equivalent to saying that if the family $\mathcal{A}$ is dense
and quasirandom (in the sense of Definition \ref{d1.7}), then it smashes all graphs
with at most one edge on some $K\in \binom{[n]}{k}$. It would be interesting to find
quasirandomness conditions that ensure that the family~$\mathcal{A}$ smashes
richer families of small graphs. In this direction, the following problem
is the most intriguing.
\begin{problem} \label{pr8.9}
Find natural quasirandomness conditions on a family of graphs~$\mathcal{A}$
that ensure that $\mathcal{A}$ smashes \emph{all} graphs on some $K\in \binom{[n]}{k}$.
\end{problem}
\end{rem}


\appendix


\section{Examples} \label{Appendix-A}

\numberwithin{equation}{section}

Our goal in this appendix is to present examples that show that the box independence condition
in Theorems \ref{t1.4} and \ref{t5.1} is essentially optimal. We focus on boolean random arrays
as this case already covers all underlying phenomena.

\subsection{Boxes and faces} \label{subsecA.1}

We start by introducing some terminology that will be used throughout this section.
Let $d\meg 2$ be an integer; we say that a subset $B$ of $\binom{\nn}{d}$ is a
\emph{$d$-dimensional box of $\nn$} if it is a $d$-dimensional box of $[n]$
for some integer $n\meg 2d$ (see Subsection \ref{subsec3.1}). Moreover, we say that a subset $F$
of $\binom{\nn}{d}$ is a \emph{$(d-1)$-face of $\nn$} if it is of the form $\mathrm{Box}(\mathcal{H})$,
where $\mathcal{H}=(H_1,\dots,H_d)$ is a finite sequence of nonempty subsets of $\nn$ of cardinality
at most $2$ with $\max(H_i)<\min(H_{i+1})$ for all $i\in [d-1]$, and such that
$\sum_{i=1}^{d}|H_i|=2d-1$. (Thus, $|H_i|=2$ for all but at one $i\in [d]$.)

\subsection{The two-dimensional case} \label{subsecA.2}

We have the following proposition.
\begin{prop} \label{pa.1}
There exists a boolean, exchangeable, two-dimensional random array
$\bbx = \langle X_s: s\in \binom{\nn}{2}\rangle$ on $\nn$ with the following properties.
\begin{enumerate}
\item[($\mathcal{P}$1)] \label{aP11} For every $s\in \binom{\nn}{2}$ we have $\ave[X_s]=\frac{1}{2}$.
\item[($\mathcal{P}$2)] \label{aP12} For every distinct $s,t\in \binom{\nn}{2}$ we have $\ave[X_s X_t]=\frac{1}{4}$.
\item[($\mathcal{P}$3)] \label{aP13} For every $2$-dimensional box $B$ of $\nn$ and every nonempty subset
$G$ of $B$ with $G\neq B$ we have $\ave\big[\prod_{s\in G} X_s\big]=(\frac{1}{2})^{|G|}$.
\item[($\mathcal{P}$4)] \label{aP14} For every $2$-dimensional box $B$ of $\nn$ we have
$\ave\big[ \prod_{s\in B} X_s \big] =\frac{3}{2}(\frac{1}{2})^4$.
\item[($\mathcal{P}$5)] \label{aP15} Let $n\meg 8$ be an integer, and let $\bbx_{n}$ denote the subarray of $\bbx$
determined by~$[n]$ $($see Definition \emph{\ref{d1.1}}$)$. Then there exists a  multilinear
polynomial $f\colon \rr^{\binom{[n]}{2}}\to \rr$ of degree $4$ with $\ave[f(\bbx_n)] = 0$ and
$\|f(\bbx_n)\|_{L_\infty}\mik 1$, such that for every subset $I$~of~$[n]$ with $|I|\meg 8$ we have
$\prob\big( \big|\ave[f(\bbx_n)\,|\, \mathcal{F}_I]\big|\meg 2^{-11}\big) \meg 2^{-11}$.
\end{enumerate}
\end{prop}
\begin{proof}
We will define the random array $\bbx$ by providing an integral representation of its
distribution. (Of course, this maneuver is expected  by the Aldous--Hoover representation
theorem \cite{Ald81,Hoo79}.) Specifically, set $V\coloneqq \{0,1\}$ and $A\coloneqq\{(0,0),(1,1)\}\subseteq V^2$;
we view $V$ as a discrete probability space equipped with the uniform probability measure.
We also set $\Omega\coloneqq\{0,1\}^{\binom{\nn}{2}}$ and we equip $\Omega$ with the product
$\sigma$-algebra, which we denote by~$\Sigma$. Let $\prob$ denote the $(1/2,1/2)$-mixture
of the uniform distribution on $\{0,1\}^{\binom{\nn}{2}}$ and the law of the random array
$\bbx_A$ associated with $A$ via \eqref{e7.1}; that is, $\prob$ is the unique probability
measure on $(\Omega,\Sigma)$ that satisfies, for every nonempty finite subset $\mathcal{F}$
of $\binom{\nn}{2}$,~that
\begin{equation} \label{ea.1}
\prob\Big(\big\{ (x_t)_{t\in \binom{\nn}{2}}\in\Omega :x_s=1 \text{ for all }s\in\mathcal{F}\big\}\Big) =
\frac{1}{2}\Big(\frac{1}{2}\Big)^{|\mathcal{F}|}+ \frac{1}{2}\int\prod_{s\in \mathcal{F}}
\mathbf{1}_A(\boldsymbol{v}_s)\, d\boldsymbol{\mu}(\boldsymbol{v}),
\end{equation}
where: (i) $\boldsymbol{\mu}$ denotes the product measure on $V^\nn$ obtained by equipping each
factor with the uniform probability measure on $V$, and (ii) for every $\boldsymbol{v}=(v_i)\in V^\nn$
and every $s=\{i_1<i_2\}\in \binom{\nn}{2}$ by $\boldsymbol{v}_s=(v_{i_1},v_{i_2})\in V^2$ we denote
the restriction of $\boldsymbol{v}$ on the coordinates determined by $s$.
Next, for every $s\in \binom{\nn}{2}$ let $X_s\colon \Omega\to \{0,1\}$ denote the projection on the
$s$-th coordinate, that is, $X_s\big((x_t)_{t\in \binom{\nn}{2}}\big)=x_s$ for every
$(x_t)_{t\in \binom{\nn}{2}}\in\Omega$.
The fact that the set $A$ is symmetric implies that the random array $\bbx=\langle X_s: s\in \binom{\nn}{2}\rangle$
is exchangeable; moreover, for every nonempty finite subset $\mathcal{F}$ of $\binom{\nn}{2}$ we have
\begin{equation} \label{ea.2}
\ave\Big[ \prod_{s\in\mathcal{F}}X_s \Big] = \frac{1}{2}\Big(\frac{1}{2}\Big)^{|\mathcal{F}|} +
 \frac{1}{2}\int\prod_{s\in\mathcal{F}}\mathbf{1}_A(\boldsymbol{v}_s)\, d\boldsymbol{\mu}(\boldsymbol{v}).
\end{equation}
Using \eqref{ea.2}, properties (\hyperref[aP11]{$\mathcal{P}$1})--(\hyperref[aP14]{$\mathcal{P}$4})
follow from a direct computation.

In order to verify property (\hyperref[aP15]{$\mathcal{P}$5}) we argue as in the proof of Proposition \ref{p2.8}.
Fix an integer $n\meg 8$. Let $\mathrm{Box}(2)$ be the $2$-dimensional box of $\nn$ defined in \eqref{e3.2}.
We define $f\colon \rr^{\binom{[n]}{2}}\to \rr$ by setting for every
$\boldsymbol{x}=(x_t)_{t\in \binom{[n]}{2}}\in \rr^{\binom{[n]}{2}}$
\begin{align} \label{ea.3}
f(\boldsymbol{x}) & \coloneqq \prod_{s\in \mathrm{Box}(2)} \!\! x_s-
\ave\Big[\prod_{s\in \mathrm{Box}(2)} \!\! X_s\Big] \\
& \! \stackrel{\eqref{e3.2}}{=} x_{\{1,3\}}x_{\{1,4\}}x_{\{2,3\}}x_{\{2,4\}}-
\ave[ X_{\{1,3\}}X_{\{1,4\}}X_{\{2,3\}}X_{\{2,4\}}]. \nonumber
\end{align}
It is clear that $f$ is a multilinear polynomial of degree $4$ that satisfies $\ave[f(\bbx_n)]=0$
and $\|f(\bbx_n)\|_{L_\infty}\mik 1$. (Recall that $\bbx_n$ denotes the subarray of $\bbx$
determined by $[n]$.) Let $I$ be an arbitrary subset of $[n]$ with $|I|\meg 8$. Since $|I|\meg 8$,
there exists a $2$-dimensional box $B$ of $\nn$ with $B \subseteq \binom{I}{2}$ and
such that $\min(s)\meg 5$ for every $s\in B$. Set $C\coloneqq \bigcap_{s\in B}[X_s=1]$
and observe that $C\in \mathcal{F}_I$. Hence, by the exchangeability of $\bbx$, we have
\begin{align} \label{ea.4}
\ave\big[ \ave[f(\bbx_n)\,|\, \mathcal{F}_I] \, \mathbf{1}_C\big] & =
\ave[f(\bbx_n)\mathbf{1}_C] \\
& = \ave\Big[ \prod_{s\in \mathrm{Box}(2)\cup B} X_s \Big] -
\ave\Big[ \prod_{s\in \mathrm{Box}(2)} X_s\Big]^2 \stackrel{\eqref{ea.2}}{=} \frac{1}{2^{10}}, \nonumber
\end{align}
which implies that $\prob\big( \big|\ave[f(\bbx_n)\,|\, \mathcal{F}_I]\big|\meg 2^{-11}\big) \meg 2^{-11}$.
The proof is completed.
\end{proof}

\subsection{The higher-dimensional case} \label{subsecA.3}

The following result is the higher-dimensional analogue of Proposition \ref{pa.1}.
\begin{prop} \label{pa.2}
Let $d\meg 3$ be an integer. Also let $\delta>0$. Then there exists a boolean, exchangeable, $d$-dimensional
random array $\bbx=\langle X_s : s\in \binom{\nn}{d}\rangle$ on $\nn$ with the following properties.
\begin{enumerate}
\item[($\mathcal{P}$1)] \label{aP21} For every $s\in \binom{\nn}{d}$ we have
$\big|\ave[X_s]-\frac{1}{2}\big|\mik\delta$.
\item[($\mathcal{P}$2)] \label{aP22} For every distinct $s, t\in \binom{\nn}{d}$
we have $\big|\ave[X_s X_t]-\frac{1}{4}\big|\mik\delta$.
\item[($\mathcal{P}$3)] \label{aP23} For every $(d-1)$-face $F$ of $\nn$ we have
$\big|\ave\big[\prod_{s\in F} X_s\big] - (\frac{1}{2})^{|F|}\big|\mik\delta$.
\item[($\mathcal{P}$4)] \label{aP24} For every $d$-dimensional box $B$ of $\nn$ we have
$\big|\ave\big[\prod_{s\in B} X_s\big] - \frac{3}{2}(\frac{1}{2})^{|B|}\big|\mik\delta$.
\item[($\mathcal{P}$5)] \label{aP25} Set $\vartheta\coloneqq 16^{-1} 2^{-2^{d+1}}$
$($note that $\vartheta$ does not dependent on $\delta$$)$. Let $n\meg 4d$ be an integer,
and let $\bbx_{n}$ denote the subarray of $\bbx$ determined by~$[n]$. Then there exists a
multilinear polynomial $f\colon \rr^{\binom{[n]}{d}}\to \rr$ of degree~$2^d$ with $\ave[f(\bbx_{n})]=0$ and
$\|f(\bbx_{n})\|_{L_\infty}\mik 1$, such that for every subset $I$ of $[n]$ with $|I|\meg 4d$ we have
$\prob\big( \big|\ave[f(\bbx_n)\,|\, \mathcal{F}_I]\big|\meg \vartheta \big) \meg \vartheta$.
\end{enumerate}
\end{prop}
\begin{rem} \label{ra.3}
We point out that property (\hyperref[aP23]{$\mathcal{P}$3}) is rather strong. Indeed, arguing
as in the proof of Theorem \ref{t3.2}, it is not hard to show that if
$\bbx=\langle X_s: s\in \binom{\nn}{d}\rangle$ is any boolean, spreadable,
$d$-dimensional random array on $\nn$ that satisfies properties (\hyperref[aP21]{$\mathcal{P}$1})
and (\hyperref[aP23]{$\mathcal{P}$3}) of Proposition \ref{pa.2}, then for every
$d$-dimensional box $B$ of $\nn$ and every nonempty subset $G$ of $B$ with $G\neq B$ we have
\begin{equation} \label{ea.5}
\bigg|\ave\Big[\prod_{s\in G} X_s\Big] - \Big(\frac{1}{2}\Big)^{|G|}\bigg| = o_{\delta\to 0;d}(1).
\end{equation}
Note that \eqref{ea.5} barely misses to imply that $\bbx$ satisfies the box independence condition.
\end{rem}
The examples provided by Proposition \ref{pa.2} can be roughly described as semi-random, in the sense
that they are part random and part deterministic. The following lemma provides us with the random component.
\begin{lem} \label{la.3}
Let $d\meg 3$ be an integer, and let $\ee>0$. Then there exist a nonempty finite set $V$ and a
symmetric\footnote{That is, for every $(v_1,\dots,v_{d-1})\in V^{d-1}$ and every permutation $\pi$ of $[d-1]$
we have that $(v_1,\dots,v_{d-1})\in A$ if and only if $(v_{\pi(1)},\dots,v_{\pi(d-1)})\in A$.}
subset $A$ of $V^{d-1}$ such that, denoting by $A^{\complement}$ the complement of~$A$, for every pair $F,G$
of disjoint $($possibly empty$)$ subsets of $\binom{[2d]}{d-1}$ we have
\begin{align} \label{ea.6}
\bigg| \int \Big(\prod_{s\in F} \mathbf{1}_{A}(\boldsymbol{v}_s)\Big) \,
\Big( \prod_{s\in G} \mathbf{1}_{A^{\complement}}(\boldsymbol{v}_s)\Big)\, d\boldsymbol{\mu}(\boldsymbol{v}) -
\Big(\frac{1}{2}\Big)^{|F|+|G|}\bigg|\mik\ee,
\end{align}
where: \emph{(i)} $\boldsymbol{\mu}$ denotes the product measure on $V^\nn$ obtained by equipping each
factor with~the~uniform probability measure on $V$, \emph{(ii)} for every $\boldsymbol{v}=(v_i)\in V^\nn$
and every $s=\{i_1<\dots<i_{d-1}\}\in \binom{\nn}{d}$ we have $\boldsymbol{v}_s=(v_{i_1},\dots,v_{i_{d-1}})\in V^{d-1}$,
and \emph{(iii)} in \eqref{ea.6} we use the convention that the product of an empty family of functions is equal
to the constant function $1$.
\end{lem}
Lemma \ref{la.3} follows from a standard random selection and the Azuma--Hoeffding
inequality; see, \emph{e.g.}, \cite[Fact 3.3 and Lemma 3.4]{DTV21} for a proof.

We are ready to proceed to the proof of Proposition \ref{pa.2}.
\begin{proof}[Proof of Proposition \emph{\ref{pa.2}}]
Let $V$ and $A$ be the sets obtained by Lemma \ref{la.3} applied for
\begin{equation} \label{ea.7}
\ee\coloneqq  \min\Big\{ \delta\, 2^{-d2^d}, 8^{-1} 2^{-(d+2)2^d} \Big\},
\end{equation}
and observe that $V$ can be selected so that its cardinality is an even positive integer.
We also note that in the rest of the proof we follow the notational conventions in Lemma \ref{la.3}.

First, for every $i\in [d]$ we define $h^0_i,h^1_i\colon V^d\to\{0,1\}$ by setting
for every $\boldsymbol{v}\in V^d$,
\begin{align} \label{ea.8}
h^0_i(\boldsymbol{v})\coloneqq \mathbf{1}_A(\boldsymbol{v}_{[d]\setminus\{i\}}) \ \ \ \text{ and } \ \ \
h^1_i(\boldsymbol{v})\coloneqq \mathbf{1}_{A^\complement}(\boldsymbol{v}_{[d]\setminus\{i\}}).
\end{align}
Next, for every $\boldsymbol{x}\in\{0,1\}^d$ define $h_{\boldsymbol{x}}\colon V^d\to \{0,1\}$ by
\begin{align} \label{ea.9}
h_{\boldsymbol{x}} \coloneqq \prod_{i=1}^d h_i^{\boldsymbol{x}(i)}.
\end{align}
Finally, set
\begin{equation} \label{ea.10}
\mathbb{A}\coloneqq \big\{\boldsymbol{x}\in\{0,1\}^d : \boldsymbol{x}(1)+\cdots+\boldsymbol{x}(d) \text{ is even}\big\},
\end{equation}
and define $H\colon V^d\to\{0,1\}$ by
\begin{align} \label{ea.11}
H\coloneqq \sum_{\boldsymbol{x}\in\mathbb{A}} h_{\boldsymbol{x}}.
\end{align}
For instance, if $d=3$, then
\begin{align}
H(v_1,v_2,v_3) & = \mathbf{1}_{A}(v_1,v_2) \mathbf{1}_{A}(v_2,v_3) \mathbf{1}_{A}(v_1,v_3) +
\mathbf{1}_{A^{\complement}}(v_1,v_2) \mathbf{1}_{A^{\complement}}(v_2,v_3) \mathbf{1}_{A}(v_1,v_3) + \nonumber \\
& \ \ \ + \mathbf{1}_{A^{\complement}}(v_1,v_2) \mathbf{1}_{A}(v_2,v_3) \mathbf{1}_{A^{\complement}}(v_1,v_3) +
\mathbf{1}_{A}(v_1,v_2) \mathbf{1}_{A^{\complement}}(v_2,v_3) \mathbf{1}_{A^{\complement}}(v_1,v_3). \nonumber
\end{align}
Note that the function $H$ is symmetric\footnote{That is, we have $H(v_1,\dots,v_d)=H(v_{\pi(1)},\dots,v_{\pi(d)})$
for every $(v_1,\dots,v_d)\in V^d$ and every permutation $\pi$ of $[d]$.}.
In the following series of claims we isolate several properties of $H$ that will be used
in the proofs of properties (\hyperref[aP21]{$\mathcal{P}$1})--(\hyperref[aP25]{$\mathcal{P}$5}).
\begin{claim} \label{ca.5}
For every $k\in [d+1]$ set $t_k\coloneqq\{k,\dots,k+d-1\}\in \binom{\nn}{d}$. Then we have
\begin{equation} \label{ea.12}
 \bigg| \int H(\boldsymbol{v}_{t_1})\, d\boldsymbol{\mu}(\boldsymbol{v}) - \frac{1}{2} \bigg| \mik 2^{d-1} \ee.
\end{equation}
Moreover, for every $k\in \{2,\dots, d+1\}$ we have
\begin{equation} \label{ea.13}
\bigg| \int H(\boldsymbol{v}_{t_1}) H(\boldsymbol{v}_{t_k})\, d\boldsymbol{\mu}(\boldsymbol{v})
- \frac{1}{4} \bigg| \mik 2^{2d-3}\ee.
\end{equation}
\end{claim}
\begin{proof}[Proof of Claim \emph{\ref{ca.5}}]
First observe that \eqref{ea.12} follows from \eqref{ea.6}, the fact that
$|\mathbb{A}|= 2^{d-1}$ and the definition of $H$. Next, fix $k\in \{2,\dots,d+1\}$.
Then, for every $\boldsymbol{v}\in V^\nn$ we have
\begin{equation} \label{ea.14}
H(\boldsymbol{v}_{t_1}) H(\boldsymbol{v}_{t_k})
=\sum_{\boldsymbol{x},\boldsymbol{y}\in\mathbb{A}} \bigg(
\Big(\prod_{i=1}^d h_i^{\boldsymbol{x}(i)}(\boldsymbol{v}_{t_1})\Big) \,
\Big(\prod_{j=1}^dh_j^{\boldsymbol{y}(j)}(\boldsymbol{v}_{t_k})\Big) \bigg).
\end{equation}
Therefore, if $k>2$, then \eqref{ea.13} also follows from \eqref{ea.6} and the fact that
$|\mathbb{A}|= 2^{d-1}$. So assume that $k=2$. By \eqref{ea.14}, for every $\boldsymbol{v}\in V^\nn$ we have
\begin{equation} \label{ea.15}
H(\boldsymbol{v}_{t_1}) H(\boldsymbol{v}_{t_2}) = \!\!
\sum_{\boldsymbol{x},\boldsymbol{y}\in\mathbb{A}} \! \bigg(
\Big(\prod_{i=2}^dh_i^{\boldsymbol{x}(i)}(\boldsymbol{v}_{t_1})\Big) \,
\Big(\prod_{j=1}^{d-1}h_j^{\boldsymbol{y}(j)}(\boldsymbol{v}_{t_2})\Big) \,
h_1^{\boldsymbol{x}(1)}(\boldsymbol{v}_{t_1}) \,
h_d^{\boldsymbol{y}(d)}(\boldsymbol{v}_{t_2}) \bigg).
\end{equation}
Notice that for every $\boldsymbol{v}\in V^\nn$ we have $h^0_{\boldsymbol{v}_{t_1}} = h^0_{\boldsymbol{v}_{t_2}}$
and $h^1_{\boldsymbol{v}_{t_1}} = h^1_{\boldsymbol{v}_{t_2}}$. Thus, setting
$\mathcal{W}\coloneqq\{(\boldsymbol{x},\boldsymbol{y})\in\mathbb{A}\times\mathbb{A}:\boldsymbol{x}(1)=\boldsymbol{y}(d)\}$,
we see that $h_1^{\boldsymbol{x}(1)}(\boldsymbol{v}_{t_1})h_d^{\boldsymbol{y}(d)}(\boldsymbol{v}_{t_2})=0$
for every $(\boldsymbol{x},\boldsymbol{y})\in\mathbb{A}\times\mathbb{A}\setminus\mathcal{W}$.
Combining this information with \eqref{ea.15}, we obtain that
\begin{equation} \label{ea.16}
H(\boldsymbol{v}_{t_1}) H(\boldsymbol{v}_{t_2}) = \!\!\!
\sum_{(\boldsymbol{x},\boldsymbol{y})\in\mathcal{W}} \!\bigg(
\Big(\prod_{i=2}^dh_i^{\boldsymbol{x}(i)}(\boldsymbol{v}_{t_1})\Big) \,
\Big(\prod_{j=1}^{d-1}h_j^{\boldsymbol{y}(j)}(\boldsymbol{v}_{t_2})\Big) \,
h_1^{\boldsymbol{x}(1)}(\boldsymbol{v}_{t_1}) \,
h_d^{\boldsymbol{y}(d)}(\boldsymbol{v}_{t_2}) \bigg)
\end{equation}
for every $\boldsymbol{v}\in V^\nn$. On the other hand, by \eqref{ea.6},
for every $(\boldsymbol{x},\boldsymbol{y})\in\mathcal{W}$ we have
\begin{equation} \label{ea.17}
\bigg|\int \Big(\prod_{i=2}^d h_i^{\boldsymbol{x}(i)}(\boldsymbol{v}_{t_1})\Big)
\Big(\prod_{j=1}^{d-1}h_j^{\boldsymbol{y}(j)}(\boldsymbol{v}_{t_2})\Big)
h_1^{\boldsymbol{x}(1)}(\boldsymbol{v}_{t_1})h_d^{\boldsymbol{y}(d)}(\boldsymbol{v}_{t_2})\,
d\boldsymbol{\mu}(\boldsymbol{v}) - \Big(\frac{1}{2}\Big)^{2d-1}\bigg|\mik\ee.
\end{equation}
Since $|\mathcal{W}|=2^{2d-3}$, we conclude that \eqref{ea.13} for $k=2$ follows from
\eqref{ea.16} and \eqref{ea.17}. The proof of Claim \ref{ca.5} is completed.
\end{proof}
\begin{claim} \label{ca.6}
Set $C\coloneqq \big\{u\cup\{2d-1\}:u\in \mathrm{Box}(d-1)\big\}$, where $\mathrm{Box}(d-1)$
is as in \eqref{e3.2}, and notice that $C\subseteq \binom{\nn}{d}$. Then we have
\begin{equation} \label{ea.18}
\bigg| \int \prod_{s\in C} H(\boldsymbol{v}_s)\, d\boldsymbol{\mu}(\boldsymbol{v}) -
\Big(\frac{1}{2}\Big)^{|C|}\bigg|\mik (d+1)2^{d-2 +(d-1)2^{d-2}}\ee.
\end{equation}
\end{claim}
\begin{proof}[Proof of Claim \emph{\ref{ca.6}}]
We start by setting $j_i^0\coloneqq 2i-1$ and $j_i^1\coloneqq 2i$ for every $i\in [d-1]$.
Next, for every $\boldsymbol{\epsilon}=(\epsilon_i)_{i=1}^{d-1}\in \{0,1\}^{d-1}$ set
$s(\boldsymbol{\epsilon})\coloneqq \big\{ j_i^{\epsilon_i}: i\in [d-1]\big\}\cup\{2d-1\}$,
and notice that $C = \big\{s(\boldsymbol{\epsilon}): \boldsymbol{\epsilon}\in\{0,1\}^{d-1}\big\}$.
Moreover, by \eqref{ea.9} and \eqref{ea.11}, we have
\begin{align} \label{ea.19}
\int \prod_{s\in C} H(\boldsymbol{v}_s)\, d\boldsymbol{\mu}(\boldsymbol{v})& =
\int \prod_{\boldsymbol{\epsilon}\in\{0,1\}^{d-1}} H(\boldsymbol{v}_{s(\boldsymbol{\epsilon})})\,
d\boldsymbol{\mu}(\boldsymbol{v}) = \\
& = \sum_{(\boldsymbol{x}_{\boldsymbol{\epsilon}})_{\boldsymbol{\epsilon}\in\{0,1\}^{d-1}}\in \mathbb{A}^{\{0,1\}^{d-1}}}
\int \prod_{\boldsymbol{\epsilon}\in\{0,1\}^{d-1}}
h_{\boldsymbol{x}_{\boldsymbol{\epsilon}}}(\boldsymbol{v}_{s(\boldsymbol{\epsilon})})\,
d\boldsymbol{\mu}(\boldsymbol{v}). \nonumber
\end{align}
We define a subset $\mathcal{R}$ of $\mathbb{A}^{\{0,1\}^{d-1}}$ by the rule
\begin{figure}[htb] \label{figure5}
\centering \includegraphics[width=.5\textwidth]{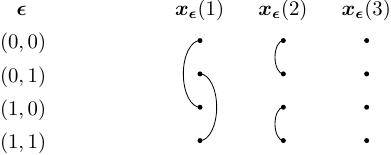}
\caption{The structure of the set $\mathcal{R}$ for $d=3$. Connected dots imply equality
of the corresponding coordinates.}
\end{figure}
\begin{align}\label{ea.20}
(\boldsymbol{x}_{\boldsymbol{\epsilon}})_{\boldsymbol{\epsilon}\in \{0,1\}^{d-1}}\in \mathcal{R} \Leftrightarrow
& \text{ for every } j\in[d-1], \text{ every } \boldsymbol{\epsilon}=(\epsilon_i)_{i=1}^{d-1}\in \{0,1\}^{d-1} \\
& \text{ and every } \boldsymbol{\epsilon}'=(\epsilon'_i)_{i=1}^{d-1}\in \{0,1\}^{d-1} \nonumber \\
& \text{ with } \epsilon_i = \epsilon'_i \nonumber \text{ for all } i\in [d-1]\setminus\{j\} \nonumber \\
& \text{ we have } \boldsymbol{x}_{\boldsymbol{\epsilon}}(j) = \boldsymbol{x}_{\boldsymbol{\epsilon}'}(j) \nonumber
\end{align}
\noindent (see Figure \hyperref[figure5]{5}). Observe that if $j\in [d-1]$ and
$\boldsymbol{\epsilon}=(\epsilon_i)_{i=1}^{d-1},
\boldsymbol{\epsilon}'=(\epsilon'_i)_{i=1}^{d-1}\in \{0,1\}^{d-1}$ with
$\epsilon_i = \epsilon'_i$ for all $i\in[d-1]\setminus\{j\}$, then
$h_j^0(\boldsymbol{v}_{s(\boldsymbol{\epsilon})}) = h_j^0(\boldsymbol{v}_{s(\boldsymbol{\epsilon}')})$
and $h_j^1(\boldsymbol{v}_{s(\boldsymbol{\epsilon})})  = h_j^1(\boldsymbol{v}_{s(\boldsymbol{\epsilon}')})$
for every $\boldsymbol{v}\in V^\nn$ which, in turn, implies that $h_j^0(\boldsymbol{v}_{s(\boldsymbol{\epsilon})})
h_j^1(\boldsymbol{v}_{s(\boldsymbol{\epsilon}')}) = 0$. 
Consequently, $\prod_{\boldsymbol{\epsilon}\in\{0,1\}^{d-1}}
h_{\boldsymbol{x}_{\boldsymbol{\epsilon}}}(\boldsymbol{v}_{s(\boldsymbol{\epsilon})})=0$
for every $(\boldsymbol{x}_{\boldsymbol{\epsilon}})_{\boldsymbol{\epsilon}\in\{0,1\}^{d-1}}
\in \mathbb{A}^{\{0,1\}^{d-1}}\setminus\mathcal{R}$ and every $\boldsymbol{v}\in V^{\nn}$.
Therefore, by \eqref{ea.19}, we obtain that
\begin{equation} \label{ea.21}
\int \prod_{s\in C} H(\boldsymbol{v}_s)\, d\boldsymbol{\mu}(\boldsymbol{v}) = \!\!\!\!
\sum_{(\boldsymbol{x}_{\boldsymbol{\epsilon}})_{\boldsymbol{\epsilon}\in\{0,1\}^{d-1}} \in \mathcal{R}}
\!\! \int \prod_{\boldsymbol{\epsilon}\in\{0,1\}^{d-1}}
h_{\boldsymbol{x}_{\boldsymbol{\epsilon}}}(\boldsymbol{v}_{s(\boldsymbol{\epsilon})})\, d\boldsymbol{\mu}(\boldsymbol{v}).
\end{equation}
On the other hand, by \eqref{ea.6}, for every
$(\boldsymbol{x}_{\boldsymbol{\epsilon}})_{\boldsymbol{\epsilon}\in\{0,1\}^{d-1}}\in\mathcal{R}$ we have
\begin{align} \label{ea.22}
& \bigg| \int \prod_{\boldsymbol{\epsilon}\in\{0,1\}^{d-1}}
h_{\boldsymbol{x}_{\boldsymbol{\epsilon}}}(\boldsymbol{v}_{s(\boldsymbol{\epsilon})})\,
d\boldsymbol{\mu}(\boldsymbol{v}) - \Big(\frac{1}{2}\Big)^{(d+1)2^{d-2}}\bigg| \\
=\, \bigg| \int & \prod_{i=1}^{d} \prod_{\boldsymbol{\epsilon}\in\{0,1\}^{d-1}}
h_i^{\boldsymbol{x}_{\boldsymbol{\epsilon}}(i)}(\boldsymbol{v}_{s(\boldsymbol{\epsilon})})\,
d\boldsymbol{\mu}(\boldsymbol{v}) - \Big(\frac{1}{2}\Big)^{(d+1)2^{d-2}}\bigg| \nonumber \\
=\, \bigg| \int & \Big(\prod_{i=1}^{d-1}\prod_{\boldsymbol{\epsilon}\in\{0,1\}^{d-1}}\!\!
h_i^{\boldsymbol{x}_{\boldsymbol{\epsilon}}(i)}(\boldsymbol{v}_{s(\boldsymbol{\epsilon})})\Big)
\Big( \prod_{\boldsymbol{\epsilon}\in\{0,1\}^{d-1}}\!\!
h_d^{\boldsymbol{x}_{\boldsymbol{\epsilon}}(d)}(\boldsymbol{v}_{s(\boldsymbol{\epsilon})})\Big)\,
d\boldsymbol{\mu}(\boldsymbol{v}) - \Big(\frac{1}{2}\Big)^{(d+1)2^{d-2}}\bigg| \nonumber\\
\, \mik  (d+& 1) 2^{d-2} \ee. \nonumber
\end{align}
The estimate \eqref{ea.18} follows from \eqref{ea.21}, \eqref{ea.22}, and the fact that
$|\mathcal{R}|=2^{(d-1)2^{d-2}}$ and $|C|=2^{d-1}$. The proof of Claim \ref{ca.6} is completed.
\end{proof}
\begin{claim} \label{ca.7}
Let\, $\mathrm{Box}(d)$ be as in \eqref{e3.2}. Then we have
\begin{equation} \label{ea.23}
\bigg| \int \prod_{s\in \mathrm{Box}(d)} \!\! H(\boldsymbol{v}_s)d\boldsymbol{\mu}(\boldsymbol{v}) -
2\Big(\frac{1}{2}\Big)^{|\mathrm{Box}(d)|}\bigg| \mik d\,2^{d+(d-2)2^{d-1}}\ee.
\end{equation}
\end{claim}
\begin{proof}[Proof of Claim \emph{\ref{ca.7}}]
As in the proof of Claim \ref{ca.6}, for every $i\in [d]$ set $j_i^0\coloneqq 2i-1$ and $j_i^1\coloneqq 2i$.
Moreover, for every $\boldsymbol{\epsilon}=(\epsilon_i)_{i=1}^d\in \{0,1\}^d$ set
$s(\boldsymbol{\epsilon})\coloneqq \big\{j_i^{\epsilon_i}: i\in [d]\big\}$, and observe that
$\mathrm{Box}(d)=\big\{s(\boldsymbol{\epsilon}): \boldsymbol{\epsilon}\in\{0,1\}^d\big\}$.
We define a subset $\mathcal{Q}$ of $\mathbb{A}^{\{0,1\}^d}$ by setting
\begin{figure}[htb] \label{figure6}
\centering \includegraphics[width=.5\textwidth]{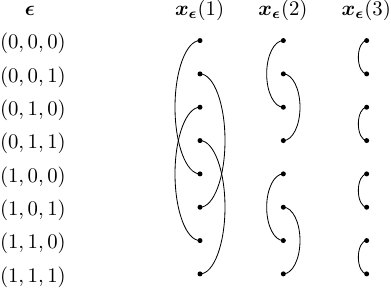}
\caption{The structure of the set $\mathcal{Q}$ for $d=3$. As in Figure \hyperref[figure5]{$5$},
connected dots imply equality of the corresponding coordinates.}
\end{figure}
\begin{align} \label{ea.24}
(\boldsymbol{x}_{\boldsymbol{\epsilon}})_{\boldsymbol{\epsilon}\in\{0,1\}^d}\in \mathcal{Q} \Leftrightarrow &
\text{ for every } j\in[d], \text{ every } \boldsymbol{\epsilon}=(\epsilon_i)_{i=1}^d\in \{0,1\}^d \\
& \text{ and every } \boldsymbol{\epsilon}'=(\epsilon'_i)_{i=1}^d\in \{0,1\}^d \nonumber \\
& \text{ with } \epsilon_i = \epsilon'_i \nonumber \text{ for all } i\in[d]\setminus\{j\} \nonumber \\
& \text{ we have } \boldsymbol{x}_{\boldsymbol{\epsilon}}(j) = \boldsymbol{x}_{\boldsymbol{\epsilon}'}(j) \nonumber
\end{align}
\noindent (see Figure \hyperref[figure6]{6}). By \eqref{ea.9}, \eqref{ea.11},
the definition of $\mathcal{Q}$ and arguing as in Claim~\ref{ca.6},
\begin{align} \label{ea.25}
\int \prod_{s\in \mathrm{Box}(d)} H(\boldsymbol{v}_s)\, d\boldsymbol{\mu}(\boldsymbol{v}) &  =
\int \prod_{\boldsymbol{\epsilon}\in\{0,1\}^d} H(\boldsymbol{v}_{s(\boldsymbol{\epsilon})})\,
d\boldsymbol{\mu}(\boldsymbol{v}) \\
& = \sum_{(\boldsymbol{x}_{\boldsymbol{\epsilon}})_{\boldsymbol{\epsilon}\in\{0,1\}^d}\in \mathbb{A}^{\{0,1\}^d}}
\int \prod_{\boldsymbol{\epsilon}\in\{0,1\}^d} h_{\boldsymbol{x}_{\boldsymbol{\epsilon}}}(\boldsymbol{v}_{s(\boldsymbol{\epsilon})})\,
d\boldsymbol{\mu}(\boldsymbol{v}) \nonumber \\
& = \sum_{(\boldsymbol{x}_{\boldsymbol{\epsilon}})_{\boldsymbol{\epsilon}\in\{0,1\}^d} \in \mathcal{Q}}
\int \prod_{\boldsymbol{\epsilon}\in\{0,1\}^d}
h_{\boldsymbol{x}_{\boldsymbol{\epsilon}}}(\boldsymbol{v}_{s(\boldsymbol{\epsilon})})\,
d\boldsymbol{\mu}(\boldsymbol{v}). \nonumber
\end{align}
By \eqref{ea.6}, for every
$(\boldsymbol{x}_{\boldsymbol{\epsilon}})_{\boldsymbol{\epsilon}\in\{0,1\}^{d-1}}\in\mathcal{Q}$ we have
\begin{equation} \label{ea.26}
\bigg| \int \prod_{\boldsymbol{\epsilon}\in\{0,1\}^d}
h_{\boldsymbol{x}_{\boldsymbol{\epsilon}}}(\boldsymbol{v}_{s(\boldsymbol{\epsilon})})\,
d\boldsymbol{\mu}(\boldsymbol{v}) - \Big(\frac{1}{2}\Big)^{d2^{d-1}}\bigg| \mik d\,2^{d-1}\ee.
\end{equation}
Finally, note that $|Q| = 2^{(d-2)2^{d-1}+1}$. Using this information, \eqref{ea.23} follows from
\eqref{ea.25}, \eqref{ea.26} and the fact that $|\mathrm{Box}(d)|=2^d$. The proof of Claim \ref{ca.7} is completed.
\end{proof}
\begin{claim} \label{ca.8}
Let $B$ be a $d$-dimensional box of\, $\nn$ with $\min(s)\meg 2d+1$ for every $s\in B$. Then we have
\begin{equation} \label{ea.27}
\bigg| \int \prod_{s\in \mathrm{Box}(d) \cup B} H(\boldsymbol{v}_s)\, d\boldsymbol{\mu}(\boldsymbol{v}) -
4 \Big(\frac{1}{2}\Big)^{2|\mathrm{Box}(d)|} \bigg| \mik d\,2^{d+1+(d-2)2^{d-1}}\ee.
\end{equation}
\end{claim}
\begin{proof}[Proof of Claim \emph{\ref{ca.8}}]
It follows immediately by Claim \ref{ca.7}.
\end{proof}
After this preliminary discussion, we now enter into the main part of the proof.
Let $\bbx=\langle X_s:s\in \binom{\nn}{d}\rangle$ be a boolean, exchangeable,
$d$-dimensional random array on $\nn$ whose distribution satisfies
\begin{equation} \label{ea.28}
\ave\Big[ \prod_{s\in \mathcal{F}} X_s\Big] = \frac{1}{2} \Big(\frac{1}{2}\Big)^{|\mathcal{F}|} +
\frac{1}{2} \int\prod_{s\in \mathcal{F}} H(\boldsymbol{v}_{s})\, d\boldsymbol{\mu}(\boldsymbol{v})
\end{equation}
for every nonempty finite subset $\mathcal{F}$ of $\binom{\nn}{d}$; the existence of such a random
array follows arguing precisely as in the proof of Proposition \ref{pa.1}.

First, we will show that $\bbx$ satisfies properties (\hyperref[aP21]{$\mathcal{P}$1})
up to (\hyperref[aP24]{$\mathcal{P}$4}). For property (\hyperref[aP21]{$\mathcal{P}$1}),
let $s\in \binom{\nn}{d}$  be arbitrary and notice that, by the exchangeability of $\bbx$ and \eqref{ea.28},
\begin{equation} \label{ea.29}
\ave[X_s]  = \frac{1}{4} + \frac{1}{2}\int H(\boldsymbol{v}_{t_1})\, d\boldsymbol{\mu}(\boldsymbol{v}),
\end{equation}
where, as in Claim \ref{ca.5}, we have $t_1=\{1,\dots,d\}$. By \eqref{ea.12} and the choice of $\ee$
in \eqref{ea.7}, we obtain that $\big|\ave[X_s]-\frac{1}{2}\big|\mik 2^{d-2}\ee\mik\delta$.
For property (\hyperref[aP21]{$\mathcal{P}$2}), let $s,t\in \binom{\nn}{d}$ be distinct,
and set $k\coloneqq d-|s\cap t|+1$. Since $\bbx$ is exchangeable, by \eqref{ea.28}, we have
\begin{align} \label{ea.30}
\ave[X_s X_t] = \frac{1}{8} + \frac{1}{2} \int H(\boldsymbol{v}_{t_1})H(\boldsymbol{v}_{t_k})\,
d\boldsymbol{\mu}(\boldsymbol{v}),
\end{align}
where $t_1$ and $t_k$ are as in Claim \ref{ca.5}. By \eqref{ea.13}, \eqref{ea.30}
and invoking again \eqref{ea.7}, we see that
$\big|\ave[X_s X_t]-\frac{1}{4}\big|\mik 2^{2d-4}\ee\mik\delta$.
For property (\hyperref[aP23]{$\mathcal{P}$3}), let $F$ be a $(d-1)$-face of $\nn$.
By the exchangeability of $\bbx$, \eqref{ea.28} and the choice of the set $C$ in Claim \ref{ca.6},
\begin{align} \label{ea.31}
\ave\Big[\prod_{s\in F} X_s\Big] = \frac{1}{2}\Big(\frac{1}{2}\Big)^{|F|} + \frac{1}{2}
\int \prod_{s\in C} H(\boldsymbol{v}_s)\, d\boldsymbol{\mu}(\boldsymbol{v})
\end{align}
which implies, by \eqref{ea.18}, that
\begin{equation} \label{ea.32}
\bigg|\ave\Big[ \prod_{s\in F} X_s\Big]- \Big(\frac{1}{2}\Big)^{|F|}\bigg| \mik (d+1)2^{d-3 +(d-1)2^{d-2}}\ee
\stackrel{\eqref{ea.7}}{\mik} \delta.
\end{equation}
Lastly, for property (\hyperref[aP24]{$\mathcal{P}$4}), let $B$ be a $d$-dimensional box of $\nn$.
Using once again the exchangeability of $\bbx$ and \eqref{ea.28}, we see that
\begin{align} \label{ea.33}
\ave\Big[ \prod_{s\in B} X_s \Big] = \frac{1}{2}\Big(\frac{1}{2}\Big)^{|B|} +
\frac{1}{2} \int \prod_{s\in \mathrm{Box}(d)} H(\boldsymbol{v}_s)\, d\boldsymbol{\mu}(\boldsymbol{v}),
\end{align}
and so, by \eqref{ea.23},
\begin{equation} \label{ea.34}
\bigg|\ave\Big[\prod_{s\in B} X_s\Big]- \frac{3}{2}\Big(\frac{1}{2}\Big)^{|B|}\bigg|
\mik d\,2^{d-1+(d-2)2^{d-1}}\ee \stackrel{\eqref{ea.7}}{\mik} \delta.
\end{equation}

Thus, it remains to verify property (\hyperref[aP25]{$\mathcal{P}$5}).
As expected, we will argue as in Proposition~\ref{p2.8}. Specifically,
fix an integer $n\meg 4d$, and define $f\colon \rr^{\binom{[n]}{d}}\to \rr$ by setting
for every $\boldsymbol{x}=(x_t)_{t\in \binom{[n]}{d}}\in \rr^{\binom{[n]}{d}}$
\begin{equation} \label{ea.35}
f(\boldsymbol{x})\coloneqq \prod_{s\in \mathrm{Box}(d)}\!\! x_s- \ave\Big[ \prod_{s\in \mathrm{Box}(d)}\!\! X_s \Big].
\end{equation}
Clearly, $f$ is a multilinear polynomial of degree $2^d$, and it satisfies $\ave[f(\bbx_n)]=0$ and
$\|f(\bbx_n)\|_{L_\infty}\mik 1$. On the other hand, if $B$ is a $d$-dimensional box of $\nn$ with
$\min(s)\meg 2d+1$ for every $s\in B$, then
\begin{align} \label{ea.36}
\ave\Big[\prod_{s\in \mathrm{Box}(d) \cup B} X_s\Big] \stackrel{\eqref{ea.28}}{=}
\frac{1}{2} \Big(\frac{1}{2}\Big)^{2^{d+1}} + \frac{1}{2} \int \prod_{s\in \mathrm{Box}(d)\cup B}
H(\boldsymbol{v}_s)\,  d\boldsymbol{\mu}(\boldsymbol{v}),
\end{align}
and so, by \eqref{ea.27},
\begin{equation} \label{ea.37}
\bigg|\ave\Big[ \prod_{s\in \mathrm{Box}(d) \cup B} X_s\Big] - \frac{5}{2} \Big(\frac{1}{2}\Big)^{2^{d+1}}\bigg|
\mik d\,2^{d+(d-2)2^{d-1}}\ee.
\end{equation}
Using this estimate, property (\hyperref[aP24]{$\mathcal{P}$4}) and arguing as in the proof of Proposition \ref{pa.1},
it is easy to verify that the function $f$ satisfies property (\hyperref[aP25]{$\mathcal{P}$5}).
The proof of Proposition \ref{pa.2} is completed.
\end{proof}


\end{document}